\documentclass{amsart}

\usepackage{amsmath,amssymb,amsfonts,mathtools}
\usepackage{mathrsfs}
\usepackage{bm}
\usepackage{graphicx}
\usepackage{booktabs}
\usepackage{enumitem}
\usepackage{algorithm}
\usepackage{algpseudocode}
\usepackage{microtype}
\usepackage{wrapfig}
\usepackage{subcaption}
\usepackage{xcolor}
\usepackage{tikz}
\usetikzlibrary{arrows.meta,decorations.pathreplacing}
\definecolor{hilbertheader}{RGB}{216,231,246}
\definecolor{hilbertactive}{RGB}{232,244,253}
\definecolor{hilbertactiveline}{RGB}{158,198,226}
\definecolor{hilbertinactive}{RGB}{253,239,227}
\definecolor{hilbertinactiveline}{RGB}{235,181,141}
\definecolor{hilbertdash}{RGB}{0,113,188}
\usepackage{hyperref}

\numberwithin{equation}{section}

\theoremstyle{plain}
\newtheorem{theorem}{Theorem}[section]
\newtheorem{lemma}{Lemma}[section]

\theoremstyle{definition}
\newtheorem{definition}{Definition}[section]

\theoremstyle{plain}
\newtheorem{assumption}{Assumption}[section]
\newtheorem{remark}{Remark}[section]

\newcommand{\M}{\mathcal{M}}
\newcommand{\Lop}{\mathcal{L}}

\newcommand{\pRet}{\operatorname{p-Ret}}

\newcommand{\grad}{\operatorname{grad}}
\newcommand{\Hess}{\operatorname{Hess}}
\newcommand{\trg}{\operatorname{tr}_{g}}

\newcommand{\1}{\mathbf{1}}
\newcommand{\R}{\mathbb{R}}
\newcommand{\E}{\mathbb{E}}

\newcommand{\dd}{\,\mathrm{d}}
\newcommand{\MC}{\mathrm{MC}}
\newcommand{\RQMC}{\mathrm{RQMC}}
\newcommand{\ARQMC}{\mathrm{Array\mbox{-}RQMC}}

\newcommand{\cl}{\mathrm{cl}}
\newcommand{\bd}{\partial}

\title[QMC Meshfree Methods on Parameterized Manifolds]{Quasi-Monte Carlo based meshfree methods for parabolic equations on parameterized Riemannian manifolds}

\author{Zhijian He}
\address{School of Mathematics, South China University of Technology, Guangzhou 510641, China}
\email{hezhijian@scut.edu.cn}

\author{Zhiyuan Hui}
\address{School of Mathematics, Shanghai University of Finance and Economics, Shanghai 200433, China}
\email{2023213122@stu.sufe.edu.cn}
\thanks{$\dagger$Corresponding author. Z. He is supported by the NSF of China (No. 12671500) and  the Guangdong Basic and Applied Basic Research Foundation (No. 2025A1515011888). Z. Hui is partially supported by the Fundamental Research Funds for the Central Universities (No. CXJJ-2024-437). C. Sheng is supported by the NSF of China (No. 12571389). B. Su is supported by the NSF of China (No. 12501541) and the Scientific Research Start-up Funds of Hainan University (No. XJ2400010491). T. Zhou is supported by the NSF of China (Nos. 12288201 and 12461160275).}

\author{Changtao Sheng}
\address{School of Mathematics, Shanghai University of Finance and Economics, Shanghai 200433, China}
\email{ctsheng@sufe.edu.cn}

\author{Bihao Su}
{\address{School of Mathematics and Statistics, Hainan University, Haikou 570100, China}
\email{bihaosu@hainanu.edu.cn}}

\author{Tao Zhou$^\dagger$}
\address{Institute of Computational Mathematics and Scientific/Engineering Computing, Academy of Mathematics and Systems Science, Chinese Academy of Sciences, Beijing 100190, China}
\email{tzhou@lsec.cc.ac.cn}

\keywords{parameterized Riemannian manifold, Feynman--Kac representation, retraction random walk, randomized quasi-Monte Carlo, Array-RQMC, Hilbert ordering}
\subjclass[2020]{65C05, 65M75, 58J65, 60H35, 58J35}

\begin{document}

\begin{abstract}
We develop and analyze a randomized path sampling framework for the pointwise approximation of linear parabolic equations on parameterized Riemannian manifolds, treating both closed manifolds and manifolds with Dirichlet boundary. Based on the Feynman--Kac representation, we construct a parameter-coordinate retraction-based discretization, referred to as the p-Ret scheme, for the manifold diffusion generated by the uniformly elliptic operator, and use it to formulate the Monte Carlo (MC) and randomized quasi-Monte Carlo (RQMC) methods. We derive the local weak consistency of the p-Ret transition and first-order weak convergence of the temporal discretization on closed manifolds. For fixed $h$, we show that RQMC sampling yields asymptotically faster variance decay than standard MC. On closed manifolds, the variance is $o(n^{-1})$ and, under additional regularity, satisfies the sharper bound $O(n^{-2+\varepsilon})$ for any $\varepsilon>0$. For Dirichlet problems, despite the discontinuities induced by numerical first exit, nested-uniformly scrambled Sobol' sampling attains $O(n^{-1-1/s+\varepsilon})$ variance for any $\varepsilon>0$ under suitable regularity assumptions, where $s$ denotes the nominal path-space dimension. More notably, we further develop an Array-RQMC strategy based on stepwise low-dimensional randomized point sets and a Hilbert space-filling curve that maps high-dimensional particle states to a one-dimensional ordering, and establish a variance rate of $o(n^{-1})$ under suitable assumptions. Numerical results demonstrate that both RQMC methods reduce variance and root mean square error (RMSE) relative to standard MC and remain effective for manifolds with intrinsic dimension up to $50$.
\end{abstract}

\maketitle

\section{Introduction}
\label{sec:introduction}

Diffusion and transport on curved geometries arise in many scientific and engineering applications and lead naturally to partial differential equations (PDEs) on surfaces and Riemannian manifolds. A broad range of deterministic methods has been developed for such problems, including surface, adaptive, and trace finite element methods, embedding and closest-point techniques, and meshfree
approaches (see, e.g., \cite{BertalmioEtAl2001,DednerEtAl2013,Demlow2009,DziukElliott2013,JiangLiYanHarlim2024,LiangZhao2013,MacdonaldRuuth2009,OlshanskiiEtAl2009,Reusken2015,RuuthMerriman2008} and the references therein). Accurate mesh generation on curved surfaces and high-dimensional manifolds is often challenging for traditional mesh-based deterministic methods, while global spatial discretizations lead to rapidly growing discrete systems as the intrinsic dimension increases, making it difficult to overcome the curse of dimensionality. This motivates a complementary meshfree probabilistic approach that directly approximates the solution pointwise through stochastic trajectories.

The Feynman--Kac representation provides a fundamental bridge between PDEs and stochastic processes by expressing their solutions as expectations over
diffusion paths (cf.~\cite{Fox2002,Gobet2000,GobetMaire2005,Kac1949,SuXuSheng2023}). In Euclidean domains, this viewpoint has led to a broad class of Monte Carlo-type methods
\cite{BuchmannPetersen2006,HuiLiuShengXu2026,MilsteinTretyakov2003,Muller1956,ShengSuXu2026}, which are particularly attractive for complex domains and high-dimensional
problems. Extending this framework to Riemannian manifolds is natural, but introduces additional challenges, since the simulated paths must respect the
manifold geometry while correctly approximating the intrinsic diffusion. On Riemannian manifolds, the Feynman--Kac representation remains valid once
the Euclidean diffusion is replaced by a manifold-valued diffusion \cite{Hsu2002}, whose numerical approximation has been approached through
geodesic random walks, intrinsic SDE discretizations, constrained Brownian dynamics, and projection-based schemes (see, e.g.,\cite{ArmstrongBrigo2018,ArmstrongKing2022,HuiShengSuZhou2026,
Jorgensen1975,LelievreRoussetStoltz2012,SugimotoEtAl2024}). More recently, Schwarz et al.~\cite{SchwarzEtAl2025} developed an efficient retraction-based random walk that replaces costly geodesic computations by second-order retractions while preserving convergence to the correct Brownian motion.

From a computational perspective, the Feynman--Kac representation reduces the pointwise evaluation of the PDE solution to the approximation of an expectation over diffusion paths, for which Monte Carlo sampling provides a natural and highly parallelizable approach, although its root-mean-square sampling error decays only at the rate $O(n^{-1/2})$ (cf. \cite{Caflisch1998,ShengSuXu2023}). This relatively slow convergence motivates randomized quasi-Monte Carlo (RQMC) methods, which replace independent random samples by randomized low-discrepancy point sets while preserving unbiasedness \cite{He2018,HeWang2015,Owen1997MonteCarloVariance,Owen1995,Owen2003AlternativeScrambling}. In particular, nested-uniformly scrambled Sobol' point sets achieve variance $o(n^{-1})$ for square-integrable integrands and faster rates under additional smoothness assumptions \cite{Dick2011,DickKuoSloan2013,Owen1997}. The effectiveness of RQMC, however, depends strongly on the regularity of the integrand and on the transformations used to generate the sampling variables \cite{BasuOwen2016}. QMC methods have also proved effective for high-dimensional PDE computations under suitable regularity assumptions (cf.~\cite{GrahamEtAl2011,KuoNuyens2016,KuoSchwabSloan2012}).

Applying RQMC to Feynman--Kac discretizations on parameterized Riemannian manifolds introduces two additional challenges. First, with $K$ time steps and $r-1$ random variables per step, the direct path-space formulation has nominal dimension $s=(r-1)K$, which can become large as $K$ or $r$ increases. This motivates sequential RQMC methods, in particular Array-RQMC (cf.~\cite{HoOwen2026Array,LEcuyerLecotTuffin2008}), which apply low-dimensional randomized point sets successively over the time steps rather than to the full path space. For multidimensional state spaces, the particles must additionally be ordered before coupling with the randomized points, for which Hilbert space-filling curves provide a natural mechanism \cite{GerberChopin2015,HeOwen2016}. A second difficulty concerns the regularity of the path-space integrand. While the closed-manifold functional remains sufficiently regular under suitable assumptions, numerical first exit in the Dirichlet case creates discontinuity interfaces in path space, requiring a separate variance analysis for scrambled Sobol' point sets (cf.~\cite{He2018,HeWang2015,HoOwen2026RQMC}).

These challenges expose a gap between the existing p-Ret analysis and the available RQMC theory. To the best of our knowledge, an RQMC error theory for Feynman--Kac functionals generated by p-Ret discretizations of manifold diffusions has not yet been developed. Existing p-Ret analysis has focused primarily on the geometric approximation of manifold diffusions, while RQMC and sequential RQMC theory has been developed mainly for high-dimensional integration, SDE discretizations, and Markov processes in Euclidean state spaces. In particular, no unified analysis is currently available for direct path-space RQMC, first-exit discontinuities in Dirichlet problems, and Hilbert-ordered Array-RQMC within the p-Ret framework.

The main purpose of this paper is to develop a randomized p-Ret framework for the pointwise approximation of linear parabolic equations on parameterized
Riemannian manifolds, with the aim of improving sampling efficiency beyond standard Monte Carlo. The key to the efficiency of this framework lies in the Feynman--Kac
representation and a parameter-coordinate p-Ret discretization of the associated manifold diffusion, leading to a common discrete Feynman--Kac functional for MC, direct RQMC, and Array-RQMC sampling. Moreover, the framework naturally accommodates both closed manifolds and manifolds with Dirichlet boundary. This construction brings four key
computational advantages for the numerical treatment of high-dimensional manifold problems:
\begin{itemize}[leftmargin=1.7em,itemsep=0.25em,topsep=0.3em]

\item \emph{Meshfree and high-dimensional computation.}
The p-Ret construction relies only on local geometric information and avoids
the costly mesh generation and large scale linear solves inherent in
conventional mesh based discretizations, thereby alleviating the curse of
dimensionality for high dimensional manifold problems.

\item \emph{Pointwise and parallel evaluation.}
The Feynman--Kac formulation enables pointwise evaluation at prescribed space--time locations. The computation is naturally parallel across such locations, since each can be treated separately.

\item \emph{Reduced path-space dimension.}
Each p-Ret step samples a direction $Z\in\mathbb S^{r-1}$ from only
$r-1$ independent uniform variables, whereas an $r$-dimensional
Euler--Maruyama discretization requires $r$ random inputs per step.
Consequently, over $K$ time steps, the nominal dimension of the direct
RQMC problem is reduced from $rK$ to $(r-1)K$. This removes one stochastic
coordinate at every time step and directly lowers the high-dimensional
integration cost by direct RQMC.

\item \emph{Unified discrete probabilistic structure.}
MC, direct RQMC, and Array-RQMC are built on the same p-Ret dynamics and
discrete Feynman--Kac functional, and hence share the same time-discrete
target. This isolates the effect of the sampling strategy from the temporal
discretization and enables a consistent comparison of their variance and
sampling errors.

\end{itemize}


Beyond these computational advantages, we also analyze the errors arising from the p-Ret discretization and randomized sampling. To the best of our knowledge, no comparable RQMC error theory is available for p-Ret based Feynman--Kac discretizations on Riemannian manifolds. On closed manifolds, we first establish local weak consistency of the p-Ret transition and first-order weak convergence of the temporal discretization. Building on this result, we show that direct RQMC sampling yields variance $o(n^{-1})$ for square-integrable discrete Feynman--Kac functionals and, under additional regularity, the sharper bound $O\bigl(n^{-2}(\log n)^{3s+\kappa_h-2}\bigr)$, where $s=(r-1)K$ and $\kappa_h$ denotes a polynomial-growth exponent of the associated Gaussian-coordinate integrand. The Dirichlet case is more delicate because the numerical first-exit mechanism creates discontinuity interfaces in path space. We resolve this difficulty by a layerwise decomposition of the stopped functional and verify Owen's boundary growth condition \cite{Owen2006Halton} for the resulting smooth layer weights, which leads to the variance bound $O\bigl(n^{-1-1/s+\varepsilon}\bigr)$ for any $\varepsilon>0$. Finally, for Hilbert-ordered Array-RQMC, we establish the variance rate $o(n^{-1})$ for fixed $K$ under suitable stability and regularity assumptions.

The remainder of the paper is organized as follows. Section~\ref{sec:pret-algorithm} introduces the p-Ret discretization, the associated discrete Feynman--Kac functionals, and the MC and direct RQMC sampling schemes. The temporal discretization and sampling errors for MC and direct RQMC are analyzed in Section~\ref{sec:error-estimates}. Section~\ref{sec:array-rqmc} is devoted to the construction and analysis of the Array-RQMC method. Numerical results are reported in Section~\ref{sec:numerical-experiments}, and Section~\ref{sec:conclusion} concludes the paper.

\section{The p-Ret discretization with MC and direct RQMC sampling}
\label{sec:pret-algorithm}
In this section, we introduce the p-Ret discretization for linear parabolic equations on parameterized Riemannian manifolds and define the corresponding randomized estimators. The local consistency of the discrete transition operator will be used in the subsequent error analysis.

\subsection{Geometric setting and Feynman--Kac representations}
\label{subsec:setting-fk}

Let $(\M,g)$ be a smooth compact connected $m$-dimensional Riemannian manifold, either closed or with smooth boundary, and let $\M^\circ=\M\setminus\partial\M$ in the latter case. We assume that $\M$ admits a global regular parameterization
\begin{equation*}
    \Phi:U\longrightarrow\M\subset\R^D,
    \quad
    x=\Phi(\xi),
    \quad
    \xi=(\xi^1,\ldots,\xi^m),
\end{equation*}
subject to the prescribed parameter identifications, with the Riemannian metric induced by $\Phi$. In the closed case, $\Phi$ is allowed to be a smooth immersion and need not be globally injective in the ambient space. In the boundary case, $\Phi$ is assumed to be one-to-one modulo the prescribed periodic identifications, extends smoothly with full rank to a neighborhood of $\overline U$, and satisfies $\Phi(U)=\M^\circ$ and $\Phi(\partial U)=\partial\M$, where $\partial U$ denotes the nonperiodic boundary of the parameter domain.

For each $x\in\M$, let $A(x):T_x\M\to T_x\M$ be a smooth $g$-self-adjoint diffusion tensor, and let $b$ be a smooth tangent vector field on $\M$. For $\varphi\in C^2(\M)$, we define
\begin{equation}
    \Lop\varphi=\frac12\trg\big(A\Hess_{\M}\varphi\big)+\langle b,\grad_{\M}\varphi\rangle_g.
    \label{eq:intrinsic-operator}
\end{equation}
Here, $\grad_{\M}\varphi$ and $\Hess_{\M}\varphi$ denote the Riemannian gradient and Hessian, respectively, $\langle\cdot,\cdot\rangle_g$ is the Riemannian metric, and $\trg$ denotes the corresponding metric trace. We assume that $A$ is uniformly elliptic, namely, there exist constants $0<\lambda\le\Lambda$ such that
\begin{equation}
    \lambda\lvert v\rvert_g^2\le\langle A(x)v,v\rangle_g\le\Lambda\lvert v\rvert_g^2,
    \quad
    x\in\M,\quad v\in T_x\M.
    \label{eq:uniform-ellipticity}
\end{equation}
The special case $A=I$ and $b=0$ reduces to $\Lop=\frac12\Delta_\M$, where $\Delta_\M$ be the Laplace–Beltrami operator. With the operator $\Lop$ defined above, we consider the following two parabolic problems.

\paragraph{\underline{Closed manifold}}
Given $u_0:\M\to\R$ and $f:\M\times[0,T]\to\R$, find $u$ such that
\begin{equation}
    \begin{cases}
    \partial_tu(x,t)-\Lop u(x,t)=f(x,t),
        &(x,t)\in\M\times(0,T],\\[4pt]
    u(x,0)=u_0(x),
        &x\in\M.
    \end{cases}
    \label{eq:closed-parabolic-problem}
\end{equation}

\paragraph{\underline{Manifold with Dirichlet boundary}}
Given $\psi:\partial\M\times[0,T]\to\R$, find $u$ such that
\begin{equation}
    \begin{cases}
    \partial_tu(x,t)-\Lop u(x,t)=f(x,t),
        &(x,t)\in\M^{\circ}\times(0,T],\\[4pt]
    u(x,t)=\psi(x,t),
        &(x,t)\in\partial\M\times(0,T],\\[4pt]
    u(x,0)=u_0(x),
        &x\in\M,
    \end{cases}
    \label{eq:dirichlet-parabolic-problem}
\end{equation}
with $u_0=\psi(\cdot,0)$ on $\partial\M$.

Let $X=(X_t)_{t\ge0}$ denote the manifold diffusion with infinitesimal generator $\Lop$, started from $X_0=x$. By the Feynman--Kac representation for diffusions on Riemannian manifolds~\cite{Hsu2002}, for a fixed evaluation time $t\in(0,T]$, the closed problem \eqref{eq:closed-parabolic-problem} admits the representation
\begin{equation}
    u(x,t)=\E_x\Big[u_0(X_t)+\int_0^t f(X_s,t-s)\dd s\Big].
    \label{eq:closed-feynman-kac}
\end{equation}
For the boundary problem, define the first hitting time $\tau_{\partial}=\inf\{s>0:X_s\in\partial\M\}$, then the problem \eqref{eq:dirichlet-parabolic-problem} has the representation
\begin{equation}
    \begin{aligned}
    u(x,t)=\E_x\Big[
        &u_0(X_t)\1_{\{\tau_{\partial}>t\}}+\psi(X_{\tau_{\partial}},t-\tau_{\partial})\1_{\{\tau_{\partial}\le t\}}+\int_0^{t\wedge\tau_{\partial}}f(X_s,t-s)\dd s
    \Big].
    \end{aligned}
    \label{eq:boundary-feynman-kac}
\end{equation}

\subsection{The p-Ret discretization in parameter coordinates}
\label{subsec:coordinate-pret}


Let
\begin{equation*}
    g_{ij}(\xi)=\big\langle \partial_i\Phi(\xi),\partial_j\Phi(\xi)\big\rangle,
    \quad
    \big(g^{ij}(\xi)\big)_{i,j=1}^m=\big(g_{ij}(\xi)\big)_{i,j=1}^m{}^{-1}.
\end{equation*}
The corresponding Christoffel symbols are
\begin{equation*}
    \Gamma_{ij}^{\ell}(\xi)=\frac12\sum_{k=1}^m g^{\ell k}(\xi)\Big(\partial_i g_{jk}(\xi)+\partial_j g_{ik}(\xi)-\partial_k g_{ij}(\xi)\Big).
\end{equation*}
Assume that the diffusion tensor admits a smooth factorization
\begin{equation*}
    A=\sum_{q=1}^{r}\sigma_q\otimes\sigma_q,
    \quad
    Av=\sum_{q=1}^{r}\langle\sigma_q,v\rangle_g\,\sigma_q,
\end{equation*}
where $r$ is the dimension of the driving noise and need not equal the manifold dimension $m$. Writing, for $x=\Phi(\xi)$,
\begin{equation*}
    \sigma_q(x)
    =
    \sum_{i=1}^m\sigma_q^i(\xi)\,\partial_i\Phi(\xi),
    \quad
    b(x)
    =
    \sum_{i=1}^m b^i(\xi)\,\partial_i\Phi(\xi),
\end{equation*}
define
\begin{equation*}
    \Sigma(\xi)
    =
    \big(\sigma_q^i(\xi)\big)_{1\le i\le m,\,1\le q\le r},
    \quad
    a(\xi)
    =
    \Sigma(\xi)\Sigma(\xi)^T
    =
    \big(a^{ij}(\xi)\big)_{i,j=1}^m.
\end{equation*}

For $\widetilde\varphi=\varphi\circ\Phi$, the operator \eqref{eq:intrinsic-operator} takes the coordinate form
\begin{equation}
    \big(\Lop\varphi\big)\big(\Phi(\xi)\big)=\frac12\sum_{i,j=1}^m a^{ij}(\xi)\Big(
        \partial_{ij}\widetilde\varphi(\xi)-\sum_{\ell=1}^m\Gamma_{ij}^{\ell}(\xi)\partial_{\ell}\widetilde\varphi(\xi)
    \Big)+\sum_{\ell=1}^mb^{\ell}(\xi)\partial_{\ell}\widetilde\varphi(\xi).
    \label{eq:coordinate-form-operator}
\end{equation}
To construct a p-Ret discretization consistent with the coordinate generator \eqref{eq:coordinate-form-operator}, we use the local parameterization-based second-order retraction, whose quadratic correction is expressed in terms of the Christoffel symbols.

\begin{definition}[{\cite[Section~2.1]{SchwarzEtAl2025}}]
\label{def:quadratic-parametric-retraction}
For $x=\Phi(\xi)$ and $w=D\Phi(\xi)v\in T_x\M$, with $v\in\R^m$ sufficiently small, define the p-Ret map by
\begin{equation}
    \pRet_x(w)=\Phi\Big(\Big(\xi^\ell+v^\ell-\frac12\sum_{i,j=1}^m\Gamma_{ij}^\ell(\xi)v^iv^j\Big)_{\ell=1}^m\Big).
    \label{eq:quadratic-parametric-retraction}
\end{equation}
\end{definition}

By~\cite[Lemma~2.5]{SchwarzEtAl2025}, $\pRet$ is a second-order retraction with respect to $g$. The resulting p-Ret discretization is given in the following lemma.

\begin{lemma}
\label{lem:pret-discretization}
Let $h=T/K$, $t_k=kh$, and let $Z_k$, $k=1,\ldots,K$, be independent random vectors uniformly distributed on $\mathbb S^{r-1}$. For the operator \eqref{eq:coordinate-form-operator}, define
\begin{equation}
    V_h(\xi,Z)=h\,b(\xi)+\sqrt{rh}\,\Sigma(\xi)Z,
    \label{eq:pret-random-increment}
\end{equation}
where $b(\xi)=\big(b^1(\xi),\ldots,b^m(\xi)\big)^T$ and $a(\xi)=\Sigma(\xi)\Sigma(\xi)^T=\big(a^{ij}(\xi)\big)_{i,j=1}^m,$. Define the one-step coordinate map $\Psi_h$ componentwise by
\begin{equation*}
    \big(\Psi_h(\xi,Z)\big)^\ell=\xi^\ell+V_h^\ell(\xi,Z)-\frac12\sum_{i,j=1}^m\Gamma_{ij}^\ell(\xi)V_h^i(\xi,Z)V_h^j(\xi,Z),
    \quad
    \ell=1,\ldots,m.
\end{equation*}
Then the p-Ret discretization is given by
\begin{equation}
    \xi_{k+1}=\Psi_h(\xi_k,Z_{k+1}),
    \quad
    X_{k+1}^h=\Phi(\xi_{k+1}).
    \label{eq:one-step-pret-transition}
\end{equation}
\end{lemma}

\begin{proof}

Since $Z$ is uniformly distributed on $\mathbb S^{r-1}$, we have $\E[Z]=0$ and $\E[ZZ^T]=\frac1r I_r$. Hence
\begin{equation*}
    \E\big[V_h(\xi,Z)\big]=h\,b(\xi),
    \quad
    \E\Big[\big(V_h(\xi,Z)-h\,b(\xi)\big)\big(V_h(\xi,Z)-h\,b(\xi)\big)^T\Big]=h\Sigma(\xi)\Sigma(\xi)^T.
\end{equation*}
Thus $V_h$ has the drift and covariance required by \eqref{eq:coordinate-form-operator}. Applying the second-order retraction \eqref{eq:quadratic-parametric-retraction} to this increment gives \eqref{eq:one-step-pret-transition}.
\end{proof}

After each step, the parameter representative is adjusted according to the prescribed identifications. The p-Ret update \eqref{eq:one-step-pret-transition} defines a time-homogeneous Markov chain $\{X_k^h\}_{k\ge0}$. For an interior step, before applying the numerical boundary-exit procedure, define the one-step transition operator by
\begin{equation}
    \big(\mathcal P_h\varphi\big)(x)=\E\big[\varphi(X_{k+1}^h)\,\bigm|\,X_k^h=x\big].
    \label{eq:one-step-transition-operator}
\end{equation}
For $x=\Phi(\xi)$, equivalently,
\begin{equation}
    \big(\mathcal P_h\varphi\big)\big(\Phi(\xi)\big)=\E\big[\varphi\Big(\Phi\big(\Psi_h(\xi,Z)\big)\Big)\big]=\E\big[\widetilde\varphi\big(\Psi_h(\xi,Z)\big)\big].
    \label{eq:coordinate-one-step-transition}
\end{equation}

The following lemma establishes the local weak consistency of the one-step operator $\mathcal P_h$ with the generator $\Lop$.

\begin{lemma}
\label{lem:discrete-generator-consistency}
Let $U_0$ be compactly contained in $U$ and let $\varphi\in C^5(\M)$. Then, uniformly for $\xi\in U_0$,
\begin{equation}
    \frac{\big(\mathcal P_h\varphi\big)\big(\Phi(\xi)\big)-\varphi\big(\Phi(\xi)\big)}{h}=\big(\Lop\varphi\big)\big(\Phi(\xi)\big)+O(h),
    \quad
    h\to0.
    \label{eq:discrete-generator-consistency}
\end{equation}
\end{lemma}

\begin{proof}
Set $V=V_h(\xi,Z)$ and $\Delta=\Psi_h(\xi,Z)-\xi$. By the definitions of $\Psi_h$ in \eqref{eq:one-step-pret-transition}, we have $\Delta^\ell=V^\ell-\frac12\sum_{i,j=1}^m\Gamma_{ij}^{\ell}(\xi)V^iV^j$. Since $Z$ is uniformly distributed on $\mathbb S^{r-1}$, we have $\E[Z]=0$, $\E[ZZ^T]=\frac1r I_r$, and all odd moments of $Z$ vanish. It follows that
\begin{equation*}
    \E[\Delta^\ell]=h\Big(b^\ell-\frac12\sum_{i,j=1}^m a^{ij}\Gamma_{ij}^{\ell}\Big)+O(h^2)
\end{equation*}
and $\E[\Delta^i\Delta^j]=h\,a^{ij}+O(h^2)$, where all coefficients are evaluated at $\xi$. Moreover, the third- and fourth-order moments of $\Delta$ that occur in the Taylor expansion are $O(h^2)$.

Since $\big(\mathcal P_h\varphi\big)\big(\Phi(\xi)\big)=\E\big[\widetilde\varphi(\xi+\Delta)\big]$, expanding $\widetilde\varphi(\xi+\Delta)$ about $\xi$ and taking expectations therefore gives
\begin{eqnarray*}
    &&\big(\mathcal P_h\varphi\big)\big(\Phi(\xi)\big)=\widetilde\varphi(\xi)+h\sum_{\ell=1}^m\Big(b^\ell-\frac12\sum_{i,j=1}^m a^{ij}\Gamma_{ij}^{\ell}\Big)\partial_\ell\widetilde\varphi(\xi)\\
    &&\hspace{36pt}
    +\frac h2\sum_{i,j=1}^m a^{ij}\partial_{ij}\widetilde\varphi(\xi)+O(h^2).
\end{eqnarray*}
The coefficient of $h$ is precisely $\big(\Lop\varphi\big)\big(\Phi(\xi)\big)$ by \eqref{eq:coordinate-form-operator}, which proves the result.
\end{proof}

\subsection{Efficient algorithms and implementation}
\label{sec:efficient-algorithms}
The p-Ret transition \eqref{eq:one-step-pret-transition} provides the common discrete dynamics underlying the MC and direct RQMC methods considered in this section. We first formulate the corresponding discrete Feynman--Kac functionals, independently of the sampling rule, and then specify the MC and direct RQMC constructions through their respective choices of random inputs.

\subsubsection{Discrete Feynman--Kac functionals}
\label{subsec:discrete-path-functionals}
In this subsection, we evaluate the solution at $T$ using the discrete p-Ret path \eqref{eq:one-step-pret-transition} with the periodic wrapping convention of Subsection~\ref{subsec:coordinate-pret}.

\underline{For the closed problem \eqref{eq:closed-parabolic-problem}}, discretizing the Feynman--Kac representation \eqref{eq:closed-feynman-kac} yields the path functional
\begin{equation*}
    S_h^{\cl}=u_0(X_K)+h\sum_{k=0}^{K-1}f(X_k,T-t_k).
\end{equation*}

\underline{For the Dirichlet problem \eqref{eq:dirichlet-parabolic-problem}}, we discretize the stopped Feynman--Kac representation \eqref{eq:boundary-feynman-kac}. For the numerical localization of the first boundary crossing, let $\chi$ be a smooth defining function for the physical boundary of $U$, defined on an open neighborhood $W$ containing all relevant p-Ret interpolation paths, with $\chi>0$ in $U\cap W$ and $\chi<0$ in $W\setminus\overline U$. If the candidate $\Psi_h(\xi_k,Z_{k+1})$ leaves $U$, define the p-Ret interpolation $\gamma_k:[0,1]\to\R^m$ componentwise by
\begin{equation*}
    \big(\gamma_k(\lambda)\big)^\ell=\xi_k^\ell+\lambda V_h^\ell(\xi_k,Z_{k+1})-\frac{\lambda^2}{2}\sum_{i,j=1}^m\Gamma_{ij}^\ell(\xi_k)V_h^i(\xi_k,Z_{k+1})V_h^j(\xi_k,Z_{k+1}),
\end{equation*}
for $\ell=1,\ldots,m$, $0\le\lambda\le1$, and let
\begin{equation}
    \lambda_k=\inf\big\{\lambda\in(0,1]:\chi\big(\gamma_k(\lambda)\big)\le0\big\}.
    \label{eq:numerical-boundary-fraction}
\end{equation}
The numerical exit time and boundary point are $\widehat\tau_h=t_k+\lambda_kh$ and $\widehat X_{\partial}=\Phi\big(\gamma_k(\lambda_k)\big)$. 

Define the accumulated source term by
\begin{equation*}
    I_k=h\sum_{\ell=0}^{k-1}f(X_\ell,T-t_\ell),
    \quad
    I_0=0.
\end{equation*}
With the convention $\widehat\tau_h=+\infty$ if no numerical exit occurs before $T$, the discrete functional is
\begin{equation*}
    S_h^{\bd}=
    \begin{cases}
    u_0(X_K)+I_K,
        & \widehat\tau_h>T,\\[2mm]
    \psi\big(\widehat X_{\partial},T-\widehat\tau_h\big)+I_k+\lambda_kh\,f(X_k,T-t_k),
        & \widehat\tau_h=t_k+\lambda_kh\le T.
    \end{cases}
\end{equation*}
Here, $k$ denotes the index of the first step for which the p-Ret candidate leaves $U$, and the empty sum is understood as zero when $k=0$. The last term accounts for the source contribution over the partial exit step, whose crossing is localized along the p-Ret interpolation by \eqref{eq:numerical-boundary-fraction}.

For $\star\in\{\cl,\bd\}$, the corresponding time-discrete
approximation is defined by
\begin{equation}\label{eq:discrete-solution}
    u_h^\star(x,T)=\E_x\big[S_h^\star\big].
\end{equation}
The functionals $S_h^{\cl}$ and $S_h^{\bd}$ are common to all the sampling methods considered below. 

\subsubsection{Monte Carlo and direct RQMC sampling}
\label{subsubsec:sampling-constructions}

Throughout the remainder of the paper, $\mathcal D:[0,1)^{r-1}\rightarrow\mathbb S^{r-1}$ denotes the recursive inverse-beta transformation, so that $\mathcal D(U)\sim\operatorname{Unif}(\mathbb S^{r-1})$ whenever $U\sim\operatorname{Unif}([0,1)^{r-1})$. Thus, each p-Ret step uses exactly $r-1$ uniform input variables.

\begin{enumerate}[
    label=\textbf{\textup{(\roman*)}},
    leftmargin=2.5em,
    labelsep=0.8em,
    itemsep=0.3em,
    topsep=0.3em
]

\item \textbf{Monte Carlo sampling.}
For each path $i=1,\ldots,n$ and time step $k=1,\ldots,K$, generate independent random vectors
\begin{equation*}
    U_k^{(i)}\sim\operatorname{Unif}([0,1)^{r-1}),
    \quad
    Z_k^{(i)}=\mathcal D\big(U_k^{(i)}\big).
\end{equation*}

\item \textbf{Direct randomized quasi-Monte Carlo sampling.}
Set $s=(r-1)K$, the nominal integration dimension of direct RQMC, and let $n=2^\ell$. We use the first $n$ points of a Sobol' sequence in dimension $s$, which form a digital $(\mathfrak t,\ell,s)$-net in base $2$. Smaller $\mathfrak t$ corresponds to finer equidistribution, and $\mathfrak t$ is kept fixed as $\ell$ increases. We use nested uniform scrambling, which preserves the base-$2$ digital-net property almost surely and makes every scrambled point marginally uniform on $[0,1)^s$; see \cite{Owen1997}.

Generate a nested-uniformly scrambled Sobol' point set
\begin{equation*}
Q_n=\big\{z^{(1)},\ldots,z^{(n)}\big\}\subset[0,1)^s.
\end{equation*}
For each $i$, partition $z^{(i)}=\big(z_1^{(i)},\ldots,z_K^{(i)}\big)$, with $z_k^{(i)}\in[0,1)^{r-1}$, and define
\begin{equation*}
    Z_k^{(i)}=\mathcal D\big(z_k^{(i)}\big),
    \quad
    k=1,\ldots,K.
\end{equation*}
Each $s$-dimensional point generates one complete p-Ret path; unused blocks are discarded after a boundary exit. The nominal dimension $s=(r-1)K$ increases as $h$ decreases. In the boundary case, the associated integrand may also be discontinuous across interfaces where the numerical exit step changes. This issue is examined in the subsequent error analysis.

\end{enumerate}

\begin{remark}
Although the analysis below focuses on nested-uniformly scrambled Sobol' point sets, the sampling construction can also be implemented with scrambled digital nets in any integer base $\mathfrak b\ge2$ and with randomized lattice rules. Corresponding variance results require the randomization, projection properties, and integrand regularity appropriate to the chosen point set.
\end{remark}

\subsubsection{Estimators and implementation}
\label{subsubsec:estimators}
Let $M\in\{\MC,\allowbreak\RQMC\}$ denote the sampling method, and let $\star\in\{\cl,\bd\}$ indicate whether the problem is closed or subject to Dirichlet boundary conditions. For each method, let $S_h^{M,\star,(i)}$, $i=1,\ldots,n$, be the discrete functional generated from the common p-Ret dynamics. The corresponding estimator is
\begin{equation}
    \widehat u_{n,h}^{M,\star}(x,T)=\frac1n\sum_{i=1}^nS_h^{M,\star,(i)}.
    \label{eq:unified-estimator}
\end{equation}
The superscript $M$ records only the sampling construction; the underlying discrete functional is the same $S_h^\star$ defined above.

Algorithm~\ref{alg:mc-direct-pret-estimator} summarizes one randomized replication of the MC and direct RQMC estimators. Their only difference lies in the generation of the uniform inputs $z_k^{(i)}$.
For direct RQMC, each particle retains the blocks of its original $(r-1)K$-dimensional point throughout the simulation, and the unused blocks are discarded after a numerical exit.
\begin{algorithm}[t]
\caption{MC and direct RQMC p-Ret estimators}
\label{alg:mc-direct-pret-estimator}
\begin{algorithmic}[1]

\Require Initial coordinate $\xi_0$, particle number $n$, number of steps $K$, method $M\in\{\MC,\RQMC\}$, and problem type $\star\in\{\cl,\bd\}$.
\Ensure $\widehat u_{n,h}^{M,\star}(x,T)$.

\State Initialize $\xi_0^{(i)}=\xi_0$, $I_0^{(i)}=0$, $i=1,\ldots,n$, and mark all particles active.

\If{$M=\RQMC$}
    \State Generate the RQMC point set $Q_n=\{z^{(1)},\ldots,z^{(n)}\}\subset[0,1)^{(r-1)K}$ and write
    $z^{(i)}=(z_1^{(i)},\ldots,z_K^{(i)})$, with $z_k^{(i)}\in[0,1)^{r-1}$.
\EndIf

\For{$k=0,\ldots,K-1$}
    \State Let $\mathcal I_k$ denote the set of active particle indices.
    \If{$\mathcal I_k=\varnothing$}
        \State \textbf{break}
    \EndIf

    \If{$M=\MC$}
        \State Generate independent $z_{k+1}^{(i)}\sim\operatorname{Unif}([0,1)^{r-1})$, $i\in\mathcal I_k$.
    \EndIf

    \For{each $i\in\mathcal I_k$}
        \State Set $Z_{k+1}^{(i)}=\mathcal D(z_{k+1}^{(i)})$ and compute the p-Ret candidate
        \begin{equation*}
            \xi_{k+1}^{\mathrm{cand},(i)}=\Psi_h\big(\xi_k^{(i)},Z_{k+1}^{(i)}\big).
        \end{equation*}
        \State Apply the prescribed periodic parameter identifications to $\xi_{k+1}^{\mathrm{cand},(i)}$.

        \If{$\star=\bd$ and $\xi_{k+1}^{\mathrm{cand},(i)}\notin U$}
            \State Compute $\lambda_k^{(i)}$ from \eqref{eq:numerical-boundary-fraction}, and set $S_h^{M,\bd,(i)}$ according to the stopped functional and mark particle $i$ inactive.
        \Else
            \State Set $\xi_{k+1}^{(i)}=\xi_{k+1}^{\mathrm{cand},(i)}$ and update $I_{k+1}^{(i)}$.
        \EndIf
    \EndFor
\EndFor

\For{each particle $i$ still active at time $T$}
    \State Set $S_h^{M,\star,(i)}=u_0\big(\Phi(\xi_K^{(i)})\big)+I_K^{(i)}$.
\EndFor

\State \Return $\widehat u_{n,h}^{M,\star}(x,T)$ from \eqref{eq:unified-estimator}.

\end{algorithmic}
\end{algorithm}


\section{Error estimates for temporal discretization and path-space sampling}
\label{sec:error-estimates}
In this section, we analyze the temporal error of the p-Ret discretization, establish the MC reference bound, and derive the sampling error estimates for direct RQMC.

\subsection{Error decomposition and preliminary estimates}
\label{subsec:preliminary-error-estimates}

We begin by introducing the total error decomposition and collecting the assumptions used throughout the analysis. We use the unified notation $u^\star$, $\star\in\{\cl,\bd\}$, for the exact solution, where $u^\cl$ and $u^\bd$ correspond to the closed and Dirichlet problems, respectively. 

Recalling the time-discrete approximation defined in \eqref{eq:discrete-solution}, for $M\in\{\MC,\allowbreak\RQMC\}$ and $\star\in\{\cl,\bd\}$, we decompose the total error as
\begin{equation}
    \widehat u_{n,h}^{M,\star}(x,T)-u^\star(x,T)=\widehat u_{n,h}^{M,\star}(x,T)-u_h^\star(x,T)+u_h^\star(x,T)-u^\star(x,T).
    \label{eq:total-error-decomposition}
\end{equation}
The temporal weak error is common to both sampling methods, whereas the sampling error depends on the sampling construction.
For the Dirichlet problem, the temporal error additionally contains the error caused by the numerical boundary localization, which is discussed in Remark~\ref{rem:dirichlet-temporal-error}.

The subsequent analysis is carried out under the following standing
assumption.

\begin{assumption}\label{ass:preliminary-error}
Suppose that the assumptions in Subsections~\ref{subsec:setting-fk} and \ref{subsec:coordinate-pret} hold. In addition, assume that the coordinate drift and diffusion coefficients are globally Lipschitz. Namely, there exist constants $L_b,L_\Sigma>0$ such that
\begin{align}
    \lVert b(\xi)-b(\eta)\rVert
    &\le L_b\lVert\xi-\eta\rVert,
    \quad
    \xi,\eta\in\overline U,
    \label{eq:drift-lipschitz}\\
    \lVert\Sigma(\xi)-\Sigma(\eta)\rVert_{\mathrm F}
    &\le L_\Sigma\lVert\xi-\eta\rVert,
    \quad
    \xi,\eta\in\overline U.
    \label{eq:diffusion-lipschitz}
\end{align}
Here, $\lVert\cdot\rVert_{\mathrm F}$ denotes the Frobenius norm.
\end{assumption}

\begin{assumption}\label{ass:bounded-data}
There exist constants $M_0,M_f>0$ such that
\begin{equation}
    |u_0(x)|\le M_0,
    \quad
    |f(x,t)|\le M_f,
    \quad
    x\in\M,\quad t\in[0,T].
    \label{eq:bounded-initial-source-data}
\end{equation}
For the Dirichlet problem, assume further that there exists a constant $M_\psi>0$ such that 
\begin{equation}
    |\psi(x,t)|\le M_\psi,
    \quad
    x\in\partial\M,\quad t\in[0,T].
    \label{eq:bounded-boundary-data}
\end{equation}
\end{assumption}

We first record the weak temporal convergence of the p-Ret
discretization on closed manifolds.

\begin{lemma}
\label{lem:weak-temporal-convergence}
Suppose that Assumption~\ref{ass:preliminary-error} holds. Assume further that $u^\cl$ is twice continuously differentiable in time, that $u^\cl(\cdot,t)\in C^5(\M)$ uniformly for $t\in[0,T]$, and that $f\in C^1\big([0,T];C(\mathcal M)\big)$, with all the corresponding derivatives bounded. Then it holds that
\begin{equation}
    \|u_h^\cl(\cdot,T)-u^\cl(\cdot,T)\|_{L^\infty(\M)}\le Ch,
    \label{eq:weak-temporal-convergence}
\end{equation}
where $C$ is a positive constant independent of $h$.
\end{lemma}

\begin{proof}
For $k=0,\ldots,K$, set $v_k^\cl=u^\cl(\cdot,T-t_k)$, and let $v_{h,k}^\cl$ denote the corresponding time-discrete backward value function. Then
\begin{equation*}
    v_{h,K}^\cl=u_0,
    \quad
    v_{h,k}^\cl=\mathcal P_hv_{h,k+1}^\cl+h f(\cdot,T-t_k),
    \quad
    k=K-1,\ldots,0.
\end{equation*}
By the definition of the discrete Feynman--Kac functional, $v_{h,0}^\cl=u_h^\cl(\cdot,T)$, whereas $v_0^\cl=u^\cl(\cdot,T)$.

Applying Lemma~\ref{lem:discrete-generator-consistency} to $v_{k+1}^\cl$ and using periodic extension in the closed case, we obtain
\begin{equation*}
    \|\mathcal P_hv_{k+1}^\cl-v_{k+1}^\cl-h\Lop v_{k+1}^\cl\|_{L^\infty(\M)}\le Ch^2,
\end{equation*}
uniformly for $k=0,\ldots,K-1$. Since $T-t_k=T-t_{k+1}+h$, Taylor's theorem applied to $u^\cl(x,\cdot)$ at $T-t_{k+1}$ yields, uniformly in $x\in\M$ and $k=0,\ldots,K-1$,
\begin{equation*}
    \|v_k^\cl-v_{k+1}^\cl-h\partial_tu^\cl(\cdot,T-t_{k+1})\|_{L^\infty(\M)}\le Ch^2.
\end{equation*}
Since $T-t_k=T-t_{k+1}+h$ and $\partial_t f$ is bounded, the mean value theorem gives
\begin{equation*}
    \|f(\cdot,T-t_k)-f(\cdot,T-t_{k+1})\|_{L^\infty(\M)}\le h\|\partial_t f\|_{L^\infty(\M\times[0,T])}\le Ch.
\end{equation*}
Using $\Lop v_{k+1}^\cl=\Lop u^\cl(\cdot,T-t_{k+1})$ and the equation $\partial_tu^\cl(\cdot,T-t_{k+1})=\Lop v_{k+1}^\cl+f(\cdot,T-t_{k+1})$, we decompose the local defect as
\begin{eqnarray*}
&&\mathcal P_hv_{k+1}^\cl+h f(\cdot,T-t_k)-v_k^\cl=\big(\mathcal P_hv_{k+1}^\cl-v_{k+1}^\cl-h\Lop v_{k+1}^\cl\big)\\
&&\hspace{15pt}
-\big(v_k^\cl-v_{k+1}^\cl-h\partial_tu^\cl(\cdot,T-t_{k+1})\big)+h\big(f(\cdot,T-t_k)-f(\cdot,T-t_{k+1})\big).
\end{eqnarray*}
Consequently, the preceding estimates and the triangle inequality yield
\begin{eqnarray*}
&&\|\mathcal P_hv_{k+1}^\cl+h f(\cdot,T-t_k)-v_k^\cl\|_{L^\infty(\M)}\le\|\mathcal P_hv_{k+1}^\cl-v_{k+1}^\cl-h\Lop v_{k+1}^\cl\|_{L^\infty(\M)}\\
&&\hspace{-8pt}
+\|v_k^\cl-v_{k+1}^\cl-h\partial_tu^\cl(\cdot,T-t_{k+1})\|_{L^\infty(\M)}+h\|f(\cdot,T-t_k)-f(\cdot,T-t_{k+1})\|_{L^\infty(\M)}\\
&&\hspace{-8pt}\le Ch^2+Ch^2+h(Ch)\le Ch^2.
\end{eqnarray*}

Set $e_k=\|v_{h,k}^\cl-v_k^\cl\|_{L^\infty(\M)}$. By the backward recursion, adding and subtracting $\mathcal P_hv_{k+1}^\cl$, and using the triangle inequality together with the nonexpansiveness of the Markov operator $\mathcal P_h$ in the supremum norm, we obtain
\begin{align*}
e_k &=\big\|\mathcal P_h\big(v_{h,k+1}^\cl-v_{k+1}^\cl\big)+\mathcal P_hv_{k+1}^\cl+h f(\cdot,T-t_k)-v_k^\cl\big\|_{L^\infty(\M)}\\
&\le\big\|\mathcal P_h\big(v_{h,k+1}^\cl-v_{k+1}^\cl\big)\big\|_{L^\infty(\M)}+\big\|\mathcal P_hv_{k+1}^\cl+h f(\cdot,T-t_k)-v_k^\cl\big\|_{L^\infty(\M)}\\
&\le e_{k+1}+Ch^2,\quad k=K-1,\ldots,0.
\end{align*}
Since $v_{h,K}^\cl=v_K^\cl=u_0$, we have $e_K=0$. Iterating the preceding estimate and using $K=T/h$ yield $e_0\le KCh^2=CTh$. Since $v_{h,0}^\cl=u_h^\cl(\cdot,T)$ and $v_0^\cl=u^\cl(\cdot,T)$, the result follows.
\end{proof}

\begin{remark}
\label{rem:dirichlet-temporal-error}
For the Dirichlet problem, the preceding local consistency argument applies to all p-Ret steps completed before the numerical boundary exit. Suppose that the weak error introduced by the numerical boundary procedure is bounded by $\mathcal E_{\partial}$, including the approximation of the exit time and exit location and the corresponding evaluation of the boundary datum $\psi$. Then the same backward stability argument gives
\begin{equation}
    \big\|u_h^\bd(\cdot,T)-u^\bd(\cdot,T)\big\|_{L^\infty(\M)}\le Ch+\mathcal E_{\partial}.
    \label{eq:dirichlet-temporal-error}
\end{equation}
In particular, if $\mathcal E_{\partial}\le Ch$, then the full temporal weak error is of first order. Since the estimate of $\mathcal E_{\partial}$ depends on the specific numerical boundary procedure and the regularity of $\psi$, it is not pursued here.
\end{remark}

Lemma~\ref{lem:weak-temporal-convergence} establishes the first-order temporal weak error for the closed problem, while Remark~\ref{rem:dirichlet-temporal-error} separates the corresponding Dirichlet error into the interior temporal contribution and the additional error introduced by the numerical boundary procedure. The sampling estimates below are formulated relative to the time-discrete quantity $u_h^\star$ and therefore apply to both the closed and Dirichlet discrete functionals. We begin with the Monte Carlo reference bound.

\begin{lemma}
\label{lem:mc-reference-bound}
Suppose that Assumption~\ref{ass:bounded-data} holds. Let $\star\in\{\cl,\bd\}$ and let $x\in\M$ be fixed. Then it holds that
\begin{equation}
    \operatorname{Var}\big(\widehat u_{n,h}^{\MC,\star}(x,T)\big)\le\frac{C_{\MC}}{n},
    \quad
    \E\big[\big|\widehat u_{n,h}^{\MC,\star}(x,T)-u_h^\star(x,T)\big|^2\big]\le\frac{C_{\MC}}{n},
    \label{eq:mc-reference-bounds}
\end{equation}
where $C_{MC}$ is a positive constant independent of $n$ and $h$.
\end{lemma}

\begin{proof}
Let $S_h^{\star,[\nu]}(x,T)$, $\nu=1,\ldots,n$, be independent and
identically distributed random variables with the same distribution as $S_h^\star(x,T)$. By the definitions of $\widehat u_{n,h}^{\MC,\star}(x,T)$ and $u_h^\star(x,T)$ in \eqref{eq:unified-estimator} and \eqref{eq:discrete-solution}, respectively, we have
\begin{equation*}
    \E\big[\widehat u_{n,h}^{\MC,\star}(x,T)\big]=\frac1n\sum_{\nu=1}^n\E\big[S_h^{\star,[\nu]}(x,T)\big]=\E\big[S_h^\star(x,T)\big]=u_h^\star(x,T).
\end{equation*}
By Assumption~\ref{ass:bounded-data} and $Kh=T$, we have
\begin{equation*}
    |S_h^\star(x,T)|\le M_\star+TM_f,
    \quad
    M_\star=\begin{cases}
        M_0, & \star=\cl,\\
        \max\{M_0,M_\psi\}, & \star=\bd.
    \end{cases}
\end{equation*}
Consequently,
\begin{equation*}
    \operatorname{Var}\big(S_h^\star(x,T)\big)=\E\big[|S_h^\star(x,T)|^2\big]-\big|\E\big[S_h^\star(x,T)\big]\big|^2\le\E\big[|S_h^\star(x,T)|^2\big]\le\big(M_\star+TM_f\big)^2.
\end{equation*}
By independence,
\begin{align*}
\operatorname{Var}\big(\widehat u_{n,h}^{\MC,\star}(x,T)\big)&=\operatorname{Var}\Big(\frac1{n}\sum_{\nu=1}^nS_h^{\star,[\nu]}(x,T)\Big)=\frac1{n^2}\sum_{\nu=1}^n\operatorname{Var}\big(S_h^{\star,[\nu]}(x,T)\big)\\
&=\frac1n\operatorname{Var}\big(S_h^\star(x,T)\big)\le\frac{\big(M_\star+TM_f\big)^2}{n}\le\frac{C_{\MC}}{n}.
\end{align*}
Finally, using $\E\big[\widehat u_{n,h}^{\MC,\star}(x,T)\big]=u_h^\star(x,T)$, we obtain
\begin{align*}
\E\big[\big|\widehat u_{n,h}^{\MC,\star}(x,T)-u_h^\star(x,T)\big|^2\big]&=\E\big[\big|\widehat u_{n,h}^{\MC,\star}(x,T)-\E\big[\widehat u_{n,h}^{\MC,\star}(x,T)\big]\big|^2\big]\\
&=\operatorname{Var}\big(\widehat u_{n,h}^{\MC,\star}(x,T)\big)\le\frac{C_{\MC}}{n}.
\end{align*}
This ends the proof.
\end{proof}

Thus, the Monte Carlo variance converges at the rate $O(n^{-1})$, while the corresponding sampling RMSE converges at the rate $O(n^{-1/2})$. Together with the temporal weak error estimate, this provides the reference framework for the subsequent analysis of the randomized quasi-Monte Carlo estimators. We next study the sampling variance of the direct RQMC estimator relative to the same time-discrete solution $u_h^\star$.

\subsection{Randomized quasi--Monte Carlo error estimates}
\label{subsec:rqmc-error-estimates}

We next analyze the sampling error of the direct RQMC estimator. We first consider the closed manifold problem. Throughout this subsection, $h$ and $K=T/h$ are fixed. Recall from Subsection~\ref{subsubsec:sampling-constructions} that $s=(r-1)K$. All constants below are independent of $n$, but may depend on $h$.

Let $z=(z_1,\ldots,z_K)\in[0,1)^s$, with $z_k\in[0,1)^{r-1}$, and define $Z_k(z)=\mathcal D(z_k)$ for $k=1,\ldots,K$. Starting from the coordinate representation $\xi_0$ of the initial point $x$, define the corresponding p-Ret path by
\begin{equation}
    \widetilde{\xi}_{h,0}(z)=\xi_0,
    \quad
    \widetilde{\xi}_{h,k}(z)=\Psi_h\big(\widetilde{\xi}_{h,k-1}(z),Z_k(z)\big),
    \quad
    k=1,\ldots,K.
    \label{eq:rqmc-path-map}
\end{equation}
Periodic coordinates are wrapped according to the convention introduced previously.

Since every path on a closed manifold is propagated through all $K$ time steps, no stopping discontinuity occurs. Define
\begin{equation}
    G_h^\cl(z)=u_0\big(\Phi(\widetilde{\xi}_{h,K}(z))\big)+h\sum_{k=0}^{K-1}f\big(\Phi(\widetilde{\xi}_{h,k}(z)),T-t_k\big).
    \label{eq:closed-rqmc-integrand}
\end{equation}
Let $U=(U_1,\ldots,U_K)\sim\operatorname{Unif}\big([0,1)^s\big)$, where $U_k\in[0,1)^{r-1}$. By construction, $S_h^\cl=G_h^\cl(U)$. Hence, by the definition of $u_h^\cl$,
\begin{equation}
    u_h^\cl(x,T)=\E\big[G_h^\cl(U)\big]=\int_{[0,1)^s}G_h^\cl(z)\dd z.
    \label{eq:closed-rqmc-integral-representation}
\end{equation}

For the nested-uniformly scrambled Sobol' point set $Q_n=\big\{z^{(1)},\ldots,z^{(n)}\big\}\subset[0,1)^s$, the direct RQMC estimator is
\begin{equation}
    \widehat u_{n,h}^{\RQMC,\cl}(x,T)=\frac1n\sum_{i=1}^nG_h^\cl\big(z^{(i)}\big).
    \label{eq:closed-direct-rqmc-estimator}
\end{equation}

We first establish the basic RQMC estimate without imposing any additional smoothness condition on $G_h^\cl$.

\begin{lemma}
\label{lem:closed-direct-rqmc-general}
Suppose that Assumption~\ref{ass:bounded-data} holds and that
$Q_n$ consists of the first $n=2^\ell$ points of a nested-uniformly scrambled Sobol' sequence in dimension $s$. Then $\E\big[\widehat u_{n,h}^{\RQMC,\cl}(x,T)\big]=u_h^\cl(x,T)$, and for fixed $h$, we have
\begin{equation}
    \operatorname{Var}\big(\widehat u_{n,h}^{\RQMC,\cl}(x,T)\big)=o\big(n^{-1}\big),
    \quad
    n\to\infty.
    \label{eq:closed-direct-rqmc-general-rate}
\end{equation}
\end{lemma}

\begin{proof}
Nested uniform scrambling makes each $z^{(i)}$ marginally uniform on $[0,1)^s$ \cite{Owen1997}. Therefore,
\begin{equation*}
\E\big[\widehat u_{n,h}^{\RQMC,\cl}(x,T)\big]=\frac1n\sum_{i=1}^n\E\big[G_h^\cl\big(z^{(i)}\big)\big]=\int_{[0,1)^s}G_h^\cl(z)\dd z=u_h^\cl(x,T).
\end{equation*}

Moreover, Assumption~\ref{ass:bounded-data} and $Kh=T$ give $\big|G_h^\cl(z)\big|\le M_0+TM_f$ for $z\in[0,1)^s$, and hence $G_h^\cl\in L^2\big([0,1)^s\big)$. The general scrambled-net variance result for square-integrable integrands \cite{HoOwen2026RQMC,Owen1997} then yields \eqref{eq:closed-direct-rqmc-general-rate}.
\end{proof}

Lemma~\ref{lem:closed-direct-rqmc-general} gives an asymptotic improvement over the Monte Carlo variance $O(n^{-1})$, but does not specify a fixed algebraic rate. For the sphere transformation used in the present method, derivatives with respect to the original uniform variables may become singular near the boundary of the unit cube. We therefore pass to a Gaussian representation and first introduce the notation and polynomial-growth function classes used in the RQMC analysis of \cite{OuyangPanHe2026,OuyangWangHe2024}. We first introduce the notation and function classes needed for the subsequent analysis.

For an integer $q\ge1$, write $1:q=\{1,\ldots,q\}$. For $v\subseteq1:q$, let $\partial_y^v F:=\big(\prod_{j\in v}\partial_{y_j}\big)F$ denote the mixed partial derivative of $F$ with respect to the coordinates indexed by $v$. We say that $F:\mathbb R^q\to\mathbb R^m$ belongs to $C_p(\mathbb R^q;\mathbb R^m)$ if there exist constants $M,B,\kappa>0$ such that
\begin{equation*}
    |F(y)|\le M\|y\|_2^\kappa+B,
    \quad
    y\in\mathbb R^q.
\end{equation*}
When $m=1$, we write simply $C_p(\mathbb R^q)$.

Let $S^q(\mathbb R^q)$ denote the class of functions $F:\mathbb R^q\to\mathbb R$ such that $\partial_y^vF$ exists and is continuous for every $v\subseteq1:q$, and for $M,B,\kappa>0$, define the polynomial growth class
\begin{equation}
    G_p(M,B,\kappa):=\Big\{F\in S^q(\mathbb R^q):\sup_{v\subseteq1:q}|\partial_y^vF(y)|\le M\|y\|_2^\kappa+B,
    \quad
    y\in\mathbb R^q\Big\}.
    \label{eq:polynomial-growth-class}
\end{equation}

Let $\Phi_{\rm N}$ denote the cumulative distribution function of the standard normal distribution, and use the same notation for its componentwise action. For $y=(y_1,\ldots,y_K)\in\mathbb R^s$, with $y_k\in\mathbb R^{r-1}$, define the Gaussian-transformed integrand by
\begin{equation}
    \widetilde G_h^\cl(y):=G_h^\cl\big(\Phi_{\rm N}(y)\big),
    \quad
    y\in\mathbb R^s.
    \label{eq:closed-rqmc-gaussian-integrand}
\end{equation}
If $Y\sim N(0,I_s)$, then $\Phi_{\rm N}(Y)$ is uniformly distributed on $[0,1)^s$, and therefore $u_h^\cl(x,T)=\mathbb E\big[\widetilde G_h^\cl(Y)\big]$. 

We next establish the regularity of the sphere transformation after the Gaussian change of variables.

\begin{lemma}
\label{lem:gaussian-sphere-transform}
Suppose that $\mathcal D$ is the recursive inverse-beta transformation that generates a uniform point on $S^{r-1}$ from $r-1$ independent uniform variables. Then, for every $v\subseteq1:(r-1)$, we have
\begin{equation}
    \sup_{y\in\mathbb R^{r-1}}\Big\|\partial_y^v\mathcal D\big(\Phi_{\rm N}(y)\big)\Big\|_2<\infty.
    \label{eq:gaussian-sphere-transform-bound}
\end{equation}
\end{lemma}

\begin{proof}

For $r=2$, $\mathcal D\big(\Phi_{\rm N}(y)\big)=\big(\cos\big(2\pi\Phi_{\rm N}(y)\big),\sin\big(2\pi\Phi_{\rm N}(y)\big)\big)$, whose function value and first derivative are bounded since the standard normal density and the trigonometric functions are bounded.

Suppose now that $r>2$. At the $j$-th recursive step with $j=1,\ldots,r-2$, let $p=r-j+1$. If $F_p$ denotes the distribution function of the first coordinate of a uniformly distributed point on $S^{p-1}$, then its density is $f_p(w)=c_p(1-w^2)^{(p-3)/2}$ with $-1<w<1$, for some $c_p>0$. The inverse-beta construction used in $\mathcal D$ is precisely the inverse distribution function $F_p^{-1}$. Define
\begin{equation*}
    q_p(y)=F_p^{-1}\big(\Phi_{\rm N}(y)\big),
    \quad
    \rho_p(y)=\sqrt{1-q_p(y)^2}.
\end{equation*}
Differentiating $F_p\big(q_p(y)\big)=\Phi_{\rm N}(y)$ gives
\begin{equation*}
    q_p'(y)=\frac{\Phi_{\rm N}'(y)}{c_p\big(1-q_p(y)^2\big)^{(p-3)/2}},
    \quad
    \rho_p'(y)=-\frac{q_p(y)q_p'(y)}{\rho_p(y)}.
\end{equation*}

As $w\uparrow1$ and $y\to+\infty$, we have $1-F_p(w)\asymp(1-w)^{(p-1)/2}$ and $1-\Phi_{\rm N}(y)\asymp\frac{\Phi_{\rm N}'(y)}{y}$. Since
$F_p\big(q_p(y)\big)=\Phi_{\rm N}(y)$, we obtain
\begin{equation*}
    1-q_p(y)^2\asymp\big(1-\Phi_{\rm N}(y)\big)^{2/(p-1)},\quad y\to+\infty.
\end{equation*}
Therefore, we obtain from the expression for $q_p'(y)$ that
\begin{equation*}
    q_p'(y)\asymp y\big(1-\Phi_{\rm N}(y)\big)^{2/(p-1)}\to 0,\quad y\to+\infty.
\end{equation*}
Moreover, $\rho_p(y)=\sqrt{1-q_p(y)^2}\asymp\big(1-\Phi_{\rm N}(y)\big)^{1/(p-1)}$, and hence
\begin{equation*}
    |\rho_p'(y)|\asymp y\big(1-\Phi_{\rm N}(y)\big)^{1/(p-1)}\to0,\quad y\to+\infty.
\end{equation*}
The same conclusions hold as $y\to-\infty$. Since $q_p'$ and $\rho_p'$ are continuous on $\mathbb R$, both are uniformly bounded.

Indeed, writing $\mathcal D\big(\Phi_{\rm N}(y)\big)=\big(Z_1(y),\ldots,Z_r(y)\big)$, the recursive construction gives
\begin{align*}
    Z_1(y)&=q_r(y_1),\\
    Z_2(y)&=\rho_r(y_1)q_{r-1}(y_2),\\
    Z_3(y)&=\rho_r(y_1)\rho_{r-1}(y_2)q_{r-2}(y_3),\\
    &\ \vdots\\
    Z_{r-2}(y)&=\rho_r(y_1)\rho_{r-1}(y_2)\cdots\rho_4(y_{r-3})q_3(y_{r-2}),\\
    Z_{r-1}(y)&=\rho_r(y_1)\rho_{r-1}(y_2)\cdots\rho_3(y_{r-2})\cos\big(2\pi\Phi_{\rm N}(y_{r-1})\big),\\
    Z_r(y)&=\rho_r(y_1)\rho_{r-1}(y_2)\cdots\rho_3(y_{r-2})\sin\big(2\pi\Phi_{\rm N}(y_{r-1})\big).
\end{align*}

Hence, each variable $y_j$ appears in only one factor of each component. Since $\partial_y^v$, with $v\subseteq1:(r-1)$, differentiates each variable at most once, every mixed derivative involves only the bounded functions $q_p$, $\rho_p$ and their first derivatives, together with the bounded Gaussian and trigonometric factors. Therefore,
\begin{equation*}
    \sup_{y\in\mathbb R^{r-1}}\Big\|\partial_y^v\mathcal D\big(\Phi_{\rm N}(y)\big)\Big\|_2<\infty,
\end{equation*}
which proves \eqref{eq:gaussian-sphere-transform-bound}.
\end{proof}

\begin{assumption}
\label{ass:closed-rqmc-smoothness}

In addition to Assumption~\ref{ass:preliminary-error}, suppose that $r\ge2$. Fix $h$ and set $s=(r-1)K$. Define $\overline u_0=u_0\circ\Phi$, $\overline f(\xi,t)=f\big(\Phi(\xi),t\big)$, and write $\Gamma=(\Gamma_{ij}^{\ell})_{1\le i,j,\ell\le m}$. The functions $b$, $\Sigma$, $\Gamma$, $\overline u_0$, and $\overline f(\cdot,t)$ admit identification-compatible $C^s$ extensions to the parameter covering space. For every multi-index $\alpha\in\mathbb N_0^m$ with $|\alpha|\le s$, all components of $\partial_\xi^\alpha b,\,\partial_\xi^\alpha\Sigma,\,\partial_\xi^\alpha\Gamma$ belong to $C_p(\mathbb R^m)$, and $\partial_\xi^\alpha\overline u_0,\,\partial_\xi^\alpha\overline f(\cdot,t)\in C_p(\mathbb R^m)$. The polynomial-growth bounds for $\partial_\xi^\alpha\overline f(\cdot,t)$ are uniform in $t\in[0,T]$. Moreover, the extensions of $b$ and $\Sigma$ satisfy
\begin{equation*}
    \|b(\xi)\|_2+\|\Sigma(\xi)\|_{\mathrm F}
    \le
    C\big(1+\|\xi\|_2\big),
    \quad
    \xi\in\mathbb R^m,
\end{equation*}
for some constant $C>0$.

\end{assumption}

\begin{lemma}
\label{lem:closed-rqmc-polynomial-growth}
Suppose that Assumption~\ref{ass:closed-rqmc-smoothness} holds. Then, for every $k=0,\ldots,K$ and every $v\subseteq1:s$, we have
\begin{equation}
    \partial_y^v\Big[\widetilde\xi_{h,k}\circ\Phi_{\rm N}\Big](y)\in C_p\big(\mathbb R^s;\mathbb R^m\big).
    \label{eq:path-mixed-derivatives-Cp}
\end{equation}
Consequently, there exist constants $M_h,B_h,\kappa_h>0$ such that
\begin{equation}
    \widetilde G_h^\cl\in G_p\big(M_h,B_h,\kappa_h\big).
    \label{eq:closed-rqmc-polynomial-growth}
\end{equation}
For the regularity analysis, $\widetilde\xi_{h,k}$ denotes the unwrapped coordinate path on the parameter covering space. This does not change the corresponding manifold path or $G_h^\cl$.
\end{lemma}

\begin{proof}
Let $z=\Phi_{\rm N}(y)$ for $y=(y_1,\ldots,y_K)\in\mathbb R^s$, where $y_k\in\mathbb R^{r-1}$. Then the path recursion can be written as
\begin{equation}
    \widetilde\xi_{h,0}\big(\Phi_{\rm N}(y)\big)=\xi_0,
    \quad
    \widetilde\xi_{h,k}\big(\Phi_{\rm N}(y)\big)=\Psi_h\Big(\widetilde\xi_{h,k-1}\big(\Phi_{\rm N}(y)\big),\mathcal D\big(\Phi_{\rm N}(y_k)\big)\Big),
    \label{eq:gaussian-transformed-path-recursion}
\end{equation}
for $k=1,\ldots,K$.

Recall that $V_h(\xi,Z)=h\,b(\xi)+\sqrt{rh}\,\Sigma(\xi)Z$, and
\begin{equation*}
    \big(\Psi_h(\xi,Z)\big)^\ell=\xi^\ell+V_h^\ell(\xi,Z)-\frac12\sum_{i,j=1}^m\Gamma_{ij}^{\ell}(\xi)V_h^i(\xi,Z)V_h^j(\xi,Z),
    \quad
    \ell=1,\ldots,m.
\end{equation*}
Since $V_h$ is at most linear in $Z$, Assumption~\ref{ass:closed-rqmc-smoothness} implies that, for all multi-indices $\alpha\in\mathbb N_0^m$ and $\beta\in\mathbb N_0^r$ with $|\alpha|+|\beta|\le s$, every component of $\partial_\xi^\alpha\partial_Z^\beta V_h$ belongs to $C_p\big(\mathbb R^{m+r}\big)$, and the derivative vanishes whenever $|\beta|\ge2$. For the quadratic term in $\Psi_h$, the multivariate Leibniz rule gives
\begin{align*}
&\partial_\xi^\alpha\partial_Z^\beta\Big(\Gamma_{ij}^{\ell}(\xi)V_h^i(\xi,Z)V_h^j(\xi,Z)\Big)\\
&\quad=\sum_{\substack{\alpha_0+\alpha_1+\alpha_2=\alpha\\\beta_1+\beta_2=\beta}}C_{\alpha_0,\alpha_1,\alpha_2,\beta_1,\beta_2}\,\partial_\xi^{\alpha_0}\Gamma_{ij}^{\ell}(\xi)\,\partial_\xi^{\alpha_1}\partial_Z^{\beta_1}V_h^i(\xi,Z)\,\partial_\xi^{\alpha_2}\partial_Z^{\beta_2}V_h^j(\xi,Z),
\end{align*}
here the constants $C_{\alpha_0,\alpha_1,\alpha_2,\beta_1,\beta_2}$ are the corresponding multinomial coefficients. Every factor on the right-hand side has polynomial growth by Assumption~\ref{ass:closed-rqmc-smoothness} and the preceding polynomial growth property of the derivatives of $V_h$. Since $C_p$ is closed under finite sums and products, every component of $\partial_\xi^\alpha\partial_Z^\beta\Psi_h$ belongs to $C_p\big(\mathbb R^{m+r}\big)$ whenever $|\alpha|+|\beta|\le s$. Equivalently, the norm of each Fr\'echet derivative $D_{(\xi,Z)}^j\Psi_h$, $0\le j\le s$, has polynomial growth in $(\xi,Z)$.

We prove \eqref{eq:path-mixed-derivatives-Cp} by induction over $k$. For $k=0$, the function $\widetilde\xi_{h,0}\big(\Phi_{\rm N}(y)\big)=\xi_0$ is constant. Hence it belongs to $C_p(\mathbb R^s;\mathbb R^m)$, and all its nonzero-order mixed derivatives vanish. Fix $k\in\{1,\ldots,K\}$ and suppose that $\partial_y^w\big[\widetilde\xi_{h,k-1}\circ\Phi_{\rm N}\big]\in C_p\big(\mathbb R^s;\mathbb R^m\big)$ for every $w\subseteq1:s$. By Lemma~\ref{lem:gaussian-sphere-transform}, any mixed derivative of $\mathcal D\big(\Phi_{\rm N}(y_k)\big)$ is bounded. A derivative with respect to an input variable outside the $k$-th block is understood to be zero. For $v=\varnothing$, equations \eqref{eq:gaussian-transformed-path-recursion} and the preceding polynomial growth property of $\Psi_h$, together with the induction hypothesis and the boundedness of $\mathcal D\big(\Phi_{\rm N}(y_k)\big)$, yield $\widetilde\xi_{h,k}\big(\Phi_{\rm N}(y)\big)\in C_p\big(\mathbb R^s;\mathbb R^m\big)$. 

Let $v\subseteq 1:s$ be nonempty, and let $\Pi(v)$ denote the set of all partitions of $v$ into nonempty subsets. Applying the multivariate Faà di Bruno formula with respect to $y$ to the right-hand side of \eqref{eq:gaussian-transformed-path-recursion}, we obtain
\begin{align}\label{eq:path-faa-di-bruno-expansion}   
&\partial_y^v
\Big[
    \widetilde\xi_{h,k}\circ\Phi_{\rm N}
\Big](y)
\\
&\quad=
\sum_{w\in\Pi(v)}
D^{|w|}_{(\xi,Z)}\Psi_h
\Big(
    \widetilde\xi_{h,k-1}\big(\Phi_{\rm N}(y)\big),
    \mathcal D\big(\Phi_{\rm N}(y_k)\big)
\Big)
\Bigg[
    \partial_y^\lambda
    \begin{pmatrix}
        \widetilde\xi_{h,k-1}\big(\Phi_{\rm N}(y)\big)\\
        \mathcal D\big(\Phi_{\rm N}(y_k)\big)
    \end{pmatrix}
\Bigg]_{\lambda\in w}.\nonumber
\end{align}
Here, $D^{|w|}_{(\xi,Z)}\Psi_h$ is the $|w|$-th Fr\'echet derivative of $\Psi_h$ with respect to its state and direction arguments, and the bracket means that this multilinear map is evaluated at the vectors indexed by the blocks $\lambda\in w$.

By the preceding polynomial-growth property of the derivatives of $\Psi_h$, the induction hypothesis, and Lemma~\ref{lem:gaussian-sphere-transform}, every factor in \eqref{eq:path-faa-di-bruno-expansion} has polynomial growth in $y$. Since $C_p$ is closed under finite sums and products, we conclude that
\begin{equation*}
    \partial_y^v\Big[\widetilde\xi_{h,k}\circ\Phi_{\rm N}\Big]\in C_p\big(\mathbb R^s;\mathbb R^m\big),
    \quad
    v\subseteq1:s.
\end{equation*}
Induction over $k$ proves \eqref{eq:path-mixed-derivatives-Cp} for every $k=0,\ldots,K$.

We next consider the closed-manifold integrand. By the definitions of $\overline u_0$, $\overline f$, and $\widetilde G_h^\cl$,
\begin{equation}
    \widetilde G_h^\cl(y)
    =
    \overline u_0\Big(
        \widetilde\xi_{h,K}\big(\Phi_{\rm N}(y)\big)
    \Big)
    +
    h\sum_{k=0}^{K-1}
    \overline f\Big(
        \widetilde\xi_{h,k}\big(\Phi_{\rm N}(y)\big),
        T-t_k
    \Big).
    \label{eq:closed-gaussian-integrand-path}
\end{equation}
For every $v\subseteq1:s$, the multivariate Fa\`a di Bruno formula expresses
$\partial_y^v\widetilde G_h^\cl$ as a finite sum of products involving derivatives of $\overline u_0$ or $\overline f$, evaluated along
$\widetilde\xi_{h,k}\circ\Phi_{\rm N}$, and mixed derivatives of
$\widetilde\xi_{h,k}\circ\Phi_{\rm N}$. Assumption~\ref{ass:closed-rqmc-smoothness} and \eqref{eq:path-mixed-derivatives-Cp} therefore give
\begin{equation*}
    \partial_y^v\widetilde G_h^\cl
    \in
    C_p\big(\mathbb R^s\big),
    \quad
    v\subseteq1:s.
\end{equation*}

Since there are only finitely many subsets $v\subseteq1:s$, the constants and polynomial exponents can be chosen uniformly over $v$. Hence there exist $M_h,B_h,\kappa_h>0$ such that
\begin{equation*}
    \sup_{v\subseteq1:s}
    \big|
        \partial_y^v\widetilde G_h^\cl(y)
    \big|
    \le
    M_h\|y\|_2^{\kappa_h}+B_h,
    \quad
    y\in\mathbb R^s.
\end{equation*}
The $C^s$ regularity in Assumption~\ref{ass:closed-rqmc-smoothness}, the continuity of the mixed derivatives in Lemma~\ref{lem:gaussian-sphere-transform}, and the preceding recursion show that all these mixed derivatives are continuous. Therefore, by the definition of $G_p$, we find that
\begin{equation*}
    \widetilde G_h^\cl
    \in
    G_p\big(M_h,B_h,\kappa_h\big),
\end{equation*}
which completes the proof.
\end{proof}

Under Lemma~\ref{lem:closed-rqmc-polynomial-growth}, we now obtain an explicit variance estimate for the closed-manifold direct RQMC estimator.

\begin{lemma}
\label{prop:closed-direct-rqmc-variance}
Suppose that Assumptions~\ref{ass:bounded-data} and \ref{ass:closed-rqmc-smoothness} hold. Let
$Q_n$ consist of the first $n=2^\ell$ points of a nested-uniformly scrambled Sobol' sequence in dimension $s$.
Let $\kappa_h>0$ be an exponent for which $\widetilde G_h^\cl\in G_p(M_h,B_h,\kappa_h)$, as established in Lemma~\ref{lem:closed-rqmc-polynomial-growth}. 
Then, for any $\varepsilon>0$ and all sufficiently large $n$, it holds that
\begin{equation}
\operatorname{Var}\big(\widehat u_{n,h}^{\RQMC,\cl}(x,T)\big)
\le C_{\RQMC,\cl,\varepsilon}n^{-2+\varepsilon},
\label{eq:closed-direct-rqmc-variance}
\end{equation}
where $C_{\RQMC,\cl,\varepsilon}$ is a positive constant independent of $n$.
\end{lemma}

\begin{proof}
Nested uniform scrambling preserves the digital-net property, so $Q_n$ satisfies the discrepancy condition in \cite[Corollary~4.8]{OuyangWangHe2024}. Moreover, Lemma~\ref{lem:closed-rqmc-polynomial-growth} gives $\widetilde G_h^\cl\in G_p\big(M_h,B_h,\kappa_h\big)$. Applying the cited result with integration dimension $s$ and
polynomial growth exponent $\kappa_h$ therefore yields
\begin{equation*}
    \mathbb E\Big[\big|\widehat u_{n,h}^{\RQMC,\cl}(x,T)-u_h^\cl(x,T)\big|^2\Big]\le C_{\RQMC,\cl}n^{-2}(\log n)^{3s+\kappa_h-2}.
\end{equation*}
Since $s$ and $\kappa_h$ are fixed, for any $\varepsilon>0$ the logarithmic factor can be absorbed into $n^\varepsilon$ for all sufficiently large $n$. Hence
\begin{equation*}
    \mathbb E\Big[\big|\widehat u_{n,h}^{\RQMC,\cl}(x,T)-u_h^\cl(x,T)\big|^2\Big]
    \le C_{\RQMC,\cl,\varepsilon}n^{-2+\varepsilon}.
\end{equation*}
By Lemma~\ref{lem:closed-direct-rqmc-general}, the estimator is
unbiased for $u_h^\cl(x,T)$, this ends the proof.
\end{proof}

We now turn to the Dirichlet problem. In contrast to the closed case, the first numerical exit changes the functional used to evaluate a path. Hence the fixed-$K$ RQMC integrand is generally discontinuous across interfaces that separate different first-exit events. We therefore decompose the stopped functional according to the first numerical exit step, which allows the contribution associated with each exit layer to be analyzed separately.

Using the same block decomposition $z=(z_1,\ldots,z_K)\in[0,1)^s$, $s=(r-1)K$, let $G_h^\bd:[0,1)^s\to\R$ denote the deterministic value of the stopped functional $S_h^\bd$ generated by the input $z$. Thus, $u_h^\bd(x,T)=\int_{[0,1)^s}G_h^\bd(z)\dd z$, and the direct RQMC estimator is
\begin{equation}
    \widehat u_{n,h}^{\RQMC,\bd}(x,T)=\frac1n\sum_{i=1}^nG_h^\bd\big(z^{(i)}\big).
    \label{eq:boundary-direct-rqmc-estimator}
\end{equation}
As in the closed case, if
$Q_n$ consists of the first $n=2^\ell$ points of a nested-uniformly scrambled Sobol' sequence in dimension $s$,
then each scrambled point $z^{(i)}$ is marginally uniform on $[0,1)^s$, and hence $\E\big[\widehat u_{n,h}^{\RQMC,\bd}(x,T)\big]=u_h^\bd(x,T)$. Moreover, Assumption~\ref{ass:bounded-data} gives $\big|G_h^\bd(z)\big|\le\max\{M_0,M_\psi\}+TM_f$ with $z\in[0,1)^s$, so that $G_h^\bd\in L^2([0,1)^s)$. Therefore, the same general scrambled-net variance result as in the closed case yields
\begin{equation}
    \operatorname{Var}\big(\widehat u_{n,h}^{\RQMC,\bd}(x,T)\big)=o\big(n^{-1}\big),
    \quad
    n\to\infty.
    \label{eq:boundary-direct-rqmc-general-rate}
\end{equation}
for every fixed $h$.


To obtain an explicit algebraic rate, we decompose $G_h^\bd$ according to the first-exit layer. For $k=1,\ldots,K$, set $s_k=(r-1)k$ and write $z^{[k]}=(z_1,\ldots,z_k)\in[0,1)^{s_k}$. For fixed $h$ and $x$, the prefix $z^{[k]}$ uniquely determines the p-Ret path up to layer $k$. Let $\mathcal A_{h,k}$ denote the set of input prefixes for which no numerical exit occurs during the first $k$ steps, with $\mathcal A_{h,0}=[0,1)^0$ and $\widetilde{\xi}_{h,0}^{\bd}=\xi_0$. Recursively, once $\mathcal A_{h,k-1}$ and the surviving path map $\widetilde{\xi}_{h,k-1}^{\bd}$ are defined, define the $k$-th crossing function for $z^{[k]}\in\mathcal A_{h,k-1}\times[0,1)^{r-1}$ by
\begin{equation}
    \rho_{h,k}\big(z^{[k]}\big)=\chi\Big(\Psi_h\big(\widetilde{\xi}_{h,k-1}^{\bd}\big(z^{[k-1]}\big),\mathcal D(z_k)\big)\Big).
    \label{eq:boundary-rqmc-crossing-function}
\end{equation}
The first-exit and survival sets admit the explicit representations
\begin{equation}
    \Theta_{h,k}=\Big\{z^{[k]}\in[0,1)^{s_k}:\rho_{h,j}\big(z^{[j]}\big)>0\text{ for }j=1,\ldots,k-1,\;\rho_{h,k}\big(z^{[k]}\big)\le0
    \Big\},
    \label{eq:boundary-rqmc-first-exit-set}
\end{equation}
and
\begin{equation}
    \mathcal A_{h,k}=\Big\{z^{[k]}\in[0,1)^{s_k}:\rho_{h,j}\big(z^{[j]}\big)>0\text{ for }j=1,\ldots,k\Big\}.
    \label{eq:boundary-rqmc-survival-set}
\end{equation}
For fixed $h$ and $x$, the functions $\rho_{h,j}$ and hence the sets $\Theta_{h,k}$ and $\mathcal A_{h,k}$ are deterministic; in particular, they are independent of the sample size $n$ and of the randomization of the RQMC point set. For $z^{[k]}\in\mathcal A_{h,k}$, define
\begin{equation}
    \widetilde{\xi}_{h,k}^\bd\big(z^{[k]}\big)=\Psi_h\big(\widetilde{\xi}_{h,k-1}^\bd\big(z^{[k-1]}\big),\mathcal D(z_k)\big),
    \quad
    X_{h,k}^\bd=\Phi\big(\widetilde{\xi}_{h,k}^\bd\big).
    \label{eq:boundary-rqmc-surviving-path}
\end{equation}
Thus, $\Theta_{h,k}$ consists of the input prefixes whose first numerical exit occurs during the $k$th p-Ret step, whereas $\mathcal A_{h,K}$ consists of the full inputs that survive until $T$.

For $z^{[k]}\in\Theta_{h,k}$, let $\lambda_{k-1}\big(z^{[k]}\big)\in(0,1]$ denote the boundary fraction defined by \eqref{eq:numerical-boundary-fraction} for the $k$th p-Ret step. The corresponding numerical exit time and boundary point are
\begin{equation*}
    \widehat\tau_h\big(z^{[k]}\big)=t_{k-1}+\lambda_{k-1}\big(z^{[k]}\big)h,
    \quad
    \widehat X_{\partial}\big(z^{[k]}\big)=\Phi\big(\gamma_{k-1}\big(\lambda_{k-1}(z^{[k]})\big)\big).
\end{equation*}
The contribution associated with a first numerical exit during the $k$-th p-Ret step is
\begin{equation}\begin{split}
H_{h,k}^\bd\big(z^{[k]}\big)=\;&\psi\big(\widehat X_{\partial}\big(z^{[k]}\big),T-\widehat\tau_h\big(z^{[k]}\big)\big)+h\sum_{j=0}^{k-2}f\big(X_{h,j}^\bd,T-t_j\big)\\
& +\lambda_{k-1}\big(z^{[k]}\big)hf\big(X_{h,k-1}^\bd,T-t_{k-1}\big),
\quad
z^{[k]}\in\Theta_{h,k}.
\label{eq:boundary-rqmc-exit-weight}
\end{split}\end{equation}
If no numerical exit occurs before $T$, the contribution is
\begin{equation}
    H_{h,K+1}^\bd(z)=u_0\big(X_{h,K}^\bd(z)\big)+h\sum_{j=0}^{K-1}f\big(X_{h,j}^\bd,T-t_j\big),
    \quad
    z\in\mathcal A_{h,K}.
    \label{eq:boundary-rqmc-no-exit-weight}
\end{equation}

The layer integrands are then given by
\begin{equation*}
    F_{h,k}^\bd=
    \begin{cases}
        H_{h,k}^\bd\,\1_{\Theta_{h,k}},& k=1,\ldots,K,\\[1mm]
        H_{h,K+1}^\bd\,\1_{\mathcal A_{h,K}},& k=K+1.
    \end{cases}
\end{equation*}
Then the stopped integrand admits the exact decomposition
\begin{equation}
    G_h^\bd(z)=\sum_{k=1}^KF_{h,k}^\bd\big(z^{[k]}\big)+F_{h,K+1}^\bd(z).
    \label{eq:boundary-rqmc-layer-decomposition}
\end{equation}

By the layer decomposition \eqref{eq:boundary-rqmc-layer-decomposition}, define
\begin{equation*}
    I_{h,k}^\bd=
    \begin{cases}
        \displaystyle\int_{[0,1)^{s_k}}F_{h,k}^\bd\big(z^{[k]}\big)\dd z^{[k]},& k=1,\ldots,K,\\[7pt]
        \displaystyle\int_{[0,1)^s}F_{h,K+1}^\bd(z)\dd z,& k=K+1.
    \end{cases}
\end{equation*}
Then we have
\begin{equation}
    u_h^\bd(x,T)
    =
    \int_{[0,1)^s}G_h^\bd(z)\dd z
    =
    \sum_{k=1}^{K+1} I_{h,k}^\bd.
    \label{eq:boundary-rqmc-layer-integrals}
\end{equation}

For each RQMC point $z^{(i)}$, write $z^{(i),[k]}=\big(z_1^{(i)},\ldots,z_k^{(i)}\big)$. The corresponding layer estimators are
\begin{equation*}
    \widehat I_{n,h,k}^\bd=\begin{cases}
        \displaystyle\frac1n\sum_{i=1}^nF_{h,k}^\bd\big(z^{(i),[k]}\big),& k=1,\ldots,K,\\[8pt]
        \displaystyle\frac1n\sum_{i=1}^nF_{h,K+1}^\bd\big(z^{(i)}\big),& k=K+1.
    \end{cases}
\end{equation*}
Hence,
\begin{equation}
    \widehat u_{n,h}^{\RQMC,\bd}(x,T)=\sum_{k=1}^{K+1}\widehat I_{n,h,k}^\bd.
    \label{eq:boundary-rqmc-layer-estimator}
\end{equation}





We next establish the geometric regularity of the discontinuity interfaces generated by the numerical first-exit mechanism. For a bounded set $E\subset\mathbb R^q$, define its upper codimension-one Minkowski content by
\begin{equation*}
    \overline{\mathbb M}_{q-1}(E):=\limsup_{\varepsilon\downarrow0}\frac{\operatorname{vol}_q\big(\{z\in\mathbb R^q:\operatorname{dist}(z,E)<\varepsilon\}\big)}{2\varepsilon}.
\end{equation*}

To control the geometry of the discontinuity sets, we introduce the following transversality and finite complexity condition.

\begin{assumption}
\label{ass:boundary-rqmc-transversality}
For every $k=1,\ldots,K$, set $s_k=(r-1)k$ and $Q_{s_k}=[0,1]^{s_k}$, and define the effective interior input domain $D_{h,k}:=\big(\mathcal A_{h,k-1}\times(0,1)^{r-1}\big)\cap(0,1)^{s_k}$ and the interior switching interface
\begin{equation*}
    \Gamma_{h,k}^{\circ}:=\Big\{z^{[k]}\in D_{h,k}:\rho_{h,k}\big(z^{[k]}\big)=0\Big\}.
\end{equation*}
The crossing function $\rho_{h,k}$ defined in \eqref{eq:boundary-rqmc-crossing-function} is continuous on $D_{h,k}$ and admits a $C^1$ extension to an open neighborhood of $\Gamma_{h,k}^{\circ}$ in $(0,1)^{s_k}$, still denoted by $\rho_{h,k}$. The interface is regular in the sense that
\begin{equation}
    \nabla\rho_{h,k}\big(z^{[k]}\big)\neq0,
    \quad
    z^{[k]}\in\Gamma_{h,k}^{\circ}.
    \label{eq:boundary-rqmc-transversality}
\end{equation}
Moreover, let $\Gamma_{h,k}:=\overline{\Gamma_{h,k}^{\circ}}^{\,Q_{s_k}}$ be the relative closure of the interior switching interface in $Q_{s_k}$. We assume that $\Gamma_{h,k}$ admits a finite $C^1$ graph cover. More precisely, there exists $N_{h,k}<\infty$ such that, for every $\nu=1,\ldots,N_{h,k}$, there are a coordinate permutation $P_{h,k,\nu}$, a bounded open set $V_{h,k,\nu}\subset\mathbb R^{s_k-1}$, a compact set $K_{h,k,\nu}\subset V_{h,k,\nu}$, and a function $g_{h,k,\nu}\in C^1(V_{h,k,\nu})$ satisfying
\begin{equation*}
    \Gamma_{h,k}\subset\bigcup_{\nu=1}^{N_{h,k}}P_{h,k,\nu}^{-1}\big\{\big(x,g_{h,k,\nu}(x)\big):x\in K_{h,k,\nu}\big\}.
\end{equation*}
\end{assumption}

The continuity condition identifies every interior boundary created by the sign test with a zero of the crossing function. The regularity condition gives the local hypersurface structure of the interior switching interface, while the finite graph cover controls its possible accumulation at the boundary of the RQMC cube. Together these properties yield codimension-one tubular-neighborhood estimates, which extend recursively to the boundaries of the first-exit and survival sets. The following lemma records these geometric consequences.

\begin{lemma}
\label{lem:boundary-rqmc-minkowski}
Suppose that Assumption~\ref{ass:boundary-rqmc-transversality} holds. Then, for every $k=1,\ldots,K$, the boundaries $\partial\Theta_{h,k}$ and $\partial\mathcal A_{h,k}$ have finite upper $(s_k-1)$-dimensional Minkowski content. 
More precisely, for all sufficiently small $\varepsilon>0$, it holds that
\begin{equation}
    \operatorname{vol}_{s_k}\Big(\big\{z\in\R^{s_k}:\operatorname{dist}(z,\partial E)<\varepsilon\big\}\Big)\le C_{h,k}\varepsilon,
    \quad
    E\in\big\{\Theta_{h,k},\mathcal A_{h,k}\big\},
    \label{eq:boundary-rqmc-exit-tube}
\end{equation}
where $C_{h,k}$ is a positive constant independent of $\varepsilon$.
In particular, $\partial\mathcal A_{h,K}$ has finite upper $(s-1)$-dimensional Minkowski content.
\end{lemma}

\begin{proof}
Fix $k\in\{1,\ldots,K\}$ and set $Q_q=[0,1]^q$ for $q\ge1$, with the convention that $Q_0$ consists of a single point. Notice first that
\begin{equation*}
    \operatorname{vol}_q\Big(\big\{z\in\R^q:\operatorname{dist}(z,\partial Q_q)<\varepsilon\big\}\Big)=O(\varepsilon),
    \quad
    \varepsilon\downarrow0.
\end{equation*}

Fix $j\in\{1,\ldots,k\}$. By Assumption~\ref{ass:boundary-rqmc-transversality}, $\Gamma_{h,j}$ is covered by finitely many coordinate permutations of graphs $G_{h,j,\nu}:=\big\{\big(x,g_{h,j,\nu}(x)\big):x\in K_{h,j,\nu}\big\}$. Coordinate permutations preserve Euclidean distance and Lebesgue measure, so it suffices to estimate the tubular neighborhood of $G_{h,j,\nu}$. Since $K_{h,j,\nu}$ is compactly contained in $V_{h,j,\nu}$, there exists $\delta_{h,j,\nu}>0$ such that the closed $\delta_{h,j,\nu}$-neighborhood of $K_{h,j,\nu}$ is contained in $V_{h,j,\nu}$. Because $g_{h,j,\nu}\in C^1(V_{h,j,\nu})$, there is a constant $L_{h,j,\nu}<\infty$ such that $\big\|\nabla g_{h,j,\nu}(x)\big\|\le L_{h,j,\nu}$ throughout this closed neighborhood.

For $0<\varepsilon<\delta_{h,j,\nu}$, let $(K_{h,j,\nu})_\varepsilon:=\big\{y\in\mathbb R^{s_j-1}:\operatorname{dist}(y,K_{h,j,\nu})<\varepsilon\big\}$. If $z=(z',z_{s_j})\in(G_{h,j,\nu})_\varepsilon$, then there exists $x\in K_{h,j,\nu}$ such that $|z'-x|<\varepsilon$, and $\big|z_{s_j}-g_{h,j,\nu}(x)\big|<\varepsilon$. Hence $z'\in(K_{h,j,\nu})_\varepsilon$. Moreover, the mean value theorem gives
\begin{equation*}
\begin{split}
    \big|z_{s_j}-g_{h,j,\nu}(z')\big|&\le\big|z_{s_j}-g_{h,j,\nu}(x)\big|+\big|g_{h,j,\nu}(x)-g_{h,j,\nu}(z')\big|\\
    &<(1+L_{h,j,\nu})\varepsilon.
\end{split}
\end{equation*}
Therefore,
\begin{equation*}
    (G_{h,j,\nu})_\varepsilon\subset\Big\{(z',z_{s_j}):z'\in(K_{h,j,\nu})_\varepsilon,\ \big|z_{s_j}-g_{h,j,\nu}(z')\big|<(1+L_{h,j,\nu})\varepsilon\Big\}.
\end{equation*}
Consequently, by Fubini's theorem,
\begin{equation*}
    \operatorname{vol}_{s_j}\big((G_{h,j,\nu})_\varepsilon\big)\le2(1+L_{h,j,\nu})\varepsilon\,\operatorname{vol}_{s_j-1}\big((K_{h,j,\nu})_\varepsilon\big)\le C_{h,j,\nu}\varepsilon,
\end{equation*}
where the last inequality follows because the $\varepsilon$-neighborhoods of the bounded set $K_{h,j,\nu}$ have uniformly bounded volume for sufficiently small $\varepsilon$. Summing over the finite graph cover gives
\begin{equation}
    \operatorname{vol}_{s_j}\Big(\big\{z\in\R^{s_j}:\operatorname{dist}(z,\Gamma_{h,j})<\varepsilon\big\}\Big)\le C_{h,j}^{\Gamma}\varepsilon.
    \label{eq:boundary-rqmc-crossing-tube}
\end{equation}

For $1\le j\le k$, lift the $j$-th interface to the $s_k$-dimensional input space as $\Gamma_{h,j}\times Q_{s_k-s_j}$. For a set $E\subset\R^q$ and $\varepsilon>0$, let $E_\varepsilon=\{z:\operatorname{dist}(z,E)<\varepsilon\}$. Then,
\begin{equation*}
    \big(\Gamma_{h,j}\times Q_{s_k-s_j}\big)_\varepsilon\subset\big(\Gamma_{h,j}\big)_\varepsilon\times\big(Q_{s_k-s_j}\big)_\varepsilon.
\end{equation*}
Since $\operatorname{vol}_{s_k-s_j}\big((Q_{s_k-s_j})_\varepsilon\big)\le(1+2\varepsilon)^{s_k-s_j}$, \eqref{eq:boundary-rqmc-crossing-tube} gives
\begin{equation*}
    \operatorname{vol}_{s_k}\Big(\big(\Gamma_{h,j}\times Q_{s_k-s_j}\big)_\varepsilon\Big)\le C_{h,j}^{\Gamma}\varepsilon\big(1+2\varepsilon\big)^{s_k-s_j}\le C_{h,j,k}\varepsilon
\end{equation*}
for sufficiently small $\varepsilon$.

It remains to identify the possible boundaries of the first-exit and survival sets. Set $B_{h,k}:=\mathcal A_{h,k-1}\times[0,1)^{r-1}$. By \eqref{eq:boundary-rqmc-first-exit-set} and \eqref{eq:boundary-rqmc-survival-set}, both $\mathcal A_{h,k}$ and $\Theta_{h,k}$ are subsets of $B_{h,k}$. Using the elementary inclusion $\partial(E\times F)\subset\big(\partial E\times\overline F\big)\cup\big(\overline E\times\partial F\big)$, together with $\overline{[0,1)^{r-1}}=Q_{r-1}$ and $\overline{\mathcal A}{h,k-1}\times\partial Q{r-1}\subset\partial Q_{s_k}$, we obtain
\begin{equation*}
    \partial B_{h,k}
    \subset
    \big(\partial\mathcal A_{h,k-1}\times Q_{r-1}\big)
    \cup
    \partial Q_{s_k}.
\end{equation*}

Let $z\in\big(\partial\mathcal A_{h,k}\cup\partial\Theta_{h,k}\big)\setminus\Big[\big(\partial\mathcal A_{h,k-1}\times Q_{r-1}\big)\cup\partial Q_{s_k}\Big]$. Since $z$ is a boundary point of a subset of $B_{h,k}$ and $z\notin\partial B_{h,k}$, it follows that $z$ is an interior point of $B_{h,k}$. Since also $z\notin\partial Q_{s_k}$, we have $z\in D_{h,k}$. If $\rho_{h,k}(z)>0$, continuity of $\rho_{h,k}$ on $D_{h,k}$ would make $z$ an interior point of $\mathcal A_{h,k}$; if $\rho_{h,k}(z)<0$, it would make $z$ an interior point of $\Theta_{h,k}$. Either conclusion contradicts the choice of $z$. Therefore, $\rho_{h,k}(z)=0$, and hence $z\in\Gamma_{h,k}^{\circ}\subset\Gamma_{h,k}$. We have thus proved
\begin{equation*}
    \partial\mathcal A_{h,k}\cup\partial\Theta_{h,k}\subset\big(\partial\mathcal A_{h,k-1}\times Q_{r-1}\big)\cup\Gamma_{h,k}\cup\partial Q_{s_k}.
\end{equation*}

Repeatedly applying this inclusion to $\partial\mathcal A_{h,k-1},\ldots,\partial\mathcal A_{h,1}$, with $\partial\mathcal A_{h,0}=\varnothing$, and lifting each switching interface to the $s_k$-dimensional input space yields

\begin{equation}
    \partial\mathcal A_{h,k}\cup\partial\Theta_{h,k}\subset\partial Q_{s_k}\cup\bigcup_{j=1}^k\Big(\Gamma_{h,j}\times Q_{s_k-s_j}\Big).
    \label{eq:boundary-rqmc-interface-cover}
\end{equation}
Taking $\varepsilon$-neighborhoods in \eqref{eq:boundary-rqmc-interface-cover}, using finite subadditivity of Lebesgue measure, and applying the preceding tube estimates yields
\begin{equation*}
    \operatorname{vol}_{s_k}\Big(\big(\partial E\big)_\varepsilon\Big)\le C_{h,k}\varepsilon,
    \quad
    E\in\big\{\Theta_{h,k},\mathcal A_{h,k}\big\}.
\end{equation*}

This proves \eqref{eq:boundary-rqmc-exit-tube}, and hence, by the definition of the upper Minkowski content, $\overline{\mathbb M}_{s_k-1}(\partial E)<\infty$ for $E\in\{\Theta_{h,k},\mathcal A_{h,k}\}$.
\end{proof}

The preceding lemma controls the geometry of the discontinuity sets. To control the singular behavior of the layer weights near the boundary of the unit cube, we use the Gaussian coordinates introduced in \eqref{eq:closed-rqmc-gaussian-integrand} and impose the following polynomial-growth condition.

\begin{assumption}
\label{ass:boundary-rqmc-weight-regularity}
Fix $h$ and let $s=(r-1)K$. Assume that the conditions of Assumption~\ref{ass:closed-rqmc-smoothness} hold for the Dirichlet problem and that the coordinate boundary data $\overline\psi$ admits a $C^s$ extension to the parameter covering space times $\mathbb R$, whose derivatives up to order $s$ have polynomial growth.

For every $k=1,\ldots,K$, the Gaussian-transformed boundary fraction admits an extension $\widetilde\lambda_{h,k-1}\in G_p(M_{h,k},B_{h,k},\kappa_{h,k})$ that agrees with $\lambda_{k-1}\circ\Phi_{\rm N}$ on $\Phi_{\rm N}^{-1}\big(\Theta_{h,k}\cap(0,1)^{s_k}\big)$.
\end{assumption}

For $t\in(0,1)$, set $\phi(t)=\min\{t,1-t\}$. A function $F\in L^2([0,1]^q)$ is said to satisfy Owen's boundary growth condition (cf.~\cite{Owen2006Halton}) if its mixed derivatives $\partial_z^vF$ exist on $(0,1)^q$ and there exist exponents $A_1,\ldots,A_q<1/2$ such that, for every $v\subseteq 1:q$, we have
\begin{equation*}
    \big|\partial_z^vF(z)\big|
    \le
    C_v\prod_{j=1}^q\phi(z_j)^{-A_j-\1_{\{j\in v\}}},
    \quad z\in(0,1)^q,
\end{equation*}
where $C_v$ is independent of $z$.

In the present setting, we take $A_1=\cdots=A_q=A$ for an arbitrary $A\in(0,1/2)$. The following lemma shows that the smooth extension of every layer weight belongs to $L^2([0,1]^q)$ and satisfies Owen's boundary growth condition.

\begin{lemma}
\label{lem:boundary-rqmc-boundary-growth}
Suppose that Assumption~\ref{ass:boundary-rqmc-weight-regularity} holds. For every $k=1,\ldots,K$, there exists $H_{h,k}^{\bd,\mathrm{ext}}\in L^2\big((0,1)^{s_k}\big)$ which agrees with $H_{h,k}^\bd$ on $\Theta_{h,k}\cap(0,1)^{s_k}$ and, for each $A\in(0,1/2)$ and every $v\subseteq1:s_k$, satisfies
\begin{equation}
\big|\partial_z^v H_{h,k}^{\bd,\mathrm{ext}}(z)\big|\le C_{h,k,A}\prod_{j=1}^{s_k}\phi(z_j)^{-A-\1_{\{j\in v\}}},
\quad
z\in(0,1)^{s_k}.
\label{eq:boundary-rqmc-boundary-growth}
\end{equation}
Moreover, there exists $H_{h,K+1}^{\bd,\mathrm{ext}}\in L^2\big((0,1)^s\big)$ which agrees with $H_{h,K+1}^\bd$ on $\mathcal A_{h,K}\cap(0,1)^s$ and satisfies \eqref{eq:boundary-rqmc-boundary-growth} with $s_k$ replaced by $s$.
\end{lemma}

\begin{proof}
Fix $k\in\{1,\ldots,K\}$. Define $\widetilde H_{h,k}^\bd$ on $\mathbb R^{s_k}$ by evaluating the right-hand side of \eqref{eq:boundary-rqmc-exit-weight} along the unrestricted Gaussian-coordinate path and replacing the boundary fraction by $\widetilde\lambda_{h,k-1}$. By construction,
\begin{equation*}
\widetilde H_{h,k}^\bd(y)=H_{h,k}^\bd\big(\Phi_{\rm N}(y)\big),
\quad
y\in\Phi_{\rm N}^{-1}\big(\Theta_{h,k}\cap(0,1)^{s_k}\big).
\end{equation*}
By Lemma~\ref{lem:closed-rqmc-polynomial-growth}, Assumption~\ref{ass:boundary-rqmc-weight-regularity}, and repeated applications of the product rule and the multivariate Fa\`a di Bruno formula, there exist $M,B,\kappa>0$ such that
\begin{equation}
\sup_{v\subseteq1:s_k}
\big|\partial_y^v\widetilde H_{h,k}^\bd(y)\big|
\le
M\|y\|_2^\kappa+B,
\quad
y\in\mathbb R^{s_k}.
\label{eq:boundary-rqmc-gaussian-weight-growth}
\end{equation}
Indeed, the unrestricted path and all of its mixed derivatives have polynomial growth by \eqref{eq:path-mixed-derivatives-Cp}; the transformed direction map and its mixed derivatives are bounded by Lemma~\ref{lem:gaussian-sphere-transform}; and the remaining factors are the polynomial-growth extensions of $\overline f$, $\overline\psi$, and $\widetilde\lambda_{h,k-1}$. Thus every term produced by differentiating the finite expression \eqref{eq:boundary-rqmc-exit-weight} has polynomial growth.

Define
\begin{equation*}
H_{h,k}^{\bd,\mathrm{ext}}(z)
=
\widetilde H_{h,k}^\bd\big(\Phi_{\rm N}^{-1}(z)\big),
\quad
z\in(0,1)^{s_k},
\end{equation*}
where $\Phi_{\rm N}^{-1}$ acts componentwise. Let $y=\Phi_{\rm N}^{-1}(z)$ and write $\varphi_{\rm N}=\Phi_{\rm N}'$ for the standard normal density. Since each coordinate is differentiated at most once,
\begin{equation*}
\partial_z^v H_{h,k}^{\bd,\mathrm{ext}}(z)
=
\partial_y^v\widetilde H_{h,k}^\bd(y)
\prod_{j\in v}\frac{1}{\varphi_{\rm N}(y_j)}.
\end{equation*}
Standard Gaussian tail estimates imply, for every $A>0$,
\begin{equation*}
1+\|\Phi_{\rm N}^{-1}(z)\|_2^\kappa
\le
C_{A,\kappa,s_k}
\prod_{j=1}^{s_k}\phi(z_j)^{-A},
\quad
\frac{1}{\varphi_{\rm N}\big(\Phi_{\rm N}^{-1}(t)\big)}
\le
\frac{C}{\phi(t)}.
\end{equation*}
These estimates follow from the Gaussian quantile asymptotics and the two-sided Mills ratio. Combining these estimates with \eqref{eq:boundary-rqmc-gaussian-weight-growth} gives
\begin{equation*}
\big|\partial_z^v H_{h,k}^{\bd,\mathrm{ext}}(z)\big|
\le
C_{h,k,A}
\prod_{j=1}^{s_k}\phi(z_j)^{-A}
\prod_{j\in v}\phi(z_j)^{-1}
=
C_{h,k,A}
\prod_{j=1}^{s_k}\phi(z_j)^{-A-\1_{\{j\in v\}}},
\end{equation*}
which proves \eqref{eq:boundary-rqmc-boundary-growth}.

For the no-exit layer, use the unrestricted Gaussian path in \eqref{eq:boundary-rqmc-no-exit-weight}. The same construction and estimates, now without a boundary fraction and with $s$ in place of $s_k$, give the asserted extension $H_{h,K+1}^{\bd,\mathrm{ext}}$. Finally, taking $v=\varnothing$ in the corresponding boundary growth estimates gives
\begin{equation*}
\int_{(0,1)^{s_k}}
\big|H_{h,k}^{\bd,\mathrm{ext}}(z)\big|^2\dd z
\le
C_{h,k,A}
\prod_{j=1}^{s_k}
\int_0^1\phi(t)^{-2A}\dd t
<\infty,
\end{equation*}
because $2A<1$. The same argument applies to the no-exit extension.
\end{proof}

We can now combine the analytic boundary growth in Lemma~\ref{lem:boundary-rqmc-boundary-growth} with the geometric tube estimate in Lemma~\ref{lem:boundary-rqmc-minkowski}.

\begin{lemma}
\label{lem:boundary-direct-rqmc-variance}
Suppose that Assumptions~\ref{ass:bounded-data}, \ref{ass:boundary-rqmc-transversality}, and \ref{ass:boundary-rqmc-weight-regularity} hold. Let $Q_n$ consist of the first $n=2^\ell$ points of a nested-uniformly scrambled Sobol' sequence in dimension $s=(r-1)K$. Then, for fixed $h$ and every $\varepsilon>0$, 
it holds that
\begin{equation}
\operatorname{Var}\big(\widehat u_{n,h}^{\mathrm{RQMC},\partial}(x,T)\big)\le C_{\mathrm{RQMC},\partial,\varepsilon}n^{-1-1/s+\varepsilon},
\label{eq:boundary-direct-rqmc-variance}
\end{equation}
where $C_{\mathrm{RQMC},\partial,\varepsilon}$ is a positive constant independent of $n$.
\end{lemma}

\begin{proof}
For each $k=1,\ldots,K$, the ingredients needed in the proof of \cite[Lemma~4.5]{Liu2026BoundaryGrowth} are available: \emph{(i)} $F_{h,k}^\bd=H_{h,k}^{\bd,\mathrm{ext}}\1_{\Theta_{h,k}}$ almost everywhere, where $H_{h,k}^{\bd,\mathrm{ext}}$ belongs to $L^2$ and satisfies Owen's boundary growth condition with common exponent $A$ by Lemma~\ref{lem:boundary-rqmc-boundary-growth}; \emph{(ii)} $\partial\Theta_{h,k}$ has finite upper $(s_k-1)$-dimensional Minkowski content by Lemma~\ref{lem:boundary-rqmc-minkowski}; and \emph{(iii)} the projection of $Q_n$ onto its first $s_k$ coordinates remains a nested-uniformly scrambled Sobol' point set. Although the cited result assumes the existence of the Minkowski content, its proof only uses the corresponding upper tube estimate, which follows from \emph{(ii)}. Hence the same argument gives, for every $A\in(0,1/2)$ and $\delta\in(0,1-2A)$,
\begin{equation*}
\operatorname{Var}\big(\widehat I_{n,h,k}^\bd\big)\le C_{h,k,A,\delta}
n^{-1-\frac{(1-2A)(1-2A-\delta)}{s_k-2A-\delta}}(\log n)^{\frac{(s_k-1)(1-2A-\delta)}{s_k-2A-\delta}},
\end{equation*}
for every exit layer. The same argument applies in dimension $s$ to $H_{h,K+1}^{\bd,\mathrm{ext}}\1_{\mathcal A_{h,K}}$.

Since $s_k\le s$, the bound in dimension $s$ gives the slowest rate among all layers. Therefore, by the layer decomposition and the Cauchy--Schwarz inequality,
\begin{equation*}
\operatorname{Var}\big(\widehat u_{n,h}^{\RQMC,\bd}\big)\le C_{h,A,\delta}n^{-1-\frac{(1-2A)(1-2A-\delta)}{s-2A-\delta}}(\log n)^{\frac{(s-1)(1-2A-\delta)}{s-2A-\delta}}.
\end{equation*}
Since $A$ and $\delta$ can be chosen arbitrarily small and any fixed power of $\log n$ can be absorbed into $n^\varepsilon$, it follows that, for any $\varepsilon>0$,
\begin{equation*}
\operatorname{Var}\big(\widehat u_{n,h}^{\RQMC,\bd}\big)=O\big(n^{-1-1/s+\varepsilon}\big),
\end{equation*}
which proves \eqref{eq:boundary-direct-rqmc-variance}.
\end{proof}

Combining the temporal discretization error with the preceding RQMC sampling error estimates yields the following theorem.

\begin{theorem}
\label{thm:direct-rqmc-total-error}
Let $x\in\M$ be fixed and let $h=T/K$ with $s=(r-1)K$.

\emph{(i) \underline{Closed manifold.}} Suppose that the assumptions of Lemma~\ref{lem:weak-temporal-convergence} and Lemma~\ref{prop:closed-direct-rqmc-variance} hold. Let $\kappa_h>0$ be an exponent for which $\widetilde G_h^\cl\in G_p(M_h,B_h,\kappa_h)$, as established in Lemma~\ref{lem:closed-rqmc-polynomial-growth}. 
Then, for any $\varepsilon>0$, it holds that
\begin{equation*}
\mathbb E\Big[\big|\widehat u_{n,h}^{\RQMC,\cl}(x,T)-u^\cl(x,T)\big|^2\Big]
\le Ch^2+C_{\RQMC,\cl,\varepsilon}n^{-2+\varepsilon},
\end{equation*}
where $C$ is a positive constant independent of $h$, and
$C_{\RQMC,\cl,\varepsilon}$
is a positive constant independent of $n$.


\emph{(ii) \underline{Manifold with Dirichlet boundary.}} Suppose that Assumptions~\ref{ass:bounded-data}, \ref{ass:boundary-rqmc-transversality}, and \ref{ass:boundary-rqmc-weight-regularity} hold, suppose that the temporal weak error satisfies \eqref{eq:dirichlet-temporal-error}, and let $Q_n$ be the scrambled Sobol' point set in Lemma~\ref{lem:boundary-direct-rqmc-variance}. Then, for any $\varepsilon>0$, 
it holds that
\begin{equation*}
\mathbb E\Big[\big|\widehat u_{n,h}^{\RQMC,\bd}(x,T)-u^\bd(x,T)\big|^2\Big]\le C\big(h+\mathcal E_{\partial}\big)^2+C_{\RQMC,\bd,\varepsilon}n^{-1-1/s+\varepsilon},
\end{equation*}
where $C$ is a positive constant independent of $h$, and $C_{\RQMC,\bd,\varepsilon}$ is a positive constant independent of $n$.
In particular, if $\mathcal E_{\partial}\le Ch$, then
\begin{equation*}
\mathbb E\Big[\big|\widehat u_{n,h}^{\RQMC,\bd}(x,T)-u^\bd(x,T)\big|^2\Big]\le Ch^2+C_{\RQMC,\bd,\varepsilon}n^{-1-1/s+\varepsilon}.
\end{equation*}
\end{theorem}

\begin{proof}
The result follows immediately from the unbiasedness $\mathbb E[\widehat u_{n,h}^{\RQMC,\star}]=u_h^\star$, the identity $\mathbb E[|\widehat u_{n,h}^{\RQMC,\star}-u^\star|^2]=\operatorname{Var}(\widehat u_{n,h}^{\RQMC,\star})+|u_h^\star-u^\star|^2$, Lemma~\ref{lem:weak-temporal-convergence}, \eqref{eq:dirichlet-temporal-error}, and Lemmas~\ref{prop:closed-direct-rqmc-variance} and \ref{lem:boundary-direct-rqmc-variance}.
\end{proof}

\section{Array-RQMC sampling and error estimates}
\label{sec:array-rqmc}

As an alternative to direct path-space RQMC, we consider Array-RQMC for the p-Ret discretization. Since the direct formulation has nominal dimension $s=(r-1)K$, its effectiveness may deteriorate as $K$ or $r$ increases. Array-RQMC instead applies independently randomized $r$-dimensional point sets sequentially over the time layers, with particle reordering before each update, thereby avoiding a single full path-space randomization and motivating a layerwise error analysis.

\subsection{Stepwise Array-RQMC construction and Hilbert rank matching}
\label{subsec:array-construction}
We now describe the stepwise Array-RQMC construction and the associated Hilbert rank matching procedure. The construction follows the Array-RQMC framework for Markov chains in \cite{LEcuyerLecotTuffin2008} and is closely related to sequential QMC \cite{GerberChopin2015}.
Let $H_m:[0,1]\to[0,1]^m$ be an $m$-dimensional Hilbert space-filling curve and let $H_m^{-1}:[0,1]^m\to[0,1]$ be a fixed measurable pseudo-inverse satisfying $H_m\big(H_m^{-1}(x)\big)=x$ for every $x\in[0,1]^m$. The curve is H\"older continuous with exponent $1/m$ \cite[Theorem~2.1]{HeOwen2016}; more precisely,
\begin{equation}
    \big\|H_m(a)-H_m(b)\big\|_2\le 2\sqrt{m+3}\,|a-b|^{1/m},\quad a,b\in[0,1].
    \label{eq:hilbert-holder}
\end{equation}

At every time layer $k$, generate
\begin{equation*}
    P_k=\Big\{p_k^{(1)},\ldots,p_k^{(n)}\Big\}\subset[0,1)^r,
    \quad
    p_k^{(j)}=\big(\eta_k^{(j)},z_k^{(j)}\big), \quad j=1,\ldots,n.
\end{equation*}
For the analysis below, let $n=2^\ell$, and let each $P_k$ be an independently scrambled digital $(\mathfrak t,\ell,r)$-net with fixed $\mathfrak t$. The fresh scrambling is independent across layers, and the coordinate scramblings within each $P_k$ are independent. We require the projection onto the first coordinate to form a $(0,\ell,1)$-net, so that each interval $[(i-1)/n,i/n)$, $i=1,\ldots,n$, contains exactly one value of $\eta_k^{(j)}$. This one-dimensional stratification provides the rank variable used to match the randomized points with the ordered particle states. The remaining $r-1$ coordinates determine the p-Ret directions through
\begin{equation*}
    Z_k^{(j)}=\mathcal D\big(z_k^{(j)}\big),
    \quad j=1,\ldots,n.
\end{equation*}

After applying the prescribed parameter identifications, represent each particle by a coordinate $\xi$ in the rectangular parameter domain $\mathcal R:=\prod_{j=1}^m[a_j,b_j]\subset\mathbb R^m$, where $a_j<b_j$ for $j=1,\ldots,m$. Define the normalization $\mathcal N:\mathcal R\to[0,1]^m$ by
\begin{equation*}
    \mathcal N(\xi)_j=\frac{\xi_j-a_j}{b_j-a_j},
    \quad j=1,\ldots,m.
\end{equation*}
For a particle with coordinate $\xi\in\mathcal R$, define its scalar Hilbert key by $q(\xi)=H_m^{-1}\big(\mathcal N(\xi)\big)$. Since $H_m\big(q(\xi)\big)=\mathcal N(\xi)$, the H\"older estimate \eqref{eq:hilbert-holder} shows that particles with nearby Hilbert keys have nearby normalized coordinates.

At layer $k$, let $\mathcal I_k$ be the active index set and let $n_k=|\mathcal I_k|$. Active records are assigned their Hilbert keys, while inactive records in the Dirichlet problem retain their frozen states and accumulated contributions and are assigned the terminal key $+\infty$. Sorting the records by these keys therefore places all active records before the inactive ones. After sorting $P_k$ by the first coordinates $\eta_k^{(j)}$ and matching by rank, the first $n_k$ randomized points are paired with the active records. Only these records are propagated, while the remaining inputs are unused. Hence both the record array and the randomized point set retain size $n$ at every layer.

Let $\pi_{n,k}$ denote the permutation induced by the Hilbert ordering at layer $k$, with inactive records placed after all active records and ties resolved by their original labels. Thus, $\pi_{n,k}(i)$ is the index of the particle with Hilbert rank $i$. The subsequent analysis is formulated in terms of $\pi_{n,k}$ and therefore does not depend on a particular realization of the Hilbert key, beyond the stability assumptions stated below. The resulting Hilbert rank-matching procedure is illustrated in Figure~\ref{fig:hilbert-rank-matching}.

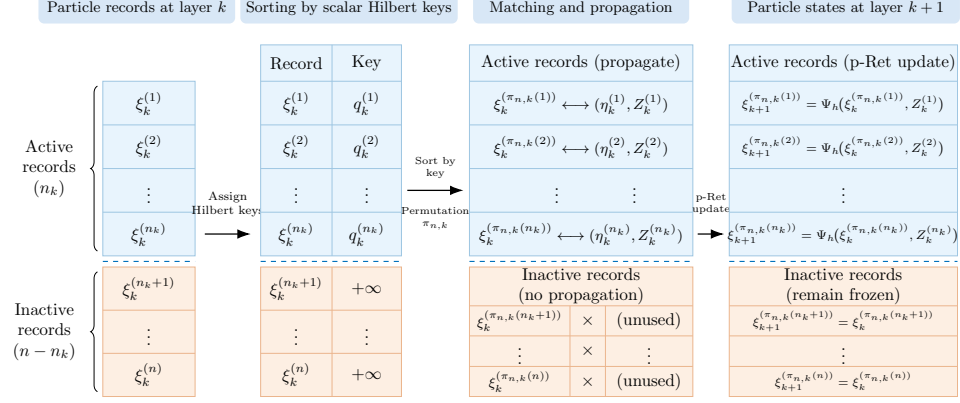
\begin{figure}[t]
    \centering
    \resizebox{\textwidth}{!}{%
    \begin{tikzpicture}[x=1cm,y=1cm,>=Latex,
        titlebox/.style={fill=hilbertheader,rounded corners=1.2mm,minimum height=0.62cm,font=\small},
        activebox/.style={draw=hilbertactiveline,fill=hilbertactive,line width=0.35pt},
        inactivebox/.style={draw=hilbertinactiveline,fill=hilbertinactive,line width=0.35pt},
        gridline/.style={draw=hilbertactiveline,line width=0.3pt},
        inactivegrid/.style={draw=hilbertinactiveline,line width=0.3pt},
        flow/.style={-{Latex[length=2.8mm,width=2mm]},line width=0.75pt}]

        \node[titlebox,minimum width=4.05cm] at (2.15,8.05) {Particle records at layer $k$};
\node[titlebox,minimum width=4.00cm] at (6.52,8.05) {Sorting by scalar Hilbert keys};
\node[titlebox,minimum width=4.45cm] at (11.30,8.05) {Matching and propagation};
        \node[titlebox,minimum width=4.70cm] at (16.75,8.05) {Particle states at layer $k+1$};

        \draw[activebox] (1.45,2.95) rectangle (3.35,6.55);
        \foreach \y in {3.85,4.75,5.65} {\draw[gridline] (1.45,\y)--(3.35,\y);}
        \node at (2.40,6.10) {$\xi_k^{(1)}$};
        \node at (2.40,5.20) {$\xi_k^{(2)}$};
        \node at (2.40,4.30) {$\vdots$};
        \node at (2.40,3.40) {$\xi_k^{(n_k)}$};

        \draw[inactivebox] (1.45,0.05) rectangle (3.35,2.72);
        \foreach \y in {0.94,1.83} {\draw[inactivegrid] (1.45,\y)--(3.35,\y);}
        \node at (2.40,2.275) {$\xi_k^{(n_k+1)}$};
        \node at (2.40,1.385) {$\vdots$};
        \node at (2.40,0.495) {$\xi_k^{(n)}$};

        \draw[decorate,decoration={brace,amplitude=4pt,mirror}] (1.34,6.45)--(1.34,3.05);
        \node[align=center,anchor=east] at (1.02,4.75) {Active\\records\\$(n_k)$};
        \draw[decorate,decoration={brace,amplitude=4pt,mirror}] (1.34,2.62)--(1.34,0.15);
        \node[align=center,anchor=east] at (1.02,1.39) {Inactive\\records\\$(n-n_k)$};

\draw[flow] (3.55,3.40)--(4.50,3.40);
\node[align=center,font=\scriptsize,inner sep=0pt] at (4.025,4.02) {Assign\\Hilbert keys};

\draw[activebox] (4.70,2.95) rectangle (7.60,7.30);
\draw[gridline] (4.70,6.55)--(7.60,6.55);
\foreach \y in {3.85,4.75,5.65} {\draw[gridline] (4.70,\y)--(7.60,\y);}
\draw[gridline] (6.17,2.95)--(6.17,7.30);
\node at (5.435,6.925) {Record};
\node at (6.885,6.925) {Key};
\node at (5.435,6.10) {$\xi_k^{(1)}$};
\node at (6.885,6.10) {$q_k^{(1)}$};
\node at (5.435,5.20) {$\xi_k^{(2)}$};
\node at (6.885,5.20) {$q_k^{(2)}$};
\node at (5.435,4.30) {$\vdots$};
\node at (6.885,4.30) {$\vdots$};
\node at (5.435,3.40) {$\xi_k^{(n_k)}$};
\node at (6.885,3.40) {$q_k^{(n_k)}$};

\draw[inactivebox] (4.70,0.05) rectangle (7.60,2.72);
\foreach \y in {0.94,1.83} {\draw[inactivegrid] (4.70,\y)--(7.60,\y);}
\draw[inactivegrid] (6.17,0.05)--(6.17,2.72);
\node at (5.435,2.275) {$\xi_k^{(n_k+1)}$};
\node at (6.885,2.275) {$+\infty$};
\node at (5.435,1.385) {$\vdots$};
\node at (6.885,1.385) {$\vdots$};
\node at (5.435,0.495) {$\xi_k^{(n)}$};
\node at (6.885,0.495) {$+\infty$};

\draw[flow] (7.72,4.30)--(8.88,4.30);
\node[align=center,font=\tiny,inner sep=0pt] at (8.30,4.78) {Sort by\\ key};
\node[align=center,font=\tiny,inner sep=0pt] at (8.30,3.72) {Permutation\\$\pi_{n,k}$};

\draw[activebox] (9.00,2.95) rectangle (13.60,7.30);
\draw[gridline] (9.00,6.55)--(13.60,6.55);
\foreach \y in {3.85,4.75,5.65} {\draw[gridline] (9.00,\y)--(13.60,\y);}
\node at (11.30,6.925) {Active records (propagate)};
\node[font=\footnotesize] at (11.30,6.10) {\scalebox{1.06}{$\xi_k^{(\pi_{n,k}(1))}\longleftrightarrow(\eta_k^{(1)},Z_k^{(1)})$}};
\node[font=\footnotesize] at (11.30,5.20) {\scalebox{1.06}{$\xi_k^{(\pi_{n,k}(2))}\longleftrightarrow(\eta_k^{(2)},Z_k^{(2)})$}};
\node at (10.80,4.30) {$\vdots$};
\node at (12.40,4.30) {$\vdots$};
\node[font=\footnotesize] at (11.30,3.40) {\scalebox{1.06}{$\xi_k^{(\pi_{n,k}(n_k))}\longleftrightarrow(\eta_k^{(n_k)},Z_k^{(n_k)})$}};

\draw[inactivebox] (9.00,0.05) rectangle (13.60,2.72);
\draw[inactivegrid] (9.00,1.92)--(13.60,1.92);
\foreach \y in {0.67,1.30} {\draw[inactivegrid] (9.00,\y)--(13.60,\y);}
\draw[inactivegrid] (11.07,0.05)--(11.07,1.92);
\draw[inactivegrid] (11.80,0.05)--(11.80,1.92);
\node[align=center,inner sep=0pt] at (11.30,2.32) {Inactive records\\(no propagation)};
\node[font=\footnotesize] at (10.035,1.61) {$\xi_k^{(\pi_{n,k}(n_k+1))}$};
\node at (11.435,1.61) {$\times$};
\node at (12.70,1.61) {(unused)};
\node at (10.035,0.985) {$\vdots$};
\node at (11.435,0.985) {$\times$};
\node at (12.70,0.985) {$\vdots$};
\node[font=\footnotesize] at (10.035,0.36) {$\xi_k^{(\pi_{n,k}(n))}$};
\node at (11.435,0.36) {$\times$};
\node at (12.70,0.36) {(unused)};

\draw[flow] (13.70,3.40)--(14.25,3.40);
\node[align=center,font=\tiny,inner sep=0pt] at (13.975,4.03) {p-Ret\\update};

        \draw[activebox] (14.35,2.95) rectangle (19.05,7.30);
        \draw[gridline] (14.35,6.55)--(19.05,6.55);
        \foreach \y in {3.85,4.75,5.65} {\draw[gridline] (14.35,\y)--(19.05,\y);}
        \node at (16.70,6.925) {Active records (p-Ret update)};
        \node[font=\scriptsize] at (16.70,6.10) {\scalebox{1.04}{$\xi_{k+1}^{(\pi_{n,k}(1))}=\Psi_h\!\big(\xi_k^{(\pi_{n,k}(1))},Z_k^{(1)}\big)$}};
        \node[font=\scriptsize] at (16.70,5.20) {\scalebox{1.04}{$\xi_{k+1}^{(\pi_{n,k}(2))}=\Psi_h\!\big(\xi_k^{(\pi_{n,k}(2))},Z_k^{(2)}\big)$}};
        \node at (16.70,4.30) {$\vdots$};
        \node[font=\scriptsize] at (16.70,3.40) {\scalebox{1.04}{$\xi_{k+1}^{(\pi_{n,k}(n_k))}=\Psi_h\!\big(\xi_k^{(\pi_{n,k}(n_k))},Z_k^{(n_k)}\big)$}};

        \draw[inactivebox] (14.35,0.05) rectangle (19.05,2.72);
        \draw[inactivegrid] (14.35,1.92)--(19.05,1.92);
        \foreach \y in {0.67,1.30} {\draw[inactivegrid] (14.35,\y)--(19.05,\y);}
        \node[align=center,inner sep=0pt] at (16.70,2.32) {Inactive records\\(remain frozen)};
        \node[font=\scriptsize] at (16.70,1.61) {$\xi_{k+1}^{(\pi_{n,k}(n_k+1))}=\xi_k^{(\pi_{n,k}(n_k+1))}$};
        \node at (16.70,0.985) {$\vdots$};
        \node[font=\scriptsize] at (16.70,0.36) {$\xi_{k+1}^{(\pi_{n,k}(n))}=\xi_k^{(\pi_{n,k}(n))}$};

\foreach \xa/\xb in {1.45/3.35,4.70/7.60,9.00/13.60,14.35/19.05}
            {\draw[hilbertdash,dashed,line width=0.65pt] (\xa,2.835)--(\xb,2.835);}
    \end{tikzpicture}%
    }
    \caption{Rank matching through scalar Hilbert keys and the resulting particle-state evolution from layer $k$ to layer $k+1$.}
    \label{fig:hilbert-rank-matching}
\end{figure}

For $\star\in\{\cl,\bd\}$, the Array-RQMC estimator is defined by
\begin{equation*}
\widehat u_{n,h}^{\ARQMC,\star}(x,T)=\frac{1}{n}\sum_{i=1}^nS_h^{\ARQMC,\star,(i)},
\end{equation*}
where $S_h^{\ARQMC,\star,(i)}$ is the same discrete Feynman--Kac functional as in Subsection~\ref{subsec:discrete-path-functionals}, evaluated along the Array-RQMC particle trajectory.

\subsection{Error estimates}
We first consider the closed manifold case with fixed $h=T/K$, where all $n$ particles remain active. Let $\xi_k^{(i)}$ be the identified parameter coordinate of particle $i$, obtained by repeatedly applying the common one-step p-Ret update $\Psi_h$ defined in \eqref{eq:one-step-pret-transition}, with the fresh Array-RQMC input assigned after rank matching, and set $X_k^{h,(i)}=\Phi(\xi_k^{(i)})$. Let $\mu_{h,k}=\mathcal L(X_k^h)$ be the probability law of $X_k^h$. We also let $\xi_k$ denote the chosen representative coordinate of $X_k^h$ and write $\nu_{h,k}=\mathcal L(\xi_k)$. Define
\begin{equation*}
    \widehat\mu_{n,h,k}(\varphi)=\frac1n\sum_{i=1}^n\varphi\big(X_k^{h,(i)}\big),
    \quad \varphi\in C(\M).
\end{equation*}
The Hilbert sort only relabels the particles and therefore leaves this empirical integral unchanged.

For $\varphi\in C(\M)$, the one-step transition operator can be written using the same coordinate update as in the direct RQMC analysis:
\begin{equation*}
    (\mathcal P_h\varphi)\big(\Phi(\xi)\big)=\int_{[0,1)^{r-1}}\varphi\big(\Phi\big(\Psi_h(\xi,\mathcal D(z))\big)\big)\,\dd z.
\end{equation*}
Set $I_{n,i}=[(i-1)/n,i/n)$, $i=1,\ldots,n$. By the rank-matching construction of Subsection~\ref{subsec:array-construction}, the particle $\xi_k^{(\pi_{n,k}(i))}$ is paired with the unique point of $P_k$ whose first coordinate lies in $I_{n,i}$. Accordingly, for $k<K$ and $\varphi\in C(\M)$, define the layerwise rank-coupled integrand
\begin{equation}
    G_{n,h,k}^{\varphi}(\eta,z)=\sum_{i=1}^n\mathbf 1_{I_{n,i}}(\eta)
    \varphi\big(\Phi\big(\Psi_h(\xi_k^{(\pi_{n,k}(i))},\mathcal D(z))\big)\big),\quad(\eta,z)\in[0,1)^r,
    \label{eq:array-layer-integrand}
\end{equation}

For $n=\mathfrak b^\ell$, let $P_{n,k}=\{(\eta_k^{(q)},z_k^{(q)}):q=1,\ldots,n\}\subset[0,1)^r$ be the independently scrambled digital net specified in Subsection~\ref{subsec:array-construction}. The associated randomized quadrature operator is defined, for $v\in L^2([0,1)^r)$, by
\begin{equation*}
    Q_{n,k}v=\frac1n\sum_{q=1}^nv(\eta_k^{(q)},z_k^{(q)}).
\end{equation*}
The standard scrambled-net bounds apply directly to each fixed $L^2$ integrand. Since the rank-coupled functions $G_{n,h,k}^{\varphi}$ are instead random and depend on $n$, the layerwise argument requires the following regularity and stability conditions. Let $\mathsf U$ denote the parameter domain with the prescribed identifications, equipped with a metric $d_{\mathsf U}$ compatible with its topology.
The $L^2$ modulus of the one-step update in the parameter variable is
\begin{equation*}
    \omega_{h,\varphi}(\delta)=\sup_{d_{\mathsf U}(\xi,\xi')\le \delta}
    \bigg(\int_{[0,1)^{r-1}}\Big|\varphi\big(\Phi\big(\Psi_h(\xi,\mathcal D(z))\big)\big)-\varphi\big(\Phi\big(\Psi_h(\xi',\mathcal D(z))\big)\big)\Big|^2\,\dd z\bigg)^{1/2}.
\end{equation*}
Since $\mathsf U$ is compact modulo the prescribed identifications and the map $(\xi,Z)\longmapsto \varphi\big(\Phi\big(\Psi_h(\xi,Z)\big)\big)$ is continuous on $\mathsf U\times S^{r-1}$, uniform continuity implies
\begin{equation*}
    \omega_{h,\varphi}(\delta)\longrightarrow0,
    \quad \delta\downarrow0.
\end{equation*}

\begin{assumption}
\label{ass:array-rqmc}
The transition operator satisfies $\mathcal P_h(C(\M))\subset C(\M)$, and $u_0,f(\cdot,T-t_k)\in C(\M)$ for $k=0,\ldots,K-1$.
\end{assumption}

For the order-stability analysis, take $d_{\mathsf U}$ to be the Euclidean distance between representatives in $\mathcal R$. For $k=1,\ldots,K-1$, define the empirical and limiting Hilbert keys and their cumulative distribution functions (CDFs) by
\begin{equation*}
    \begin{aligned}
    q_k^{(i)}&=H_m^{-1}\big(\mathcal N(\xi_k^{(i)})\big),\quad\widehat F_{n,h,k}(x)=\frac1n\sum_{i=1}^n\mathbf 1_{\{q_k^{(i)}\le x\}},\\
    q_k&=H_m^{-1}\big(\mathcal N(\xi_k)\big),\quad F_{h,k}(x)=\mathbb P\big(q_k\le x\big),\quad \xi_k\sim\nu_{h,k}.
    \end{aligned}
\end{equation*}
Let $\widehat F_{n,h,k}^{-1}$ and $F_{h,k}^{-1}$ denote the corresponding generalized inverse, or quantile, functions on $(0,1)$. At $k=0$, set $\zeta_{n,h,0}(\eta)=\zeta_{h,0}(\eta)=\xi_0$ for $\eta\in[0,1)$. For $k=1,\ldots,K-1$ and $\eta\in(0,1)$, define the empirical and deterministic Hilbert order maps by
\begin{equation*}
    \zeta_{n,h,k}(\eta)=\mathcal N^{-1}\circ H_m\circ\widehat F_{n,h,k}^{-1}(\eta),\quad\zeta_{h,k}(\eta)=\mathcal N^{-1}\circ H_m\circ F_{h,k}^{-1}(\eta).
\end{equation*}

\begin{lemma}
\label{lem:hilbert-order-stability}
Suppose that $F_{h,k}$ is continuous and strictly increasing on $[0,1]$ for every $k=1,\ldots,K-1$. Then, for every $k=0,\ldots,K-1$,
\begin{equation}
    \E\Big[\int_{[0,1)}d_{\mathsf U}\big(\zeta_{n,h,k}(\eta),\zeta_{h,k}(\eta)\big)^2\,\dd\eta\Big]\longrightarrow0,\quad n\to\infty.
    \label{eq:array-order-map-stability}
\end{equation}
\end{lemma}

\begin{proof}
The case $k=0$ is immediate from the definitions. Fix $k\in\{1,\ldots,K-1\}$. Up to the endpoints of the intervals $I_{n,i}$, $\widehat F_{n,h,k}^{-1}$ selects the $i$-th ordered key. Since $H_m\big(H_m^{-1}(x)\big)=x$, the empirical order map defined above agrees with the Hilbert rank ordering. The quantile convergence argument in \cite[Appendix~A.4.4, proof of Theorem~7]{GerberChopin2015} yields, for every $\eta\in(0,1)$,
\begin{equation*}\widehat F_{n,h,k}^{-1}(\eta)
    \longrightarrow F_{h,k}^{-1}(\eta)
    \quad\text{almost surely}.
\end{equation*}
Because $\mathcal N^{-1}(u)_j=a_j+(b_j-a_j)u_j$ for $j=1,\ldots,m$, the H\"older estimate \eqref{eq:hilbert-holder} gives
\begin{equation*}
    \begin{aligned}
    d_{\mathsf U}\big(\zeta_{n,h,k}(\eta),\zeta_{h,k}(\eta)\big)^2&=\Big\|\mathcal N^{-1}\circ H_m\circ\widehat F_{n,h,k}^{-1}(\eta)-\mathcal N^{-1}\circ H_m\circ F_{h,k}^{-1}(\eta)\Big\|_2^2\\
    &\le \max_{1\le j\le m}(b_j-a_j)^2\Big\|H_m\circ\widehat F_{n,h,k}^{-1}(\eta)-H_m\circ F_{h,k}^{-1}(\eta)\Big\|_2^2\\
    &\le 4(m+3)\max_{1\le j\le m}(b_j-a_j)^2\big|\widehat F_{n,h,k}^{-1}(\eta)-F_{h,k}^{-1}(\eta)\big|^{2/m}.
    \end{aligned}
\end{equation*}
The preceding estimate shows that $d_{\mathsf U}\big(\zeta_{n,h,k}(\eta),\zeta_{h,k}(\eta)\big)^2\to0$ almost surely for every $\eta\in(0,1)$. Dominated convergence in the randomization variable therefore gives
\begin{equation*}
    \E\Big[d_{\mathsf U}\big(\zeta_{n,h,k}(\eta),\zeta_{h,k}(\eta)\big)^2\Big]\longrightarrow0.
\end{equation*}
Hence Fubini's theorem and dominated convergence in $\eta$ yield
\begin{equation*}
    \begin{aligned}
    &\lim_{n\to\infty}\E\Big[\int_{[0,1)}d_{\mathsf U}\big(\zeta_{n,h,k}(\eta),\zeta_{h,k}(\eta)\big)^2\,\dd\eta\Big]\\
    &\quad=\int_{[0,1)}\lim_{n\to\infty}\E\Big[d_{\mathsf U}\big(\zeta_{n,h,k}(\eta),\zeta_{h,k}(\eta)\big)^2\Big]\,\dd\eta=0.
    \end{aligned}
\end{equation*}
This proves \eqref{eq:array-order-map-stability}.
\end{proof}

\begin{lemma}
\label{lem:array-integrand-stability}
Under the assumptions of Lemma~\ref{lem:hilbert-order-stability}, for $k<K$ and $\varphi\in C(\M)$, define $G_{h,k}^{\varphi}(\eta,z)=\varphi\big(\Phi\big(\Psi_h\big(\zeta_{h,k}(\eta),\mathcal D(z)\big)\big)\big)$. Then $G_{h,k}^{\varphi}\in L^2([0,1)^r)$ and
\begin{equation}
    \E\Big[\big\|G_{n,h,k}^{\varphi}-G_{h,k}^{\varphi}\big\|_2^2\Big]\longrightarrow0,
    \quad n\to\infty.
    \label{eq:array-ordered-integrand-stability}
\end{equation}
\end{lemma}

\begin{proof}
Set $M_{\varphi}=\|\varphi\|_{\infty}$ and $D_{n,k}(\eta)=d_{\mathsf U}\big(\zeta_{n,h,k}(\eta),\zeta_{h,k}(\eta)\big)$. Since $\M$ is compact and $\varphi\in C(\M)$, both $G_{n,h,k}^{\varphi}$ and $G_{h,k}^{\varphi}$ are bounded in absolute value by $M_{\varphi}$.

Fix $\delta>0$ and decompose $[0,1)$ into the sets
$\{D_{n,k}\le\delta\}$ and $\{D_{n,k}>\delta\}$. On the first set, the
definition of $\omega_{h,\varphi}$ gives
\begin{equation*}
    \int_{[0,1)^{r-1}}\big|G_{n,h,k}^{\varphi}(\eta,z)-G_{h,k}^{\varphi}(\eta,z)\big|^2\,\dd z\le\omega_{h,\varphi}(\delta)^2.
\end{equation*}
On the second set, the uniform bound gives
\begin{equation*}
    \int_{[0,1)^{r-1}}\big|G_{n,h,k}^{\varphi}(\eta,z)-G_{h,k}^{\varphi}(\eta,z)\big|^2\,\dd z\le 4M_{\varphi}^2.
\end{equation*}
Using $\mathbf 1_{\{D_{n,k}>\delta\}}\le\frac{D_{n,k}^2}{\delta^2}$, integrating with respect to $\eta$, and then taking expectation yield
\begin{equation*}
    \E\Big[\big\|G_{n,h,k}^{\varphi}-G_{h,k}^{\varphi}\big\|_2^2\Big]\le\omega_{h,\varphi}(\delta)^2+\frac{4M_{\varphi}^2}{\delta^2}\E\Big[\int_{[0,1)}D_{n,k}(\eta)^2\,\dd\eta\Big].
\end{equation*}
For fixed $\delta$, Lemma~\ref{lem:hilbert-order-stability} makes the second term tend to zero as $n\to\infty$. Hence
\begin{equation*}
    \limsup_{n\to\infty}\E\Big[\big\|G_{n,h,k}^{\varphi}-G_{h,k}^{\varphi}\big\|_2^2\Big]\le\omega_{h,\varphi}(\delta)^2.
\end{equation*}
Letting first $n\to\infty$ and then $\delta\downarrow0$ proves \eqref{eq:array-ordered-integrand-stability}. Finally, $\|G_{h,k}^{\varphi}\|_2\le M_{\varphi}$, and therefore $G_{h,k}^{\varphi}\in L^2([0,1)^r)$.
\end{proof}


We now combine Lemma~\ref{lem:array-integrand-stability}
with the scrambled-net variance bound \cite{Owen1997MonteCarloVariance}.
For a fixed $v\in L^2([0,1)^r)$, nested scrambling gives
\begin{equation*}
    \operatorname{Var}(Q_{n,k}v)\le\frac{C_{\mathrm{scr}}}{n}\Big\|v-\int_{[0,1)^r}v\Big\|_2^2,
    \quad\quad
    n\operatorname{Var}(Q_{n,k}v)\longrightarrow0.
\end{equation*}

\begin{lemma}
\label{lem:closed-array-rqmc-error}
Under Assumption~\ref{ass:array-rqmc}, the assumptions of Lemma~\ref{lem:hilbert-order-stability}, and the layer-point construction above, for every $\varphi\in C(\M)$ and every fixed $k=0,\ldots,K$, we have
\begin{equation}
    \E\Big[\big|\widehat\mu_{n,h,k}(\varphi)-\mu_{h,k}(\varphi)\big|^2\Big]
    =\operatorname{Var}\big(\widehat\mu_{n,h,k}(\varphi)\big)=o(n^{-1}).
    \label{eq:array-layer-sqmc-rate}
\end{equation}
Consequently, for every fixed $x\in\M$ and fixed $h=T/K$,
\begin{equation}
    \operatorname{Var}\big(\widehat u_{n,h}^{\ARQMC,\cl}(x,T)\big)=o(n^{-1}).
    \label{eq:array-rqmc-variance}
\end{equation}
\end{lemma}

\begin{proof}

Fix $\varphi\in C(\M)$ and $k<K$, and let $\mathcal F_{n,k}$ be the sigma-field generated by the particle array through time $t_k$, before the fresh scrambling of $P_{n,k}$. Then $G_{n,h,k}^{\varphi}$ is $\mathcal F_{n,k}$-measurable. Since it evaluates $\varphi$ after one p-Ret update from layer $k$, the first-coordinate stratification and rank matching give
\begin{equation}
    \widehat\mu_{n,h,k+1}(\varphi)
    =Q_{n,k}G_{n,h,k}^{\varphi}.
    \label{eq:array-rank-quadrature-identity}
\end{equation}
Moreover, integrating first with respect to $z$ and then with respect to $\eta$ yields
\begin{equation*}
    \int_{[0,1)^r}G_{n,h,k}^{\varphi}(\eta,z)\,\dd(\eta,z)=\widehat\mu_{n,h,k}(\mathcal P_h\varphi).
\end{equation*}
The fresh randomization and the fixed-integrand unbiasedness established above therefore imply
\begin{equation}
    \E\big[\widehat\mu_{n,h,k+1}(\varphi)\mid\mathcal F_{n,k}\big]=\widehat\mu_{n,h,k}(\mathcal P_h\varphi).
    \label{eq:array-conditional-mean}
\end{equation}

For a deterministic $v$, set $V_n(v)=\operatorname{Var}(Q_{n,k}v)$. Applying the variance inequality and the gain bound to $G_{n,h,k}^{\varphi}-G_{h,k}^{\varphi}$, and then using Lemma~\ref{lem:array-integrand-stability}, gives
\begin{equation}
\begin{split}
    &\E\Big[\operatorname{Var}\big(\widehat\mu_{n,h,k+1}(\varphi)\mid\mathcal F_{n,k}\big)\Big]
    =\E\big[V_n\big(G_{n,h,k}^{\varphi}\big)\big]\\
    &\quad\le2V_n\big(G_{h,k}^{\varphi}\big)
    +\frac{2C_{\mathrm{scr}}}{n}\E\Big[\big\|G_{n,h,k}^{\varphi}-G_{h,k}^{\varphi}\big\|_2^2\Big]
    =o(n^{-1}).
\end{split}
\label{eq:array-innovation-variance}
\end{equation}

Since $\widehat\mu_{n,h,0}=\mu_{h,0}=\delta_x$ and Assumption~\ref{ass:array-rqmc} gives $\mathcal P_h(C(\M))\subset C(\M)$, \eqref{eq:array-conditional-mean} and $\mu_{h,k+1}(\varphi)=\mu_{h,k}(\mathcal P_h\varphi)$ yield $\E\big[\widehat\mu_{n,h,k+1}(\varphi)\big]=\mu_{h,k+1}(\varphi)$. The law of total variance, \eqref{eq:array-innovation-variance}, and the same induction give
\begin{equation}
    \operatorname{Var}\big(\widehat\mu_{n,h,k+1}(\varphi)\big)
    =\E\Big[\operatorname{Var}\big(\widehat\mu_{n,h,k+1}(\varphi)\mid\mathcal F_{n,k}\big)\Big]
    +\operatorname{Var}\big(\widehat\mu_{n,h,k}(\mathcal P_h\varphi)\big)
    =o(n^{-1}).
    \label{eq:array-total-variance-recursion}
\end{equation}
Together with the layerwise unbiasedness, the variance estimate implies
\begin{equation}
    \E\Big[\big|\widehat\mu_{n,h,k}(\varphi)-\mu_{h,k}(\varphi)\big|^2\Big]=\operatorname{Var}\big(\widehat\mu_{n,h,k}(\varphi)\big)=o(n^{-1}),
\end{equation}
for every fixed $k=0,\ldots,K$. This proves \eqref{eq:array-layer-sqmc-rate}.

Finally, particle sorting only permutes the array, so the discrete Feynman--Kac estimator has the layerwise representation
\begin{equation*}
\begin{aligned}
    \widehat u_{n,h}^{\ARQMC,\cl}(x,T)
    &=\widehat\mu_{n,h,K}(u_0)+h\sum_{k=0}^{K-1}\widehat\mu_{n,h,k}\big(f(\cdot,T-t_k)\big),\\
    u_h^\cl(x,T)
    &=\mu_{h,K}(u_0)+h\sum_{k=0}^{K-1}\mu_{h,k}\big(f(\cdot,T-t_k)\big).
\end{aligned}
\end{equation*}
Thus the Array-RQMC estimator is exactly unbiased for the time-discrete Feynman--Kac quantity $u_h^\cl$. Moreover, since $K$ is fixed, the layerwise variance estimate yields
\begin{equation*}
\begin{aligned}
    &\operatorname{Var}\big(\widehat u_{n,h}^{\ARQMC,\cl}(x,T)\big)
    \le(K+1)\Big[\operatorname{Var}\big(\widehat\mu_{n,h,K}(u_0)\big)\\
    &\quad+h^2\sum_{k=0}^{K-1}\operatorname{Var}\big(\widehat\mu_{n,h,k}(f(\cdot,T-t_k))\big)\Big]
    =o(n^{-1}).
\end{aligned}
\end{equation*}
This proves \eqref{eq:array-rqmc-variance}.
\end{proof}

We finally consider the Dirichlet problem. Since first-exit stopping introduces discontinuities into the layerwise integrands, the closed-case continuity argument does not apply directly. We therefore assume $L^2$ stability of the stopped rank-coupled integrands, incorporating both Hilbert-order stability and the effect of the exit interfaces.

\begin{assumption}
\label{ass:boundary-array-rqmc}
For every fixed $k<K$, let $G_{n,h,k}^{\bd}$ denote the rank-coupled layer integrand associated with the stopped Array-RQMC update at layer $k$. Suppose that there exists a deterministic function $G_{h,k}^{\bd}\in L^2([0,1)^r)$ such that
\begin{equation}
    \E\Big[\big\|G_{n,h,k}^{\bd}-G_{h,k}^{\bd}\big\|_2^2\Big]\longrightarrow0,
    \quad
    n\to\infty.
    \label{eq:boundary-array-integrand-stability}
\end{equation}
\end{assumption}

Assumption~\ref{ass:boundary-array-rqmc} is the stopped counterpart of Lemma~\ref{lem:array-integrand-stability}. In the Dirichlet case, the additional difficulty is caused by the discontinuity of the exit decision. A sufficient verification may be obtained by combining the ordering stability with the transversality and finite-complexity control of the numerical exit interfaces used in the direct RQMC analysis.

\begin{lemma}
\label{lem:boundary-array-rqmc-error}

Suppose that Assumption~\ref{ass:bounded-data} and Assumption~\ref{ass:boundary-array-rqmc} hold. Then, for every fixed $x\in\M$ and fixed $h=T/K$,
\begin{equation*}
    \E\big[\widehat u_{n,h}^{\ARQMC,\bd}(x,T)\big]=u_h^\bd(x,T),
    \quad
    \operatorname{Var}\big(\widehat u_{n,h}^{\ARQMC,\bd}(x,T)\big)=o(n^{-1}).
\end{equation*}
\end{lemma}

\begin{proof}
Conditional on the augmented particle array at the beginning of each layer, inactive records remain fixed, while the active records are propagated using the fresh scrambled point set. Because the fresh scrambled point set is independent of the past and each scrambled point is marginally uniform on $[0,1)^r$, the stopped augmented update satisfies the same conditional-mean identity as \eqref{eq:array-conditional-mean}. Moreover, Assumption~\ref{ass:boundary-array-rqmc}, together with the fixed-integrand scrambled-net estimate, yields an $o(n^{-1})$ conditional innovation variance at each layer. Since $K$ is fixed, the law of total variance and the finite-layer induction used in Lemma~\ref{lem:closed-array-rqmc-error} give
\begin{equation*}
    \operatorname{Var}\big(\widehat u_{n,h}^{\ARQMC,\bd}(x,T)\big)=o(n^{-1}).
\end{equation*}
Iterating this conditional-mean identity over the $K$ layers and using the definition of the time-discrete stopped Feynman--Kac functional gives $\E[\widehat u_{n,h}^{\rm A\text{-}RQMC,\bd}(x,T)]=u_h^\bd(x,T)$, which completes the proof.
\end{proof}

The preceding sampling-error estimates can now be combined with the temporal discretization bounds to obtain the total mean-square error.

\begin{theorem}
\label{thm:array-rqmc-total-error}
Let $x\in\M$ be fixed and let $h=T/K$.

\emph{(i) \underline{Closed manifold.}} Suppose that Assumption~\ref{ass:array-rqmc} holds, that the assumptions of Lemmas~\ref{lem:hilbert-order-stability} and \ref{lem:weak-temporal-convergence} are satisfied, and that the layer point sets are constructed as above. Then
\begin{equation}
    \E\Big[\big|\widehat u_{n,h}^{\ARQMC,\cl}(x,T)-u^\cl(x,T)\big|^2\Big]\le Ch^2+o(n^{-1}),
    \label{eq:array-rqmc-total-mse}
\end{equation}
where $C$ is a positive constant independent of $h$, and the $o(n^{-1})$ term may depend on the fixed value of $h$.

\emph{(ii) \underline{Manifold with Dirichlet boundary.}} Suppose that Assumption~\ref{ass:bounded-data} and Assumption~\ref{ass:boundary-array-rqmc} hold, and that the temporal weak error satisfies \eqref{eq:dirichlet-temporal-error}. Then
\begin{equation}
    \E\Big[\big|\widehat u_{n,h}^{\ARQMC,\bd}(x,T)-u^\bd(x,T)\big|^2\Big]
    \le C\big(h+\mathcal E_{\partial}\big)^2+o(n^{-1}),
    \label{eq:boundary-array-rqmc-total-mse}
\end{equation}
where $C$ is a positive constant independent of $h$.
In particular, if $\mathcal E_{\partial}\le Ch$, then
\begin{equation*}
    \E\Big[\big|\widehat u_{n,h}^{\ARQMC,\bd}(x,T)-u^\bd(x,T)\big|^2\Big]\le Ch^2+o(n^{-1}).
\end{equation*}
\end{theorem}

\begin{proof}
Lemma~\ref{lem:closed-array-rqmc-error} gives unbiasedness for the time-discrete closed-manifold estimator and an $o(n^{-1})$ variance rate. The preceding Dirichlet argument gives unbiasedness for the time-discrete estimator and an $o(n^{-1})$ variance rate. The result follows from
\begin{equation*}
    \E\Big[\big|\widehat u_{n,h}^{\ARQMC,\star}-u^\star\big|^2\Big]=\operatorname{Var}\big(\widehat u_{n,h}^{\ARQMC,\star}\big)+\big|u_h^\star-u^\star\big|^2,
\end{equation*}
Lemma~\ref{lem:weak-temporal-convergence}, and \eqref{eq:dirichlet-temporal-error}.
\end{proof}

\section{Numerical experiments}
\label{sec:numerical-experiments}

In this section, we present numerical experiments to compare the MC, direct RQMC, and Array-RQMC p-Ret estimators. We consider temporal convergence, sampling variance at fixed initial points, and RMSE over prescribed test sets.

For the first two numerical experiments, the RQMC point sets are scrambled Sobol' nets generated using MATLAB's \texttt{sobolset}. For the high-dimensional Lawson-product experiment, the Sobol' nets are generated using the Joe--Kuo $D^{(6)}$ direction numbers \cite{JoeKuo2008}, based on the data set \texttt{new-joe-kuo-6.21201}, and are randomized by a lower-triangular linear-matrix scramble followed by a digital shift \cite{HongHickernell2003,Owen2003AlternativeScrambling}.

For a fixed $x\in \M$, let $\widehat u_{n,h}^{M,\star,[\nu]}(x,T)$ with $\nu=1,\ldots,R$ be $R$ independent randomized replications of \eqref{eq:unified-estimator}, where $M\in\{\MC,\allowbreak\RQMC,\allowbreak\ARQMC\}$ and $\star\in\{\cl,\bd\}$. Define
\begin{equation*}
    \overline u_{n,h,R}^{M,\star}(x,T)=\frac1R\sum_{\nu=1}^R\widehat u_{n,h}^{M,\star,[\nu]}(x,T),
\end{equation*}
and estimate the variance by
\begin{equation}
    \widehat{\operatorname{Var}}\big(\widehat u_{n,h}^{M,\star}(x,T)\big)=\frac1{R-1}\sum_{\nu=1}^R\Big(\widehat u_{n,h}^{M,\star,[\nu]}(x,T)-\overline u_{n,h,R}^{M,\star}(x,T)\Big)^2.
    \label{eq:numerical-variance}
\end{equation}

For $N_x$ fixed test points $x_j=\Phi(\xi_j)$, $j=1,\ldots,N_x$, we define the discrete maximum error by
\begin{equation}
    E_{\max}^{M,\star}(n,h)=\max_{1\le j\le N_x}\big|\overline u_{n,h,R}^{M,\star}(x_j,T)-u^\star(x_j,T)\big|.
    \label{eq:numerical-max-error}
\end{equation}
For the spatial error tests, no averaging over independent randomized replications is performed. $\widehat u_{n,h}^{M,\star}$ denotes a single randomized realization of the corresponding estimator, and the root mean squared error is defined by
\begin{equation}
    \operatorname{RMSE}_{M,\star}(n,h)=\Bigg[\frac1{N_x}\sum_{j=1}^{N_x}\big|\widehat u_{n,h}^{M,\star}(x_j,T)-u^\star(x_j,T)\big|^2\Bigg]^{1/2}.
    \label{eq:numerical-rmse}
\end{equation}

Unless otherwise stated, all experiments use $T=0.20$. For the sampling tests, the number of time steps is fixed at $K=64$, and empirical variances are estimated
from $R=50$ independent randomized replications. Spatial errors are evaluated over $N_x=100$ fixed points, with the same test set used for all sampling methods and sample sizes. Temporal convergence is examined separately with $K=2^j$, $j=2,\ldots,7$, using $n=2^{15}$ direct RQMC samples.

\subsection{A general diffusion operator on a closed star-shaped torus}
\label{subsec:star-torus}

We first consider a general diffusion problem on a closed two-dimensional star-shaped torus shown in Figure~\ref{fig:star-torus-overview}(\subref{fig:star-torus-geometry}), which is parameterized by
\begin{equation}
    \Phi(\theta,\omega)=\big(\rho(\theta)\cos\omega,\rho(\theta)\sin\omega,r(\theta)\sin\theta\big),
    \label{eq:star-torus-param}
\end{equation}
where $(\theta,\omega)\in U=[0,2\pi)\times[0,2\pi)$, $\theta,\omega$ are periodic, and 
\begin{equation*}
    r(\theta)=r_0+\varepsilon_r\sin(\nu\theta),
    \quad
    \rho(\theta)=a-r(\theta)\cos\theta,
\end{equation*}
with $a=2$, $r_0=\frac35$, $\varepsilon_r=\frac{3}{40}$, and $\nu=5$.

We next specify the diffusion tensor $A$ and tangent drift $b$ appearing in \eqref{eq:intrinsic-operator}. Let $\{\widehat e_\theta,\widehat e_\omega\}$ denote the orthonormal tangent frame induced by \eqref{eq:star-torus-param}, and set $\varphi(\theta,\omega)=0.30\sin(\theta+\omega)$. The diffusion vector fields are prescribed by
\begin{eqnarray*}
    &&\sigma_1=\big(1+0.20\sin\theta\big)\Big(\cos\varphi\,\widehat e_\theta+\sin\varphi\,\widehat e_\omega\Big),\\
    &&\sigma_2=\big(0.75+0.15\cos\omega\big)\Big(-\sin\varphi\,\widehat e_\theta+\cos\varphi\,\widehat e_\omega\Big),
\end{eqnarray*}
from which the diffusion tensor in \eqref{eq:intrinsic-operator} is constructed as $A=\sigma_1\otimes\sigma_1+\sigma_2\otimes\sigma_2$. We further prescribe the tangent drift $b$ in \eqref{eq:intrinsic-operator} by
\begin{equation*}
    b=\big(0.10\sin\omega+0.05\cos(\theta+\omega)\big)\widehat e_\theta+\big(0.08\cos\theta-0.04\sin(2\omega)\big)\widehat e_\omega.
\end{equation*}

We take the exact solution
\begin{equation}
    u_{\rm ex}\big(\Phi(\theta,\omega),t\big)=e^{-t}\Big[\sin(2\theta)+\frac12\cos(3\omega)+\frac14\sin(\theta+\omega)\Big],
    \label{eq:star-torus-exact}
\end{equation}
with $u_0=u_{\rm ex}(\cdot,0)$ and $f=\partial_tu_{\rm ex}-\mathcal L u_{\rm ex}$.


\begin{figure}[ht!]
    \centering
    \begin{subfigure}[t]{0.48\textwidth}
        \centering
        \includegraphics[width=\linewidth,trim=8 35 8 8,clip]{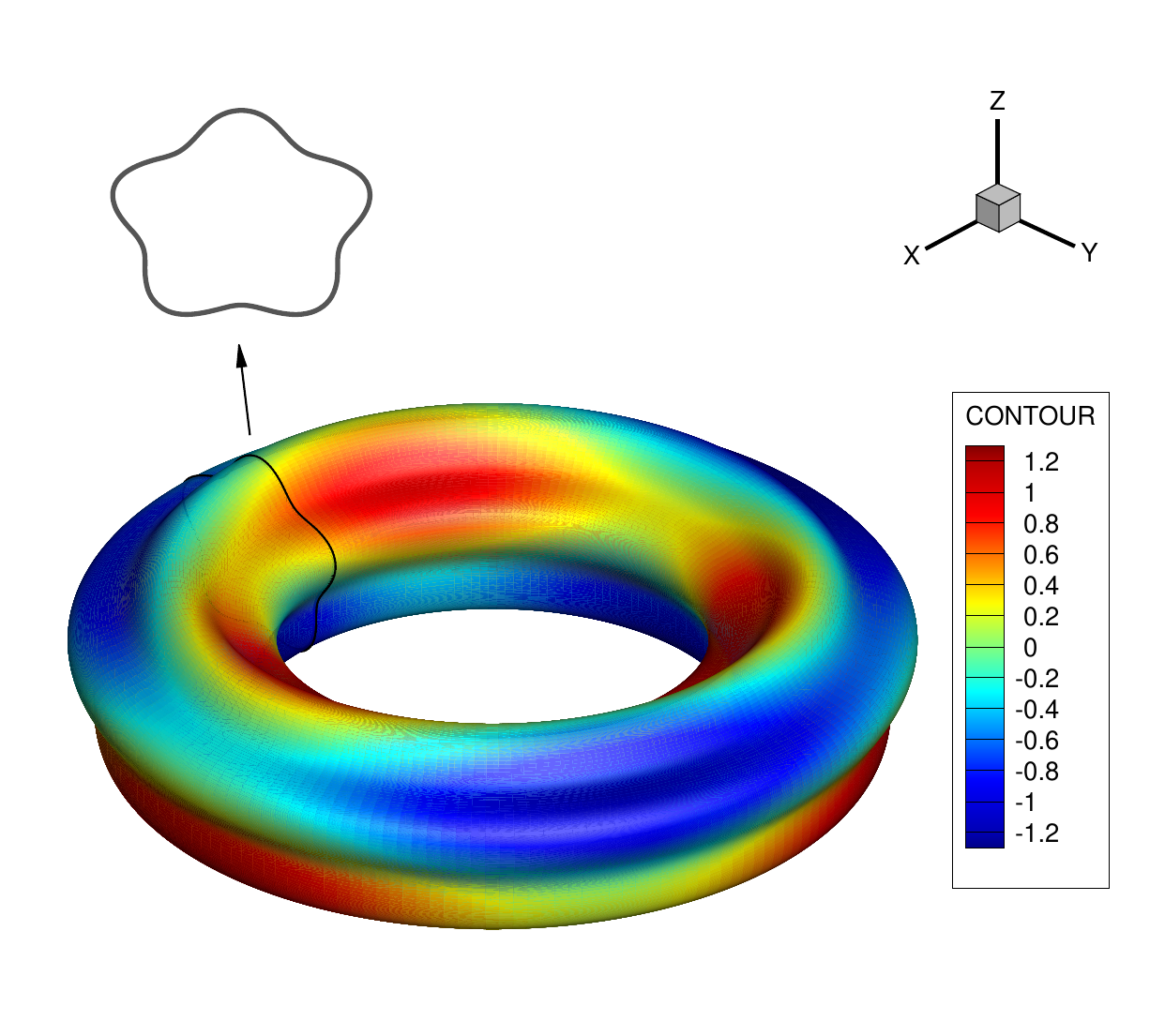}
        \caption{Star-shaped torus colored by the exact solution $u_{\rm ex}(\cdot,T)$ at $T=0.20$.}
        \label{fig:star-torus-geometry}
    \end{subfigure}
    \hfill
    \begin{subfigure}[t]{0.48\textwidth}
        \centering
        \includegraphics[width=\linewidth]{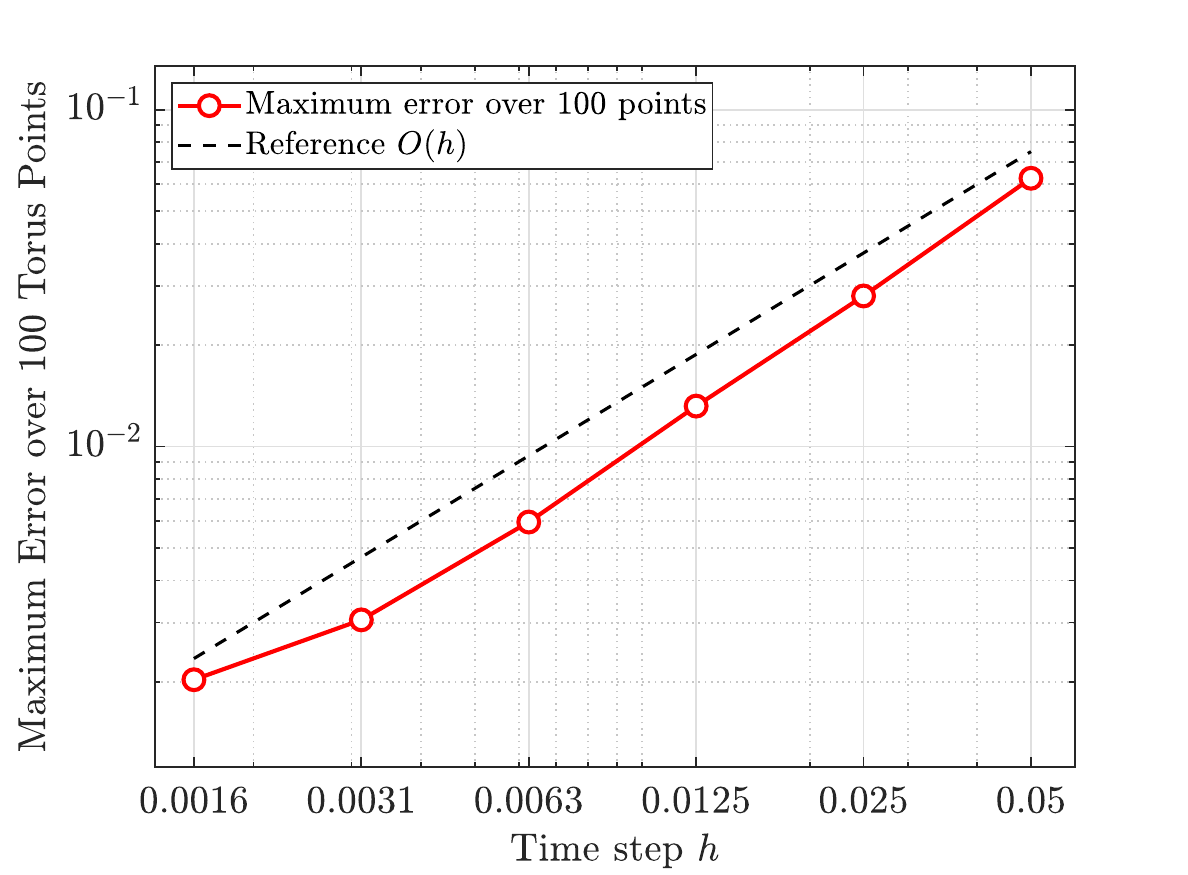}
        \caption{Temporal convergence measured by the maximum error over $100$ fixed test points.}
        \label{fig:star-torus-temporal}
    \end{subfigure}
    \caption{Star-shaped torus parameterized by \eqref{eq:star-torus-param} and temporal convergence of the p-Ret discretization.}
    \label{fig:star-torus-overview}
\end{figure}

We first verify the temporal convergence of the p-Ret discretization. Figure~\ref{fig:star-torus-overview}(\subref{fig:star-torus-temporal}) shows first-order decay in $h$, in agreement with Lemma~\ref{lem:weak-temporal-convergence}. 
We also compare the sampling performance of the three estimators using $n=2^j$, $j=3,\ldots,12$. At the fixed initial point $x_0$, Figure~\ref{fig:star-torus-sampling}(\subref{fig:star-torus-variance}) shows empirical variance rates of approximately $O(n^{-1.07})$, $O(n^{-1.27})$, and $O(n^{-1.77})$ for MC, direct RQMC, and Array-RQMC, respectively. The MC rate agrees with the expected $O(n^{-1})$ behavior, while both RQMC methods exhibit improved variance decay, consistent with the asymptotic estimate of Lemma~\ref{lem:closed-direct-rqmc-general}. The spatial RMSE in Figure~\ref{fig:star-torus-sampling}(\subref{fig:star-torus-rmse}) shows the same ordering, with fitted rates approximately $O(n^{-0.5})$, $O(n^{-0.59})$, and $O(n^{-0.79})$, respectively.



\begin{figure}[ht!]
    \centering
    \begin{subfigure}[t]{0.48\textwidth}
        \centering
        \includegraphics[width=\textwidth]{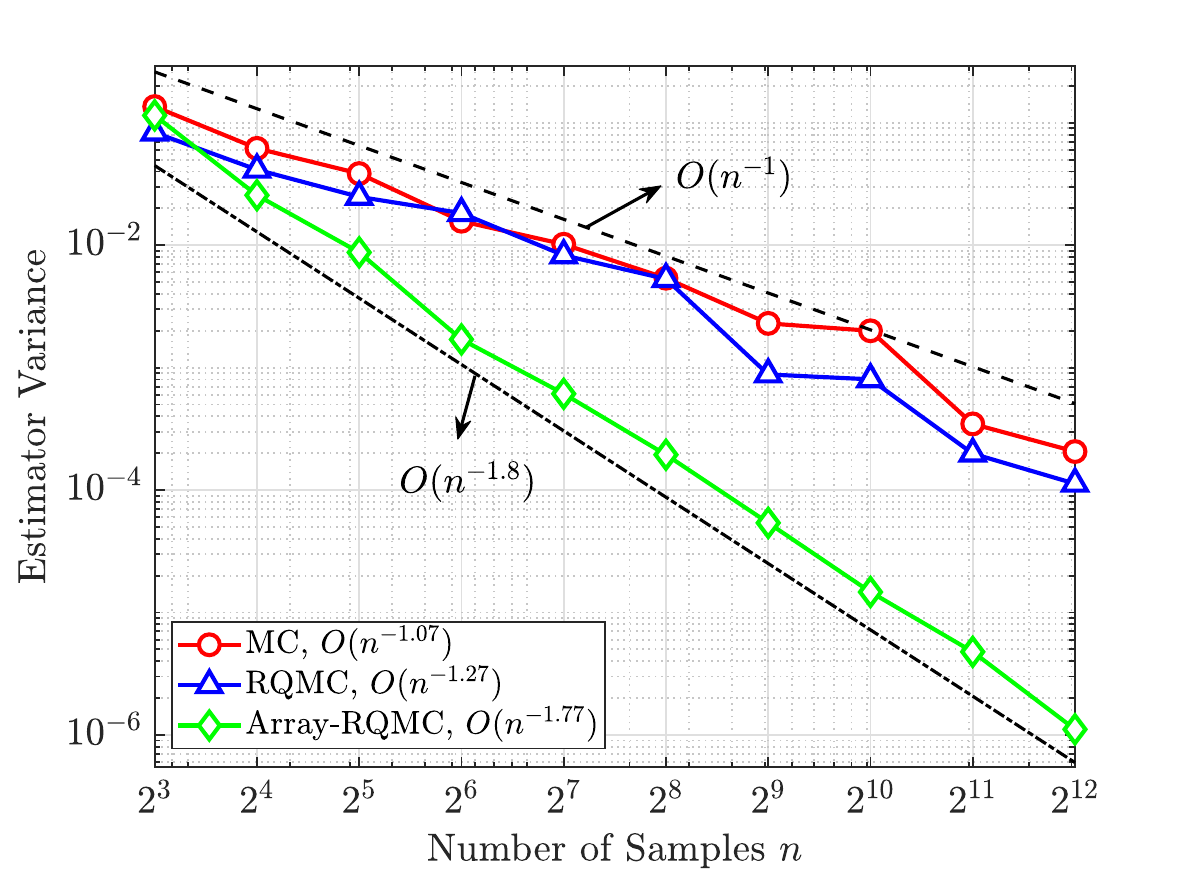}
        \caption{Empirical variance of the MC, direct RQMC, and Array-RQMC p-Ret estimators at the fixed initial point $x_0$.}
        \label{fig:star-torus-variance}
    \end{subfigure}
    \hfill
    \begin{subfigure}[t]{0.48\textwidth}
        \centering
        \includegraphics[width=\textwidth]{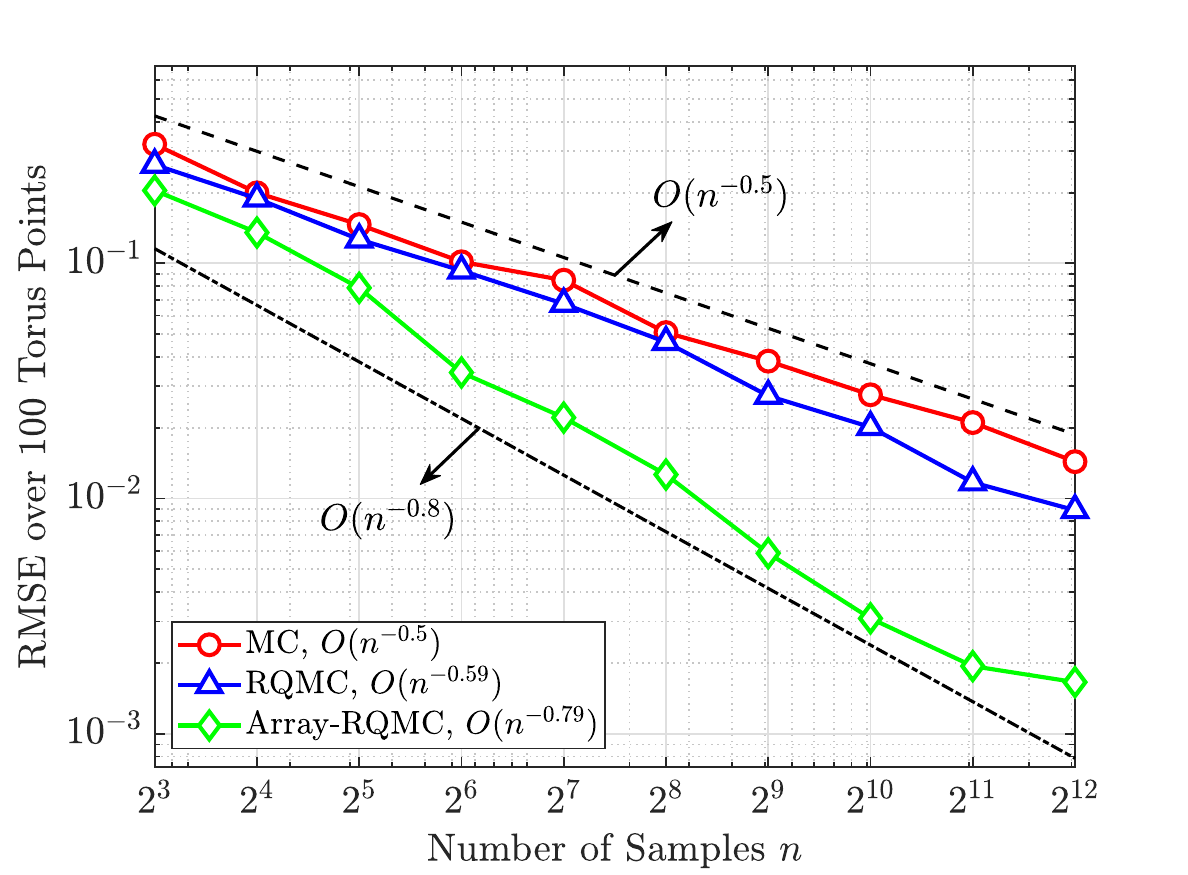}
        \caption{RMSE over $100$ fixed test points on the star-shaped torus.}
        \label{fig:star-torus-rmse}
    \end{subfigure}
    \caption{Sampling performance on the star-shaped torus parameterized by \eqref{eq:star-torus-param} for the MC, direct RQMC, and Array-RQMC p-Ret estimators.}
    \label{fig:star-torus-sampling}
\end{figure}



\subsection{A general diffusion operator on a wavy cylinder with boundary}
\label{subsec:wavy-cylinder}
We now examine the effect of a Dirichlet boundary and numerical first-exit treatment on the two-dimensional wavy cylinder shown in Figure~\ref{fig:wavy-cylinder-overview}(\subref{fig:wavy-cylinder-geometry}). This surface can be represented by the parameterization
\begin{equation} \label{eq:wavy-cylinder-param}
    \Phi(\theta,z)=\big(r(z)\cos\theta,r(z)\sin\theta,z\big),
\end{equation}
where $(\theta,z)\in U=[0,2\pi)\times(0,H)$, $\theta$ is periodic, and
\begin{equation*}
    r(z)=R_0+R_1\sin\Big(\frac{2\pi z}{H}\Big)+R_2\cos\Big(\frac{4\pi z}{H}\Big),
\end{equation*}
with $R_0=1.25$, $R_1=0.20$, $R_2=0.08$, and $H=1.60$. Its boundary consists of the two closed curves corresponding to $z=0$ and $z=H$.

In this example, the diffusion tensor $A$ and tangent drift $b$ in \eqref{eq:intrinsic-operator} are chosen as follows.
Let $y=z/H$ and let $\{\widehat e_\theta,\widehat e_z\}$ denote the induced orthonormal tangent frame. Define $\varphi(\theta,z)=0.28\sin(\theta+0.40\pi y)+0.12\cos(2\theta-\pi y)$. The diffusion vector fields are prescribed by
\begin{eqnarray*}
    &&\sigma_1=\big(0.95+0.18\sin\theta\cos(\pi y)\big)\Big(\cos\varphi\,\widehat e_\theta+\sin\varphi\,\widehat e_z\Big),\\
    &&\sigma_2=\big(0.72+0.14\cos(2\theta)+0.08\sin(\pi y)\big)\Big(-\sin\varphi\,\widehat e_\theta+\cos\varphi\,\widehat e_z\Big),
\end{eqnarray*}
which define the diffusion tensor in \eqref{eq:intrinsic-operator} by $A=\sigma_1\otimes\sigma_1+\sigma_2\otimes\sigma_2$. The tangent drift $b$ in \eqref{eq:intrinsic-operator} is given by
\begin{eqnarray*}
    &&b=\Big[0.14\sin(2\theta)\big(1+0.25\cos(\pi y)\big)+0.06\cos(\theta+\pi y)\Big]\widehat e_\theta\\
    &&\hspace{15pt}+\Big[0.10\cos\theta\sin(\pi y)-0.07\big(y-\tfrac12\big)+0.04\sin(2\theta)\sin(2\pi y)\Big]\widehat e_z .
\end{eqnarray*}

We take the exact solution
\begin{equation}
    \begin{split}
    u_{\rm ex}\big(\Phi(\theta,z),t\big)=\;&e^{-t}\Big[1+0.35\cos\theta\cos(\pi y)+0.22\sin(2\theta)\sin(2\pi y)\\
    &+0.18\cos(3\theta)\sin(\pi y)+0.12y(1-y)\sin\theta\Big],
    \label{eq:wavy-cylinder-exact}
    \end{split}
\end{equation}
and prescribe $u_0=u_{\rm ex}(\cdot,0)$, $\psi=u_{\rm ex}\big|_{\partial\M}$, and the source function $f=\partial_tu_{\rm ex}-\mathcal L u_{\rm ex}$.

\begin{figure}[ht!]
    \centering
    \begin{subfigure}[t]{0.48\textwidth}
        \centering
        \includegraphics[width=\linewidth,trim=8 45 8 8,clip]{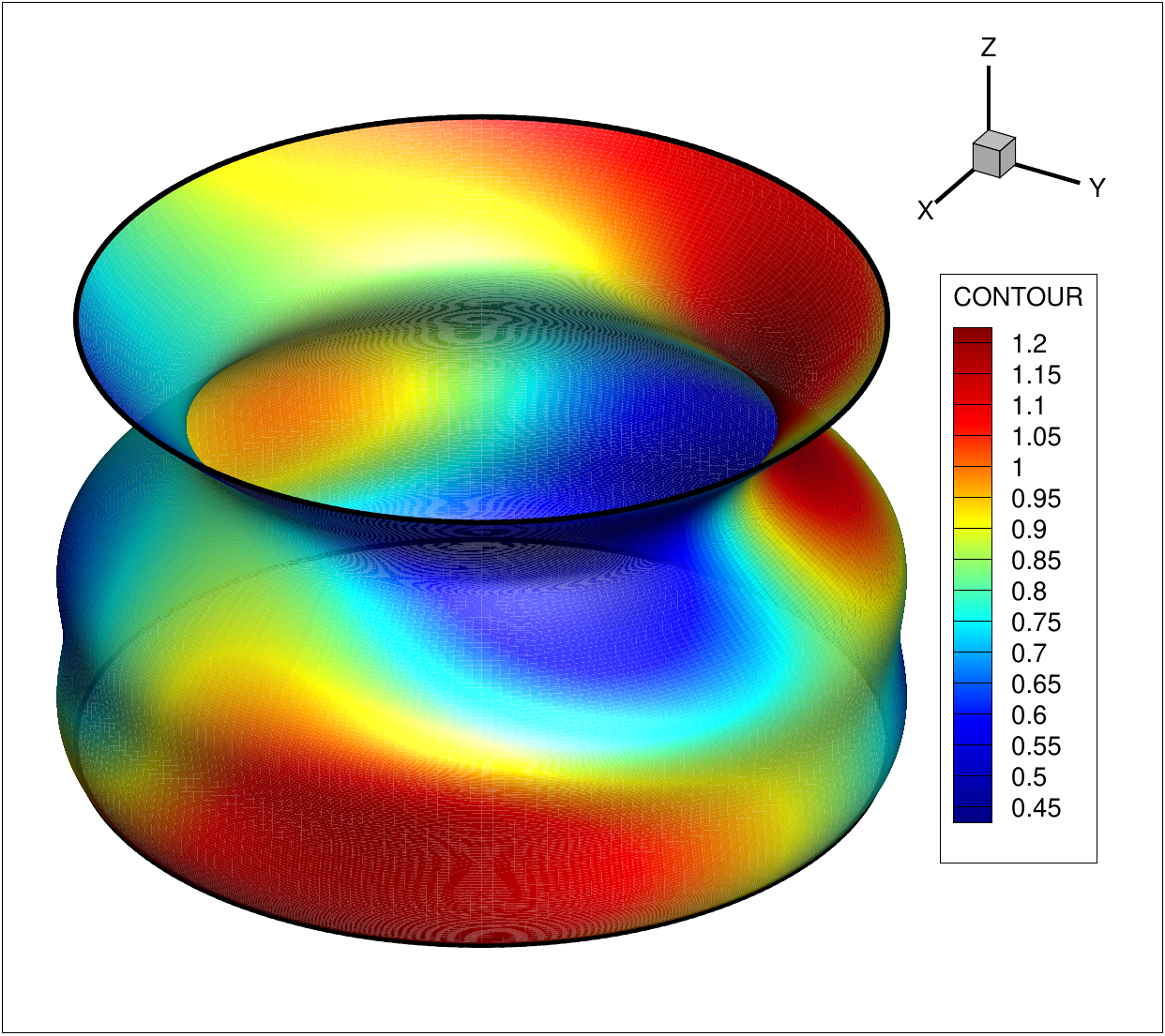}
        \caption{Wavy cylinder colored by the exact solution $u_{\rm ex}(\cdot,T)$ at $T=0.20$.}
        \label{fig:wavy-cylinder-geometry}
    \end{subfigure}
    \hfill
    \begin{subfigure}[t]{0.48\textwidth}
        \centering
        \includegraphics[width=\linewidth]{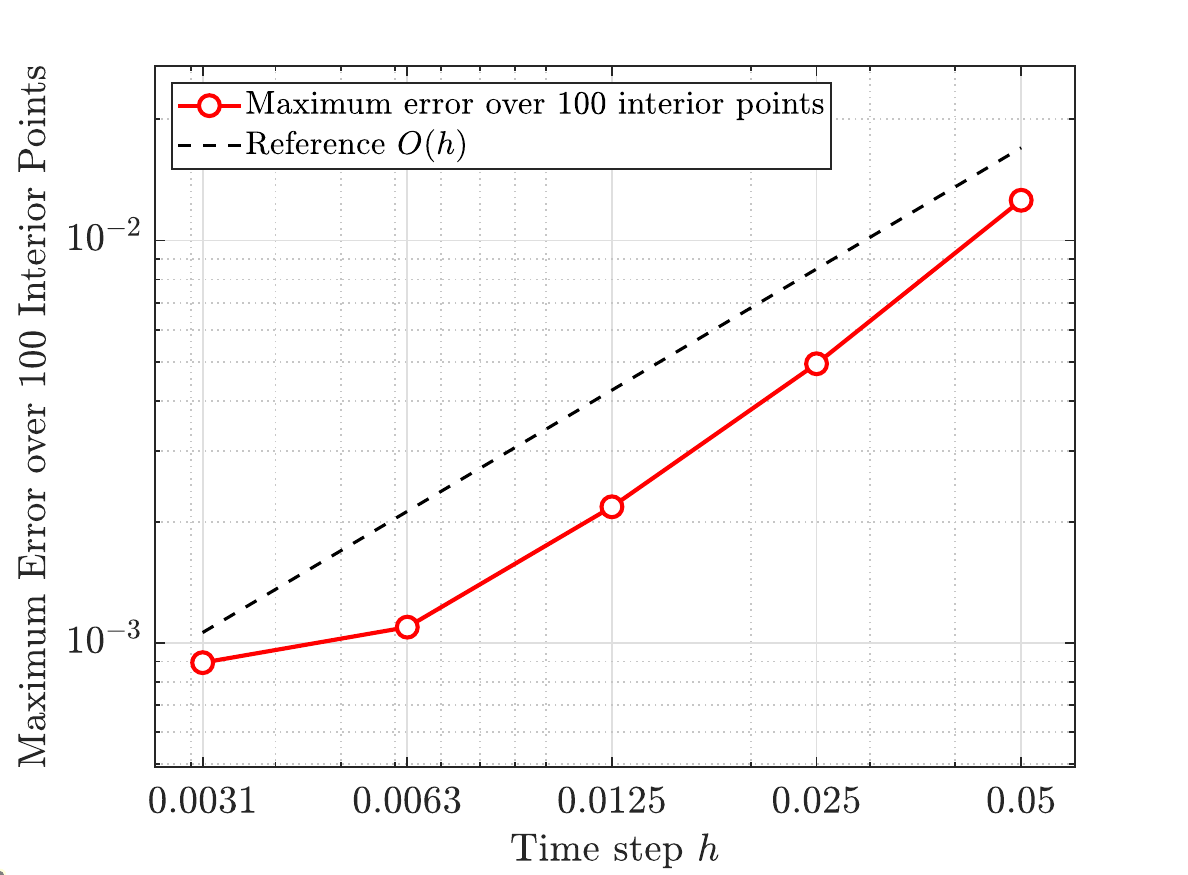}
        \caption{Temporal convergence measured by the maximum error over $100$ fixed interior points.}
        \label{fig:wavy-cylinder-temporal}
    \end{subfigure}
    \caption{Wavy-cylinder parameterized by \eqref{eq:wavy-cylinder-param} and temporal convergence of the stopped p-Ret discretization.}
    \label{fig:wavy-cylinder-overview}
\end{figure}

To examine the temporal error, the test points are restricted to the compact interior region $0.35H\le z\le0.65H$. Figure~\ref{fig:wavy-cylinder-overview}(\subref{fig:wavy-cylinder-temporal}) shows an approximately first-order decay, consistent with Remark~\ref{rem:dirichlet-temporal-error}. The slightly less regular convergence than in the closed case reflects the additional numerical boundary-exit contribution $\mathcal E_{\partial}$.


\begin{figure}[ht!]
    \centering
    \begin{subfigure}[t]{0.48\textwidth}
        \centering
        \includegraphics[width=\textwidth]{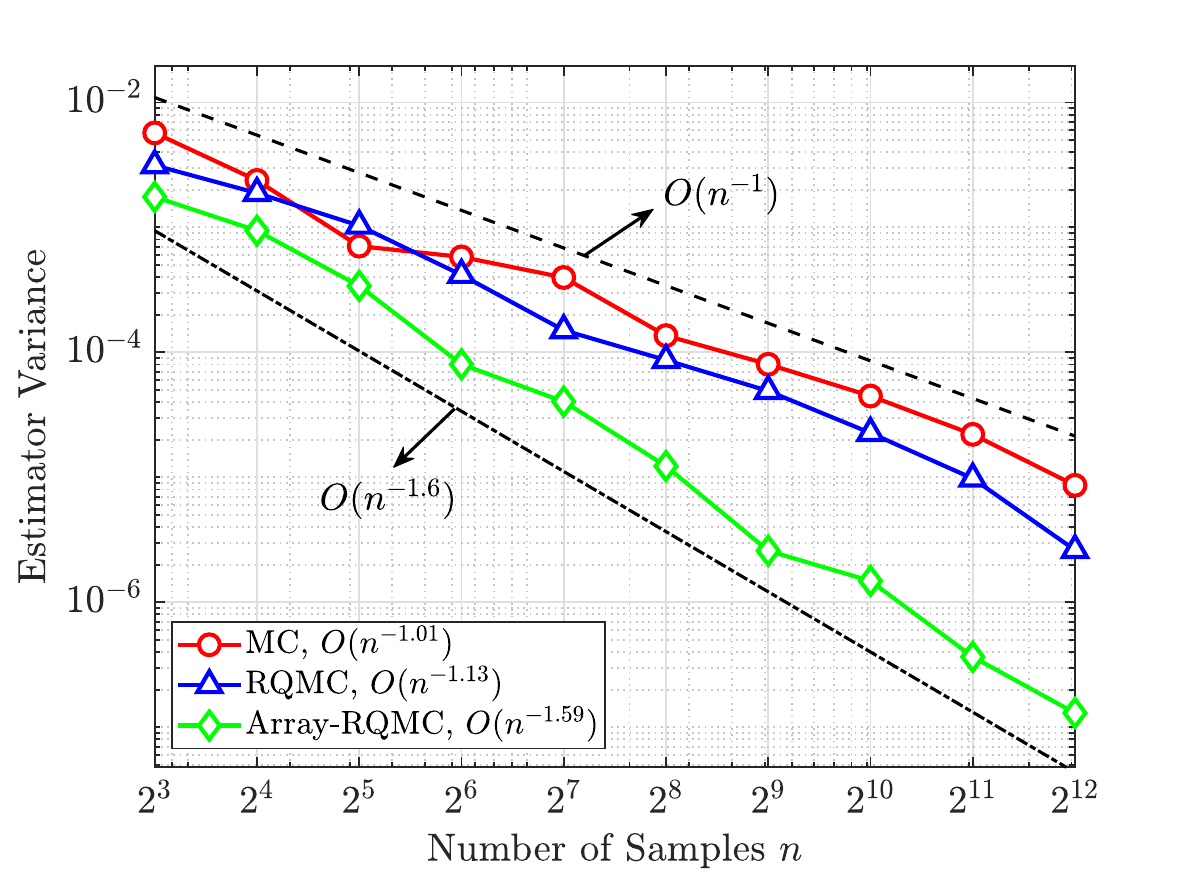}
        \caption{Empirical variance of the three p-Ret estimators at the fixed initial point $x_0$.}
        \label{fig:wavy-cylinder-variance}
    \end{subfigure}
    \hfill
    \begin{subfigure}[t]{0.48\textwidth}
        \centering
        \includegraphics[width=\textwidth]{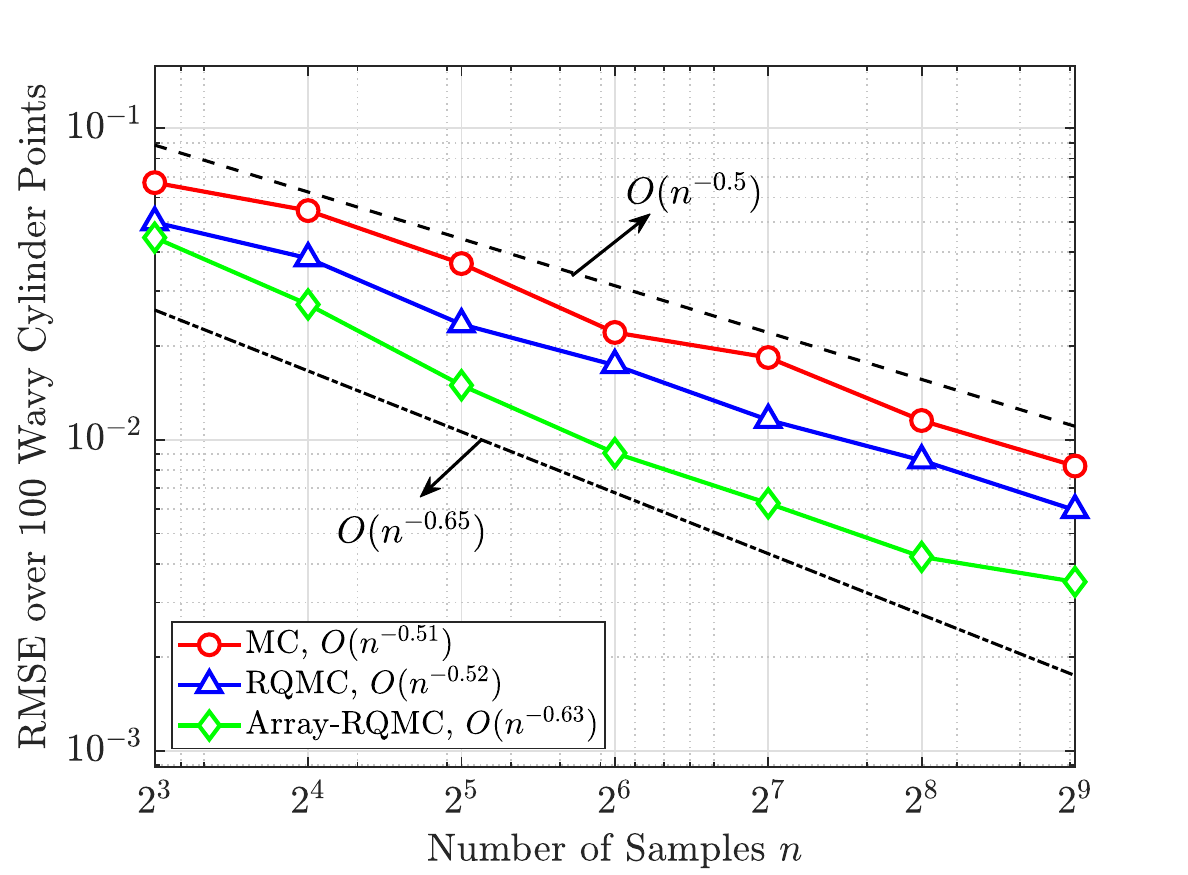}
        \caption{RMSE over $100$ fixed test points on the wavy cylinder.}
        \label{fig:wavy-cylinder-rmse}
    \end{subfigure}
    \caption{Sampling performance on the wavy cylinder parameterized by \eqref{eq:wavy-cylinder-param} for the MC, direct RQMC, and Array-RQMC stopped p-Ret estimators.}
    \label{fig:wavy-cylinder-sampling}
\end{figure}


We next compare the sampling variances at $x_0=\Phi(1.10,0.82H)$ using $n=2^j$, $j=3,\ldots,12$. As shown in Figure~\ref{fig:wavy-cylinder-sampling}(\subref{fig:wavy-cylinder-variance}), the fitted rates are approximately $O(n^{-1.01})$, $O(n^{-1.13})$, and $O(n^{-1.59})$ for MC, direct RQMC, and Array-RQMC, respectively. The RQMC improvement is weaker than in the closed-manifold experiment, consistently with the first-exit discontinuities underlying the estimate in Lemma~\ref{lem:boundary-direct-rqmc-variance}.


For the spatial test, the fixed test points cover the full interval $0<z<H$, and we take $n=2^j$, $j=3,\ldots,9$. Figure~\ref{fig:wavy-cylinder-sampling}(\subref{fig:wavy-cylinder-rmse}) shows the same qualitative behavior: direct RQMC generally improves upon MC, while Array-RQMC yields the smallest RMSE over the tested range, with an empirical rate of approximately $O(n^{-0.63})$.



\subsection{A Laplace--Beltrami problem on high-dimensional products of Lawson Klein bottles}
\label{subsec:lawson-klein-product}

We finally consider a high-dimensional closed-manifold problem on Cartesian products of Lawson Klein bottles to investigate the effect of increasing dimension. Let $\mathcal K$ denote the Lawson Klein bottle $\tau_{2,1}$ (cf.~\cite{Lawson1970}). It is a two-dimensional manifold minimally immersed in $S^3\subset\mathbb R^4$ and is represented by
\begin{equation}
    \Phi_{\mathcal K}(\theta,\omega)=\big(\cos(2\theta)\cos\omega,\sin(2\theta)\cos\omega,\cos\theta\sin\omega,\sin\theta\sin\omega\big)\in S^3\subset\mathbb R^4,
    \label{eq:lawson-klein-param}
\end{equation}
where $\theta\in[0,\pi)$ and $\omega\in[0,2\pi)$, with the identifications
\begin{equation}
    (\theta+\pi,\omega)\sim(\theta,-\omega),
    \quad
    (\theta,\omega+2\pi)\sim(\theta,\omega).
    \label{eq:lawson-klein-identification}
\end{equation}

For an even intrinsic dimension $m$, let $p=m/2$ and define
\begin{equation}
    \mathcal M_m=\underbrace{\mathcal K\times\cdots\times\mathcal K}_{p\ {\rm factors}},
    \quad
    \xi=(\theta_1,\omega_1,\ldots,\theta_p,\omega_p).
    \label{eq:lawson-product-manifold}
\end{equation}
The corresponding product parameterization is
\begin{equation}
    \Phi_m(\xi)=
    \begin{pmatrix}
        \Phi_{\mathcal K}(\theta_1,\omega_1)\\
        \Phi_{\mathcal K}(\theta_2,\omega_2)\\
        \vdots\\
        \Phi_{\mathcal K}(\theta_p,\omega_p)
    \end{pmatrix}
    \in\mathbb R^{4p}=\mathbb R^{2m}.
    \label{product para}
\end{equation}
Thus, $\mathcal M_m$ has intrinsic dimension $m$.

In this experiment, we take $A=I_{T\mathcal M_m}$ and $b=0$ in \eqref{eq:intrinsic-operator} so that $\mathcal L=\frac12\Delta_{\mathcal M_m}$. For the exact solution, define
\begin{equation*}
    \eta_j=0.45\cos(2\theta_j)+0.30\cos\omega_j+0.25\sin\omega_j\cos\theta_j,
    \quad
    j=1,\ldots,p,
\end{equation*}
which is compatible with the identifications in \eqref{eq:lawson-klein-identification}. With $\alpha=0.30$ and the cyclic convention $\eta_{p+1}=\eta_1$, set
\begin{equation}
    S_m(\xi)=\frac1{\sqrt p}\sum_{j=1}^{p}\eta_j+\frac{\alpha}{p}\sum_{j=1}^{p}\eta_j\eta_{j+1},
    \quad
    Q_m(\xi)=1+\frac12\sin S_m(\xi).
    \label{eq:lawson-coupled-q}
\end{equation}
The second term in $S_m$ couples neighboring Lawson factors. We take
\begin{equation}
    u_{\rm ex}\big(\Phi_m(\xi),t\big)
    =
    \cos(t)Q_m(\xi),
    \label{eq:lawson-exact}
\end{equation}
with $u_0=u_{\rm ex}(\cdot,0)$ and $f=\partial_tu_{\rm ex}-\frac12\Delta_{\mathcal M_m}u_{\rm ex}$.

For the present problem, we take $r=m$, so that each p-Ret step uses $r-1=m-1$ uniform variables. Accordingly, direct RQMC operates in dimension $s=(m-1)K$, whereas Array-RQMC uses an $m$-dimensional randomized point set at each time layer.
Direct RQMC uses a single randomized net in dimension $s$, while Array-RQMC generates an independently randomized $m$-dimensional net at each time layer. For the Lawson manifold, the Hilbert ordering is made compatible with the identification \eqref{eq:lawson-klein-identification} by replacing a parameter representative, when necessary, with its equivalent representative obtained through $\omega\mapsto-\omega$ modulo $2\pi$ before normalization and ordering. This preprocessing changes only the representative used for sorting and leaves the p-Ret dynamics and the discrete Feynman--Kac functional unchanged.


\begin{wrapfigure}{r}{0.50\textwidth}
\vspace{-12pt}
\centering
\includegraphics[width=0.47\textwidth]{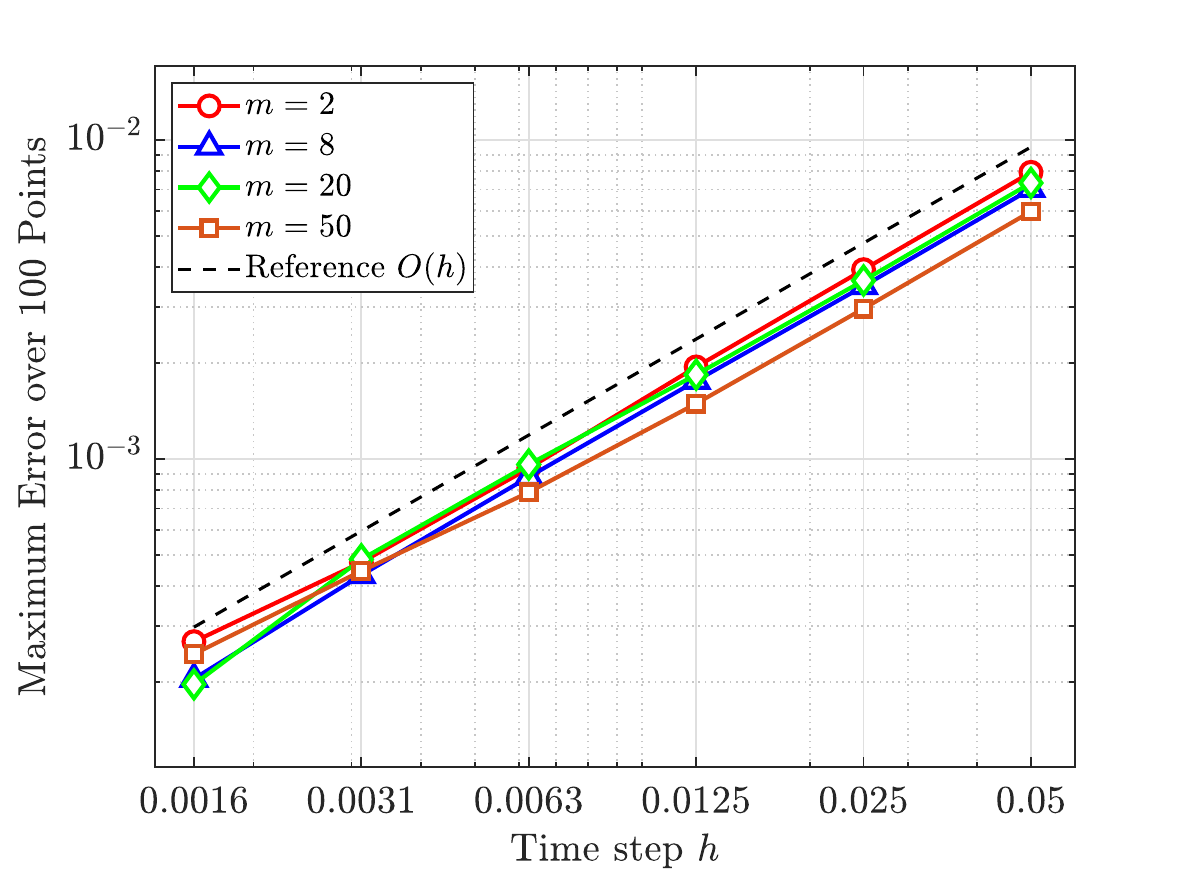}
\vspace{2pt}
\caption{Temporal convergence on products of Lawson Klein bottles parameterized by \eqref{product para} for $m=2,8,20,$ and $50$.}
\label{fig:lawson-temporal}
\vspace{-4pt}
\end{wrapfigure}


We first examine the temporal convergence for the representative intrinsic dimensions $m=2,8,20,$ and $50$, using fixed test points obtained from coordinate prefixes of a common $50$-dimensional point set. For each dimension, the temporal error is measured by the discrete maximum error over the prescribed test set while the time step is successively refined. Figure~\ref{fig:lawson-temporal} shows a clear first-order decay across all four dimensions. The fitted convergence orders remain close to the reference rate $O(h)$ predicted by Lemma~\ref{lem:weak-temporal-convergence}, with no evident loss of temporal order as the intrinsic dimension increases. These results therefore support the first-order temporal accuracy of the p-Ret discretization from a single Lawson Klein bottle up to the $50$-dimensional product manifold.

We next investigate the dependence of the sampling variance on the intrinsic dimension. For $m=2p$, the fixed initial coordinate is chosen as
\begin{equation*}
    \xi_0^{(m)}=(\theta_{0,1},\omega_{0,1},\ldots,\theta_{0,p},\omega_{0,p}),
\end{equation*}
where
\begin{eqnarray*}
    &&\theta_{0,j}=\operatorname{mod}\big(0.42+0.37(j-1),\pi\big),\\
    &&\omega_{0,j}=\operatorname{mod}\big(1.10+0.51(j-1)+0.17\sin j,2\pi\big),
    \quad
    j=1,\ldots,p.
\end{eqnarray*}
We take $m=2,4,8,16,20,30,40,50$ and $n=2^j$, $j=3,\ldots,12$. Figure~\ref{fig:lawson-variance} shows how the variance reduction changes with the intrinsic dimension. For MC, the fitted rates remain close to $O(n^{-1})$ for all tested dimensions, showing little sensitivity of the convergence order to $m$. Direct RQMC gives a clear improvement in low dimensions, with the fitted rate decreasing from approximately $O(n^{-1.30})$ for $m=2$ to $O(n^{-1.10})$ for $m=20$. As $m$ increases further, the rate approaches the MC rate. This deterioration is consistent with the rapid growth of the nominal integration dimension $s=(m-1)K$; with $K=64$, $s$ increases from $64$ for $m=2$ to $3136$ for $m=50$, making the full path-space RQMC construction progressively less effective.

\begin{figure}[htbp]
    \centering
    \begin{subfigure}[t]{0.32\textwidth}
        \centering
        \includegraphics[width=\textwidth]{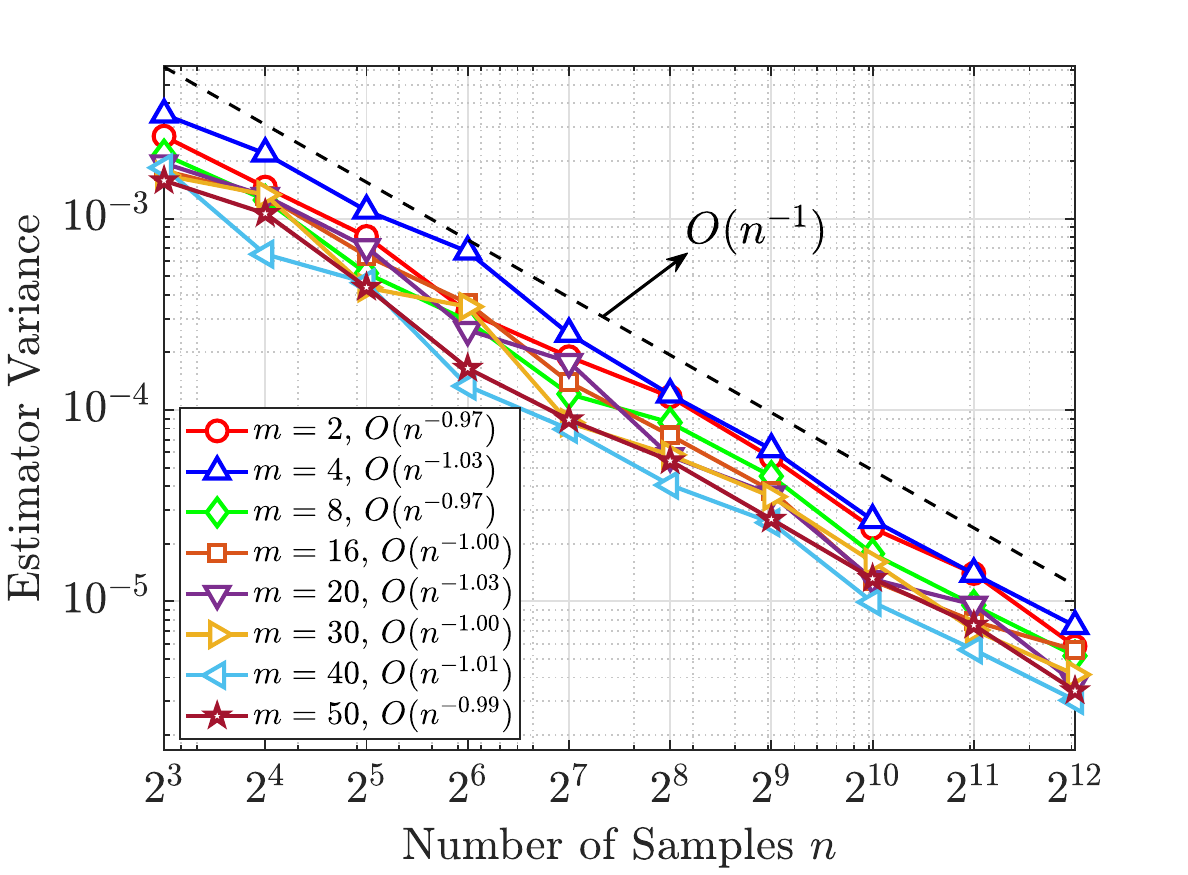}
        \caption{MC variance.}
        \label{fig:lawson-variance-mc}
    \end{subfigure}
    \hfill
    \begin{subfigure}[t]{0.32\textwidth}
        \centering
        \includegraphics[width=\textwidth]{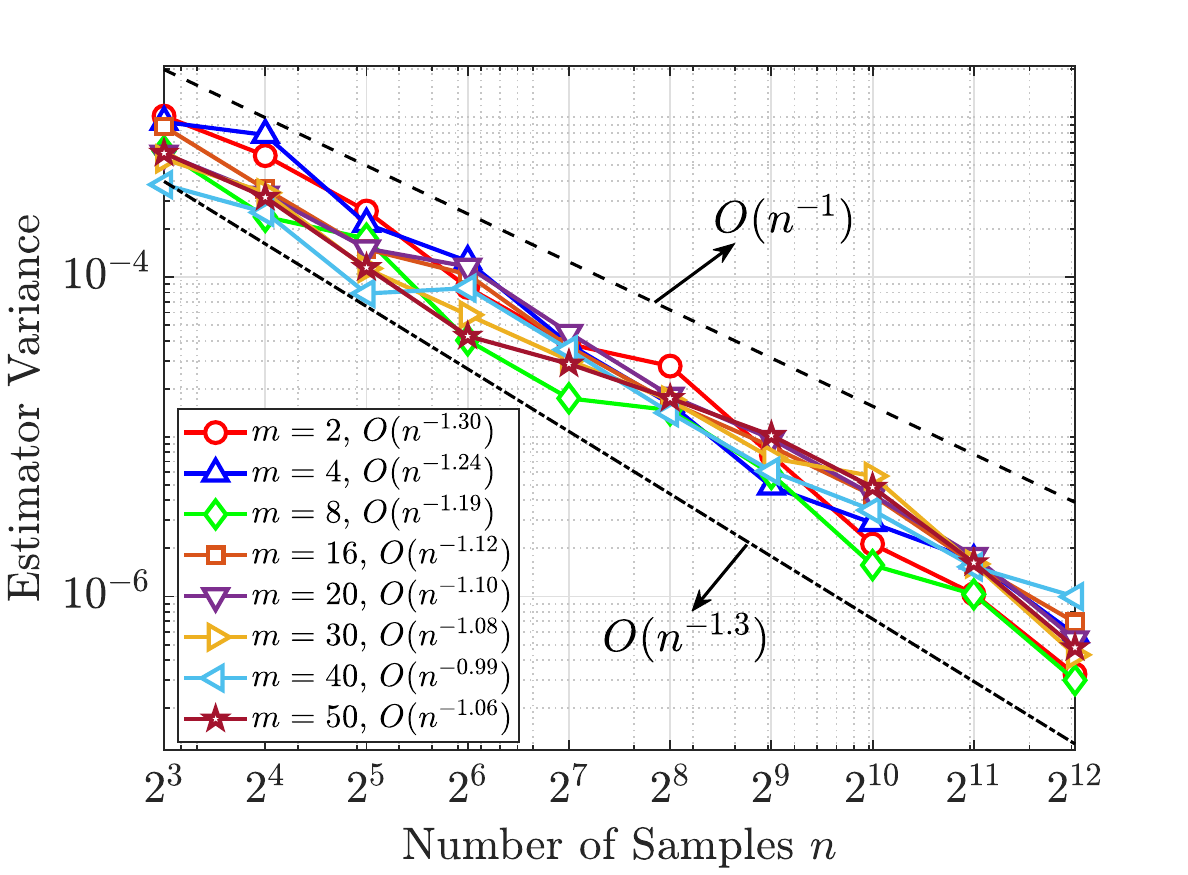}
        \caption{Direct RQMC variance.}
        \label{fig:lawson-variance-rqmc}
    \end{subfigure}
    \hfill
    \begin{subfigure}[t]{0.32\textwidth}
        \centering
        \includegraphics[width=\textwidth]{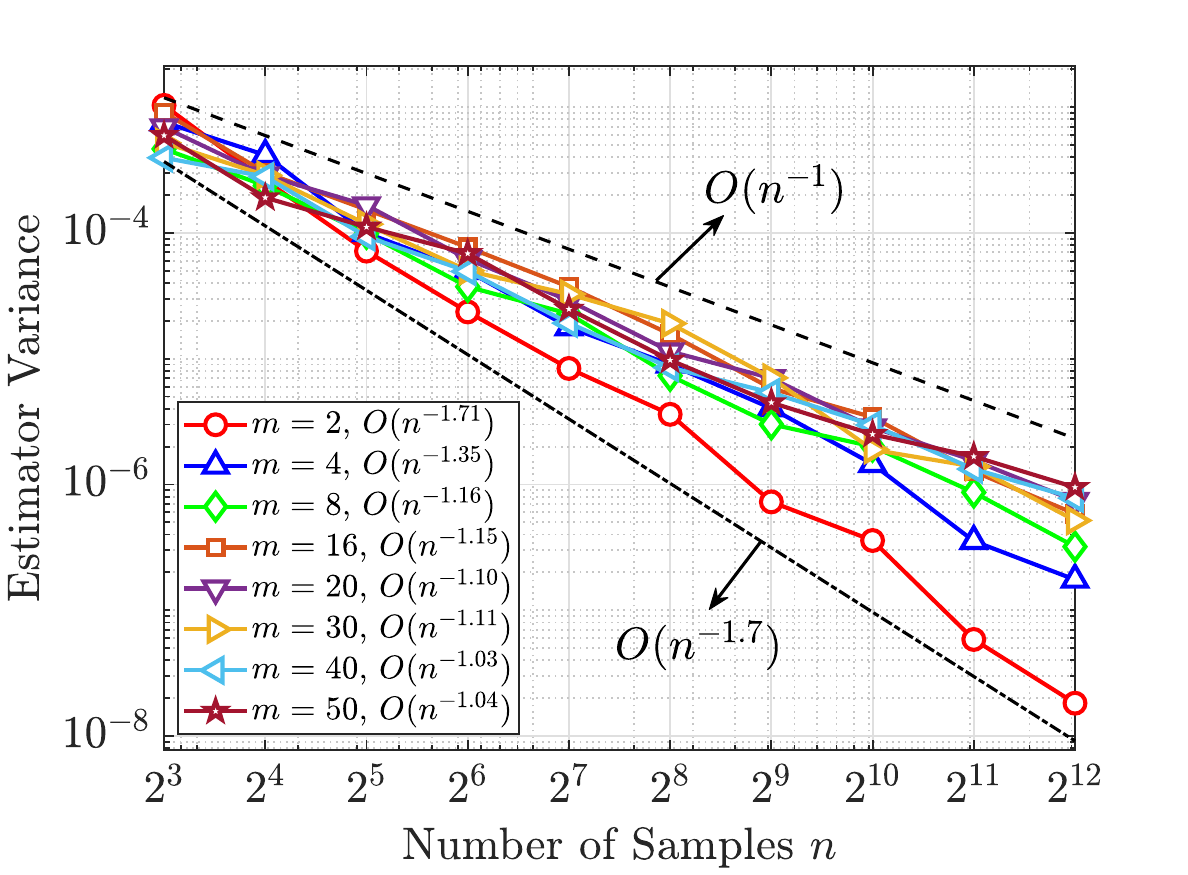}
        \caption{Array-RQMC variance.}
        \label{fig:lawson-variance-arqmc}
    \end{subfigure}
    \caption{Empirical variance of the MC, direct RQMC, and Array-RQMC p-Ret estimators on products of Lawson Klein bottles parameterized by \eqref{product para} for different intrinsic dimensions.}
    \label{fig:lawson-variance}
\end{figure}

Array-RQMC gives the strongest variance reduction in the low-dimensional cases, with a fitted rate of approximately $O(n^{-1.71})$ for $m=2$ and $O(n^{-1.35})$ for $m=4$. Its empirical rate also decreases as $m$ increases, approaching $O(n^{-1})$ in the highest-dimensional cases. Nevertheless, unlike direct RQMC, Array-RQMC does not operate on the full $s=(m-1)K$ dimensional path space, but only on an $m$-dimensional randomized point set at each time layer. The results therefore show that both RQMC constructions become less effective as the intrinsic dimension grows, while Array-RQMC provides the largest variance reduction over the practically relevant low- and moderate-dimensional range and avoids the much faster growth of the direct RQMC integration dimension.

Finally, we examine the spatial RMSE over different initial positions for $m=2,4,8,16,20,30,40,50$ and $n=2^j$, $j=3,\ldots,12$. At each test point, the numerical approximation is compared with \eqref{eq:lawson-exact}, and the spatial error is evaluated using \eqref{eq:numerical-rmse}. Figure~\ref{fig:lawson-rmse} shows a trend consistent with Figure~\ref{fig:lawson-variance}. MC remains close to the classical $O(n^{-1/2})$ rate, while the improvements of direct RQMC and Array-RQMC gradually weaken as $m$ increases, with Array-RQMC generally giving the smallest RMSE. The slight flattening at large $n$ in some cases is consistent with the fixed temporal discretization error becoming comparable to the sampling error.

\begin{figure}[htbp]
    \centering
    \begin{subfigure}[t]{0.32\textwidth}
        \centering
        \includegraphics[width=\textwidth]{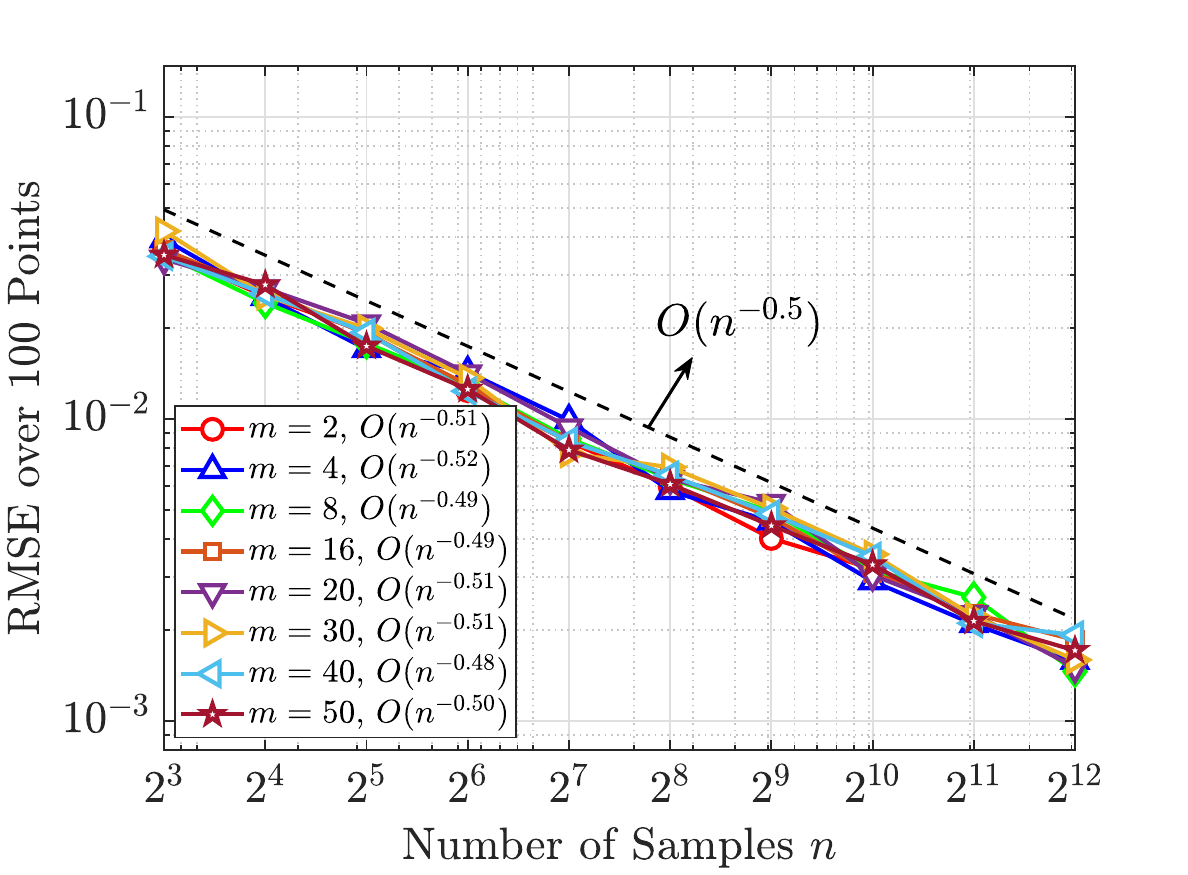}
        \caption{MC RMSE.}
        \label{fig:lawson-rmse-mc}
    \end{subfigure}
    \hfill
    \begin{subfigure}[t]{0.32\textwidth}
        \centering
        \includegraphics[width=\textwidth]{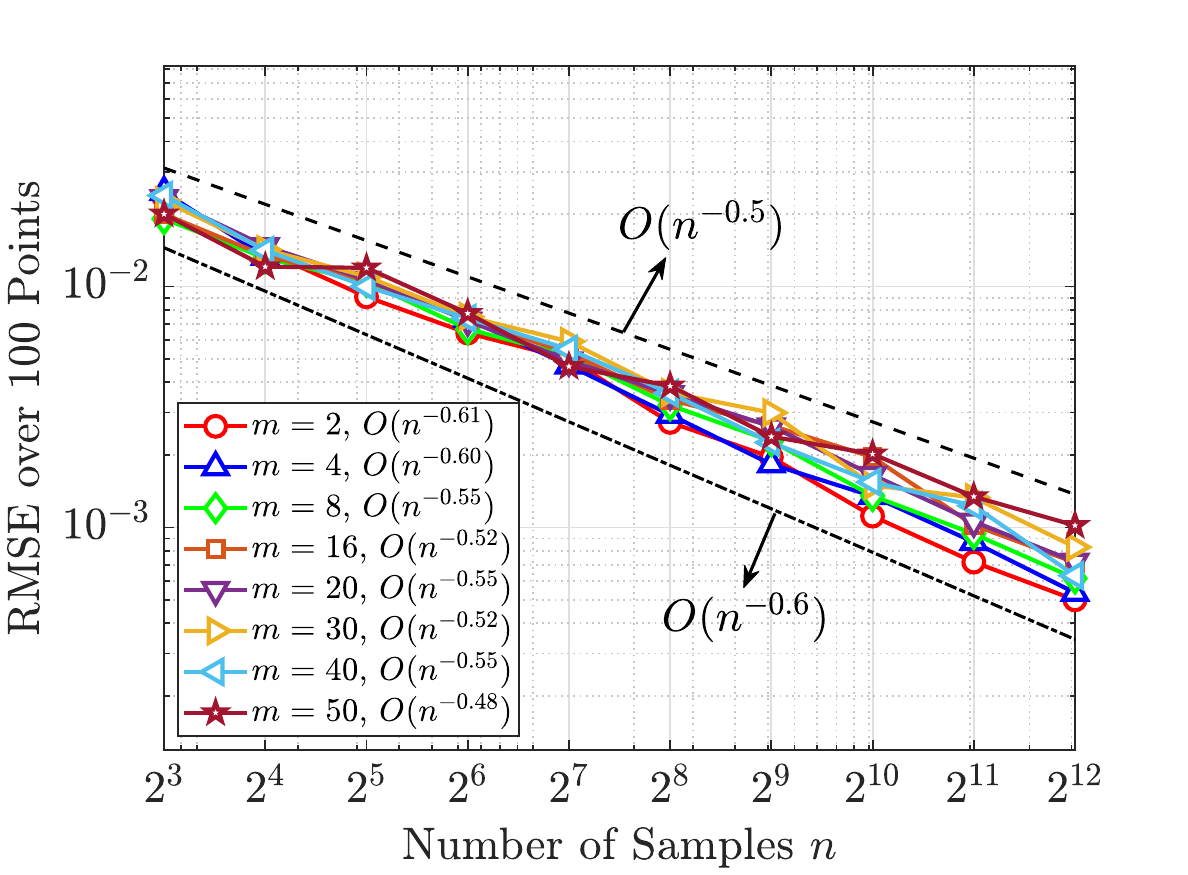}
        \caption{Direct RQMC RMSE.}
        \label{fig:lawson-rmse-rqmc}
    \end{subfigure}
    \hfill
    \begin{subfigure}[t]{0.32\textwidth}
        \centering
        \includegraphics[width=\textwidth]{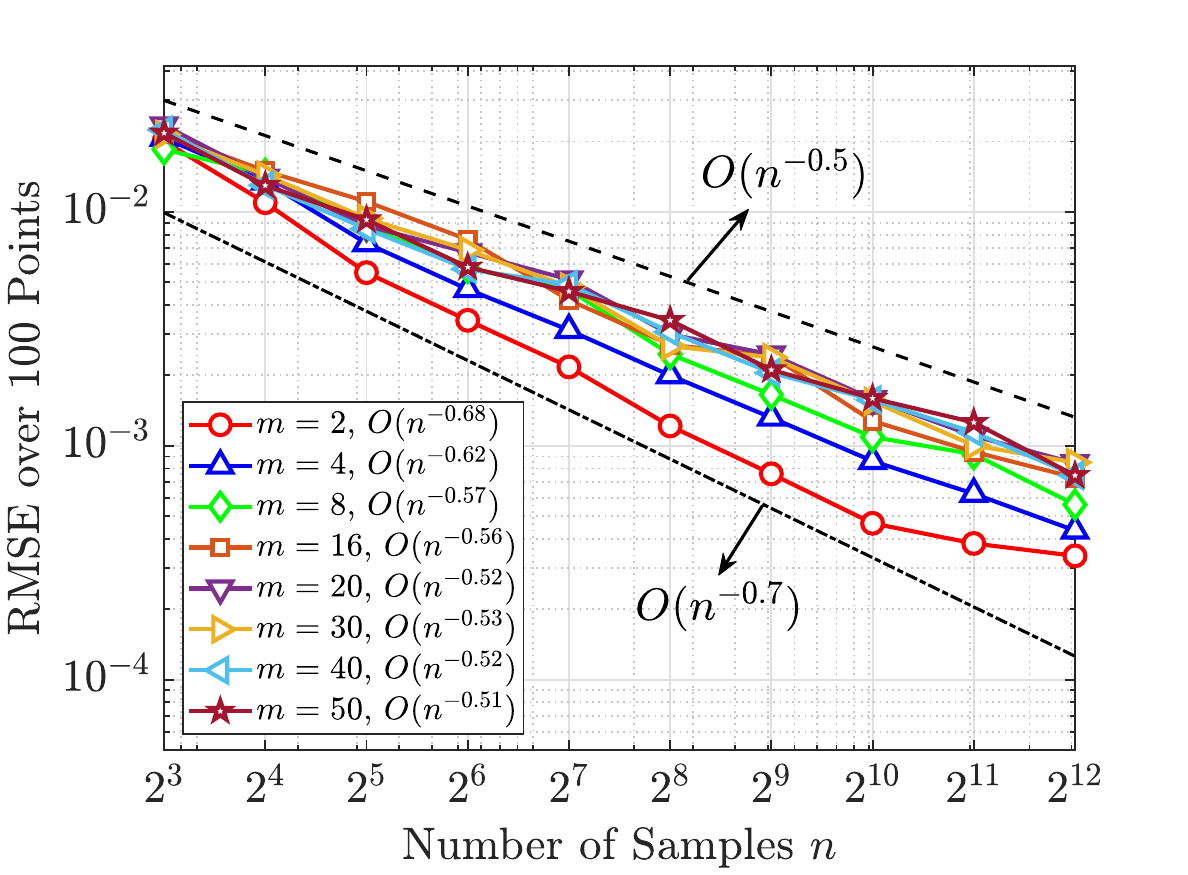}
        \caption{Array-RQMC RMSE.}
        \label{fig:lawson-rmse-arqmc}
    \end{subfigure}
    \caption{RMSE over $100$ fixed test points on products of Lawson Klein bottles parameterized by \eqref{product para} for different intrinsic dimensions.}
    \label{fig:lawson-rmse}
\end{figure}

Overall, the Lawson-product experiment confirms the first-order temporal accuracy of p-Ret and illustrates the effect of increasing intrinsic dimension on the sampling performance. The advantages of direct RQMC and Array-RQMC gradually weaken as the dimension increases, while Array-RQMC generally provides the strongest variance and RMSE reduction over the tested range.

\section{Conclusion}
\label{sec:conclusion}

We have developed a probabilistic framework for the pointwise solution of
linear parabolic equations on parameterized Riemannian manifolds. The main
feature of the framework is that the manifold diffusion is propagated through
local p-Ret dynamics, while the resulting discrete Feynman--Kac functional can
be sampled by MC, direct RQMC, or Array-RQMC without changing the underlying
time discretization. This separation of geometric approximation from sampling
provides a meshfree and naturally parallel approach that remains applicable to
manifolds of relatively high intrinsic dimension.

The analysis also reveals a marked difference between closed and Dirichlet
manifolds. On closed manifolds, the discrete Feynman--Kac functional retains
sufficient regularity for direct RQMC to improve on the standard Monte Carlo
variance rate, with nearly quadratic variance decay, up to logarithmic factors,
under stronger regularity assumptions. Boundary stopping changes this picture
substantially. In the Dirichlet case, numerical first exit creates
codimension-one discontinuity interfaces in path space. A layerwise
decomposition, together with Owen's boundary growth condition and geometric
control of the switching interfaces, nevertheless recovers a variance rate
strictly better than the Monte Carlo benchmark. Array-RQMC provides a different
resolution of the high-dimensional sampling problem: it replaces full
path-space randomization by sequential low-dimensional randomizations and
retains variance $o(n^{-1})$ for fixed $K$ under the stated stability
assumptions.

The numerical experiments support these theoretical findings and clarify where
the different sampling strategies are most effective. Both RQMC approaches
provide clear variance and RMSE reductions in low and moderate intrinsic
dimensions, while their advantage decreases as the dimension grows. The
experiments also highlight the different sources of dimensional growth:
direct RQMC is affected by the full path-space dimension $(r-1)K$, whereas
Array-RQMC works with stepwise randomized point sets whose dimension does not
grow with the number of time steps. Further improvements may come from sharper
control of the numerical first-exit error, verifiable stability conditions for
Array-RQMC, and extensions of the present construction beyond globally
parameterized manifolds.

\end{document}